\documentclass[11pt]{amsart}

\usepackage[all]{xy}
\usepackage{color, pstricks}
\usepackage{mathrsfs}
\usepackage{amsmath,amsthm, amssymb, amsgen,amsxtra,amsfonts, amsbsy} 
\usepackage[all]{xy}
\usepackage{verbatim}
\usepackage{url}
\usepackage{hyperref}

\DeclareMathOperator{\Ker}{Ker}
\DeclareMathOperator{\Coker}{Coker}

\DeclareMathOperator{\length}{length}
\DeclareMathOperator{\depth}{depth}
\DeclareMathOperator{\diag}{diag}
\DeclareMathOperator{\Der}{Der}

\DeclareMathOperator{\rank}{rank}

\begin{document}
%
%
%
\theoremstyle{definition}
\newtheorem{Definition}{Definition}[section]
\newtheorem*{Definitionx}{Definition}
\newtheorem{Convention}[Definition]{Convention}
\newtheorem{Construction}{Construction}[section]
\newtheorem{Example}[Definition]{Example}
\newtheorem{Examples}[Definition]{Examples}
\newtheorem{Remark}[Definition]{Remark}
\newtheorem*{Remarkx}{Remark}
\newtheorem{Remarks}[Definition]{Remarks}
\newtheorem{Caution}[Definition]{Caution}
\newtheorem{Conjecture}[Definition]{Conjecture}
\newtheorem*{Conjecturex}{Conjecture}
\newtheorem{Question}[Definition]{Question}
\newtheorem{Questions}[Definition]{Questions}
\newtheorem*{Acknowledgements}{Acknowledgements}
\newtheorem*{Organization}{Organization}
\newtheorem*{Disclaimer}{Disclaimer}
\theoremstyle{plain}
\newtheorem{Theorem}[Definition]{Theorem}
\newtheorem*{Theoremx}{Theorem}
\newtheorem{Theoremy}{Theorem}
\newtheorem{Proposition}[Definition]{Proposition}
\newtheorem*{Propositionx}{Proposition}
\newtheorem{Lemma}[Definition]{Lemma}
\newtheorem{Corollary}[Definition]{Corollary}
\newtheorem*{Corollaryx}{Corollary}
\newtheorem{Fact}[Definition]{Fact}
\newtheorem{Facts}[Definition]{Facts}
\newtheoremstyle{voiditstyle}{3pt}{3pt}{\itshape}{\parindent}%
{\bfseries}{.}{ }{\thmnote{#3}}%
\theoremstyle{voiditstyle}
\newtheorem*{VoidItalic}{}
\newtheoremstyle{voidromstyle}{3pt}{3pt}{\rm}{\parindent}%
{\bfseries}{.}{ }{\thmnote{#3}}%
\theoremstyle{voidromstyle}
\newtheorem*{VoidRoman}{}

%
\newcommand{\prf}{\par\noindent{\sc Proof.}\quad}
\newcommand{\blowup}{\rule[-3mm]{0mm}{0mm}}
\newcommand{\cal}{\mathcal}
\newcommand{\Aff}{{\mathds{A}}}
\newcommand{\BB}{{\mathds{B}}}
\newcommand{\CC}{{\mathds{C}}}
\newcommand{\EE}{{\mathds{E}}}
\newcommand{\FF}{{\mathds{F}}}
\newcommand{\GG}{{\mathds{G}}}
\newcommand{\HH}{{\mathds{H}}}
\newcommand{\NN}{{\mathds{N}}}
\newcommand{\ZZ}{{\mathds{Z}}}
\newcommand{\PP}{{\mathds{P}}}
\newcommand{\QQ}{{\mathds{Q}}}
\newcommand{\RR}{{\mathds{R}}}
\newcommand{\Sph}{{\mathds{S}}}
\newcommand{\Liea}{{\mathfrak a}}
\newcommand{\Lieb}{{\mathfrak b}}
\newcommand{\Lieg}{{\mathfrak g}}
\newcommand{\Liem}{{\mathfrak m}}
\newcommand{\ideala}{{\mathfrak a}}
\newcommand{\idealb}{{\mathfrak b}}
\newcommand{\idealg}{{\mathfrak g}}
\newcommand{\idealm}{{\mathfrak m}}
\newcommand{\idealn}{{\mathfrak n}}
\newcommand{\idealp}{{\mathfrak p}}
\newcommand{\idealq}{{\mathfrak q}}
\newcommand{\idealI}{{\cal I}}
\newcommand{\lin}{\sim}
\newcommand{\num}{\equiv}
\newcommand{\dual}{\ast}
\newcommand{\iso}{\cong}
\newcommand{\homeo}{\approx}
\newcommand{\mathds}[1]{{\mathbb #1}}
\newcommand{\mm}{{\mathfrak m}}
\newcommand{\pp}{{\mathfrak p}}
\newcommand{\qq}{{\mathfrak q}}
\newcommand{\rr}{{\mathfrak r}}
\newcommand{\pP}{{\mathfrak P}}
\newcommand{\qQ}{{\mathfrak Q}}
\newcommand{\rR}{{\mathfrak R}}
%
%
\newcommand{\OO}{{\cal O}}
\newcommand{\calA}{{\cal A}}
\newcommand{\calO}{{\cal O}}
\newcommand{\calU}{{\cal U}}
\newcommand{\numero}{{n$^{\rm o}\:$}}
\newcommand{\mf}[1]{\mathfrak{#1}}
\newcommand{\mc}[1]{\mathcal{#1}}
\newcommand{\into}{{\hookrightarrow}}
\newcommand{\onto}{{\twoheadrightarrow}}
\newcommand{\Spec}{{\rm Spec}\:}
\newcommand{\BigSpec}{{\rm\bf Spec}\:}
\newcommand{\Spf}{{\rm Spf}\:}
\newcommand{\Proj}{{\rm Proj}\:}
\newcommand{\Pic}{{\rm Pic }}
\newcommand{\Picloc}{{\rm Picloc }}
\newcommand{\Br}{{\rm Br}}
\newcommand{\NS}{{\rm NS}}
\newcommand{\id}{{\rm id}}
\newcommand{\Sym}{{\mathfrak S}}
\newcommand{\Aut}{{\rm Aut}}
\newcommand{\Autp}{{\rm Aut}^p}
\newcommand{\End}{{\rm End}}
\newcommand{\Hom}{{\rm Hom}}
\newcommand{\Ext}{{\rm Ext}}
\newcommand{\ord}{{\rm ord}}
\newcommand{\coker}{{\rm coker}\,}
\newcommand{\divisor}{{\rm div}}
\newcommand{\Def}{{\rm Def}}
\newcommand{\et}{{\rm \acute{e}t}}
\newcommand{\loc}{{\rm loc}}
\newcommand{\ab}{{\rm ab}}
\newcommand{\piet}{{\pi_1^{\rm \acute{e}t}}}
\newcommand{\pietloc}{{\pi_{\rm loc}^{\rm \acute{e}t}}}
\newcommand{\piN}{{\pi^{\rm N}_1}}
\newcommand{\piNloc}{{\pi_{\rm loc}^{\rm N}}}
\newcommand{\Het}[1]{{H_{\rm \acute{e}t}^{{#1}}}}
\newcommand{\Hfl}[1]{{H_{\rm fl}^{{#1}}}}
\newcommand{\Hcris}[1]{{H_{\rm cris}^{{#1}}}}
\newcommand{\HdR}[1]{{H_{\rm dR}^{{#1}}}}
\newcommand{\hdR}[1]{{h_{\rm dR}^{{#1}}}}
\newcommand{\Torsloc}{{\rm Tors}_{\rm loc}}
\newcommand{\defin}[1]{{\bf #1}}
\newcommand{\oX}{\cal{X}}
\newcommand{\oA}{\cal{A}}
\newcommand{\oY}{\cal{Y}}
\newcommand{\calC}{{\cal{C}}}
\newcommand{\calL}{{\cal{L}}}
\newcommand{\bmu}{\boldsymbol{\mu}}
\newcommand{\balpha}{\boldsymbol{\alpha}}
\newcommand{\bL}{{\mathbf{L}}}
\newcommand{\bM}{{\mathbf{M}}}
\newcommand{\bW}{{\mathbf{W}}}
\newcommand{\bD}{{\mathbf{D}}}
\newcommand{\bT}{{\mathbf{T}}}
\newcommand{\bO}{{\mathbf{O}}}
\newcommand{\bI}{{\mathbf{I}}}
\newcommand{\BD}{{\mathbf{BD}}}
\newcommand{\BT}{{\mathbf{BT}}}
\newcommand{\BI}{{\mathbf{BI}}}
\newcommand{\BO}{{\mathbf{BO}}}
\newcommand{\C}{{\mathbf{C}}}
\newcommand{\Dic}{{\mathbf{Dic}}}
\newcommand{\SL}{{\mathbf{SL}}}
\newcommand{\PSL}{{\mathbf{PSL}}}
\newcommand{\MC}{{\mathbf{MC}}}
\newcommand{\GL}{{\mathbf{GL}}}
\newcommand{\Tors}{{\mathbf{Tors}}}

\newcommand{\GM}[1]{\textcolor{blue}{(GM: #1)}}

\newcommand{\CL}[1]{\textcolor{red}{(CL: #1)}}

\newcommand{\YM}[1]{\textcolor{green}{(YM: #1)}}

\makeatletter
\@namedef{subjclassname@2020}{\textup{2020} Mathematics Subject Classification}
\makeatother

\title[LRQ singularities]{Linearly Reductive Quotient Singularities}

\author{Christian Liedtke}
\address{TU M\"unchen, Zentrum Mathematik - M11, Boltzmannstr. 3, 85748 Garching bei M\"unchen, Germany}
\email{liedtke@ma.tum.de}

\author{Gebhard Martin}
\address{Mathematisches Institut der Universit\"at Bonn, Endenicher Allee 60, 53115 Bonn, Germany}
\curraddr{}
\email{gmartin@math.uni-bonn.de}

\author{Yuya Matsumoto}
\address{Department of Mathematics, Faculty of Science and Technology, Tokyo University of Science, 2641 Yamazaki, Noda, Chiba, 278-8510, Japan}
\email{\url{matsumoto_yuya@rs.tus.ac.jp}}
\email{matsumoto.yuya.m@gmail.com}


\begin{abstract}
 We study isolated quotient singularities by finite and linearly reductive group schemes
 (lrq singularities for short) and show that they satisfy many, but not all, of the known properties of finite quotient singularities in characteristic zero:
 \begin{enumerate}
    \item From the lrq singularity we can recover the group scheme and the quotient presentation.
\item We establish canonical lifts to characteristic zero,  
which leads to a bijection between lrq singularities
and certain characteristic zero counterparts.
 \item We classify subgroup schemes of $\GL_d$ and $\SL_d$ that correspond to lrq singularities.
 For $d=2$, this generalises results of Klein, Brieskorn, and Hashimoto 
 \cite{Klein, Brieskorn, Hashimoto}.
 Also, our classification is closely related to the spherical space form problem.
     \item F-regular (resp. F-regular and Gorenstein) surface singularities
 are precisely the lrq singularities by finite and linearly reductive subgroup schemes
 of $\GL_2$ (resp. $\SL_2$).
 This generalises results of Klein and Du~Val \cite{Klein, DuVal}.
  \item Lrq singularities in dimension $\geq 4$ are infinitesimally rigid. 
   We classify lrq singularities in dimension $3$ that
   are not infinitesimally rigid
   and compute their deformation spaces.
   This generalises Schlessinger's rigidity theorem \cite{Schlessinger} 
   to positive and mixed characteristic.
 \end{enumerate}
 Finally, we study Riemenschneider's conjecture \cite{Riemenschneider} 
 in this context, that is, 
 whether lrq singularities deform to lrq 
 singularities.
\end{abstract}

\setcounter{tocdepth}{1}
\maketitle

\tableofcontents

\section{Introduction}

\subsection{Linearly reductive group schemes}
Let $k$ be an algebraically closed field of characteristic $p>0$, let
$G$ be a finite group scheme over $k$, and let 
$$
 1\,\to\,G^\circ\,\to\,G\,\to\,G^\et\,\to\,1
$$
be the \emph{connected-\'etale sequence}.

A particularly nice class of finite group schemes are the
\emph{linearly reductive} ones, which have been classified by Nagata \cite{Nagata61}:
these are precisely those $G$, where $G^\et$ is of order prime to $p$ and
where $G^\circ$, if non-trivial, is a finite subgroup scheme of $\GG_m^N$ for some $N$.
For example, the representation theory of a finite and linearly
reductive group scheme is equivalent to that of a finite group
over the complex numbers, see Corollary \ref{cor: reptheoriesthesame}.

Moreover, we will see in Proposition \ref{prop: linearizable} that if $G$ is a 
finite group scheme that acts on $\Spec k[[u_1,...,u_d]]$ freely outside the closed 
point and with schematic fixed locus 
equal to the closed point, then this action can be \emph{linearised} if and only if 
$G$ is linearly reductive.
This generalises a well-known lemma of Cartan \cite{Cartan}
in the complex case. 

In this article, we study linearly reductive group schemes $G$, actions as just
described, and the associated quotient singularities.
Most of our results are known in the case where $G$ is reduced, that is,
if $G^\circ$ is trivial, where they
become statements about groups of order prime to $p$.

\subsection{Very small representations}
A linear representation 
$$
 \rho\,:\,G\,\to\, \GL_d
$$ 
of a finite and linearly reductive group scheme is
\emph{very small} if for every $\bmu_n\subset G$ with $n\geq2$, 
the $\rho(\bmu_n)$-invariant subspace is zero-dimensional.
This is equivalent to the induced $G$-action on $\Spec k[[u_1,....,u_d]]$
being free outside the closed point (Proposition \ref{prop: lambda}).

The first objective of this article is the classification of finite 
linearly reductive group schemes that admit very small representations.
Rather than repeating the somewhat involved and technical results from 
Section \ref{sec: classificationofverysmall},
let us sketch the main ideas and reduction steps in this introduction:
\begin{enumerate}
  \item Every finite linearly reductive group scheme $G$ 
  admits a distinguished (``canonical'') geometric lift to characteristic
  zero (Proposition \ref{prop: liftinggroupscheme}),
  whose geometric points are a finite group that we
  denote $G_{\mathrm{abs}}$, the \emph{abstract group
  associated to $G$}.
  \item  Linear representations of finite linearly reductive group schemes
   admit distinguished lifts
  (Proposition \ref{prop: liftingrepresentation}).
  Specialisation yields a bijection 
  $$
   {\rm sp } \,:\, {\rm Rep}_{\CC}(G_{\mathrm{abs}}) \,\to\, {\rm Rep}_k(G)  
  $$
  between (very small) representations
  of a linearly reductive group scheme $G$ and (very small)
  complex representations of the group $G_{\mathrm{abs}}$,
  see Corollary \ref{cor: reptheoriesthesame} and Remark \ref{rem: cde}.
  \item Admitting a very small representation puts
  strong restrictions on a finite linearly reductive group scheme $G$:
  namely, every abelian subgroup scheme of $G_{\mathrm{abs}}$
  must be cyclic (Theorem \ref{thm: repsofGlr}).
\end{enumerate}
Finite groups all of whose abelian subgroups are cyclic
have been studied in different contexts 
by Milnor, Suzuki, Thomas, Wall, Wolf, Zassenhaus, and others,
see, for example, \cite{Milnor, Wolf, Zassenhaus}.

Putting all these results and observations together, we obtain
a (rough) classification of linearly reductive
group schemes that admit very small representations.
Let us highlight some of the results:
\begin{enumerate}
    \item We recover the classification of finite, small, and linearly 
    reductive subgroup schemes of $\SL_2$, due to
    Klein \cite{Klein} (characteristic zero) and Hashimoto \cite{Hashimoto}
    (positive characteristic), see Theorem \ref{thm: hashimoto}.
    \item We extend Brieskorn's classification \cite{Brieskorn}
    of finite and small subgroups
    of $\GL_{2}(\CC)$ to finite and linearly reductive subgroup schemes
    of $\GL_2$ (any characteristic), see Theorem \ref{thm: smallingl2}.
    \item We show that all very small subgroup schemes
    of $\GL_d$ for $d$ odd are cyclic or split metacyclic, see
    Theorem \ref{thm: verysmallgroupsinsmalldimension}.
    \item We obtain a classification of very small
    subgroup schemes of $\SL_3$ (any characteristic),
    see Corollary \ref{cor: SL3}.
   \item We show that if $d$ is a power of ${\rm char}(k)$, 
    then every finite and very small subgroup scheme of $\GL_{d,k}$ 
    is cyclic.
\end{enumerate}
We refer to Section \ref{sec: classificationofverysmall} for precise 
definitions and statements
and more applications, corollaries, and examples.

\subsection{Linearly reductive quotient singularities}
Associated to a finite and very small subgroup scheme $G\subseteq \GL_{d,k}$
over some algebraically closed field $k$ of characteristic $p\geq0$,
we have a $G$-action on $\mathbb{A}^d_k$ that is is free outside 
the closed point.
By definition, the quotient $\mathbb{A}_k^d/G$ (and every
$k$-scheme that is formally isomorphic to it) is called a
\emph{linearly reductive quotient singularity} or 
\emph{lrq singularity} for short.

It is known that invariant rings by linearly reductive group schemes 
satisfy many nice properties. 
We will summarise some of them in the following proposition, where 
most of the statements are more or less 
well-known --- if $G^\circ$ is trivial, then they are definitely 
well-known.

\begin{Proposition}[Proposition \ref{prop: lrq}]
 \label{prop: introduction}
  Assume that $G$ is a finite, linearly reductive, and very small subgroup scheme of $\GL_d$ that
  acts linearly on $S:=k[[u_1,...,u_d]]$ and let $R:=S^G$ be the ring of invariants.
  Then
  \begin{enumerate}
   \item The ring $R$ is F-regular.
   \item The class group of $R$ is isomorphic to $\Hom(G,\GG_m)$.
   \item The F-signature of $R$ is equal to $\frac{1}{|G|}$.
   \item The Hilbert--Kunz multiplicity of $(R,\idealm_R)$ is equal to
   $\frac{1}{|G|}{\rm length}(S/\idealm_RS)$.
   \item The local fundamental group of $\Spec R$ is isomorphic to $G^\et$.
  \end{enumerate}
\end{Proposition}

In particular, the F-signature detects the length of $G$, whereas
the Hilbert--Kunz multiplicity depends on the embedding of $G$
into $\GL_d$, see Remark \ref{rem: HK}.
Let us also note that isolated toric singularities that are $\mathbb{Q}$-factorial are precisely the lrq singularities by 
linearly reductive group schemes of the form $\bmu_n$, which might be more or less 
well-known:

\begin{Proposition}[Proposition \ref{prop: toric}]
If $x\in X$ be an isolated singularity of dimension $d$ over an algebraically closed
field $k$.
Then, the following are equivalent:
\begin{enumerate}
    \item $x\in X$ is toric and $\mathbb{Q}$-factorial,
    \item $x\in X$ is an lrq singularity by an abelian group scheme,
    \item $x\in X$ is an lrq singularity by $\bmu_n$ for some $n > 0$,
    \item $x\in X$ is an lrq singularity by $\bmu_n = {\rm Cl}(X)^D$, 
    where $\bmu_n$ is embedded into $\GL_d$ as $\zeta \mapsto {\rm diag}(\zeta^{q_1},\hdots,\zeta^{q_d})$ 
    with $1 = q_1 \leq \hdots \leq q_d < n$.
\end{enumerate}
\end{Proposition}

Isolated toric singularities are characterised among the isolated lrq singularities by the property that they have ``good reduction'' modulo every prime $p$ (in the sense of Definition \ref{def: goodreduction}).

\begin{Proposition}[Proposition \ref{prop: goodreductionimpliestoric}]
Let $X$ be an lrq singularity by a finite group $G$ over an algebraically closed 
field $K$ of characteristic $0$. 
Then, $X$ has good reduction modulo every prime $p \in \ZZ$ if and only if $X$ is 
toric and then, $X$ can be defined over $\ZZ$ (as an lrq singularity).
\end{Proposition}

\subsection{Uniqueness}

Over the complex numbers, the universal topological cover of 
a finite quotient and isolated singularity is smooth and the local fundamental group
acts on it by deck transformations. 
This shows uniqueness of the group of a finite quotient singularity,
as well as uniqueness of the action (up to conjugation).
For linearly reductive quotient singularities, we have a similar
uniqueness statement, but the proof is much more involved
(the reason is that if $G$ is not \'etale, then topological methods
do not give any information about $G^\circ$).

\begin{Theorem}[Theorem \ref{thm: lrqmain}]
 Let $R=S^G$ be as in Proposition \ref{prop: introduction}.
 If $H$ is another finite group scheme acting on $S$, 
 such that the $H$-action is free outside the closed point, $H$ fixes the 
 closed point, and such that the invariant ring $S^H$ is isomorphic to $R$, then 
 \begin{enumerate}
 \item $H$ is isomorphic to $G$,
 \item the $H$-action is linearisable, and then, 
 \item the two actions are conjugate in $\GL_d$.
 \end{enumerate}
\end{Theorem} 

Note that we do not assume $H$ to be linearly reductive  - this is a consequence
of the theorem.

\subsection{Rigidity of lrq singularities}
Over the complex numbers, a classical theorem of Schlessinger \cite{Schlessinger}
states that finite and isolated quotient singularities in dimension $\geq3$
are infinitesimally rigid.
Moreover, by loc.cit. this is still true in characteristic $p>0$
for quotient singularities by finite groups of order prime to $p$.

For lrq singularities, we will prove that Schlessinger's rigidity
theorem is still true in dimension $\geq4$,
which was already known to the experts as it
follows from a theorem of Satriano \cite{SatrianoDeRham}.
A little bit to our surprise, it turns out that
the situation is more complicated and subtle in dimension $3$.

\begin{Theorem}[Corollary \ref{cor: rigidity}]
Let $x \in X = 0 \in (\mathbb{A}^d/G)^\wedge$ be an lrq singularity. 
\begin{enumerate}
    \item If $d \geq 4$, then $X$ is infinitesimally rigid.
    \item If $d = 3$, then $X$ is rigid, but not necessarily infinitesimally rigid. 
\end{enumerate}
\end{Theorem}

Here, (infinitesimal) rigidity is defined in terms of deformations
over equicharacteristic Artin rings (Definition \ref{def: rigid2}).
An lrq singularity over an algebraically closed field $k$ of
characteristic $p>0$ admits a \emph{canonical lift}, see
Corollary \ref{cor: canonicallift}.
In particular, lrq singularities in mixed characteristic
are not rigid.
Nevertheless, we refer to Section \ref{subsec: arithmetic rigidity}
for deformations in mixed characteristic
and a version of arithmetic (infinitesimal) rigidity.

\subsection{Deformation spaces in dimension 3}
The previous theorem begs for the computation of the miniversal
deformation spaces of 3-dimensional lrq singularities.
To state the results, we consider the following two $G$-representations,
which are at most one dimensional.
\begin{enumerate}
  \item the adjoint representation $\chi_{\rm ad} :G\stackrel{{\rm ad}}{\to} \Aut({\rm Lie}(G^\circ)) \in \{\{\id\}, \GG_m\}$, 
  which is at most one-dimensional and depends on $G$ only, and 
  \item the determinant $\chi_{\rm det}:G\stackrel{\rho}{\to}\GL_3\to\GG_m$,
  which depends on the $G$-action $\rho$.
\end{enumerate}
Then, we can determine which 3-dimensional lrq singularities are infinitesimally
rigid and compute the deformation spaces in the other cases:

\begin{Theorem}[Theorem \ref{thm: non-rigid 3foldclassification} and Theorem \ref{thm: deformation space}]
 Let $x\in X=(\mathbb{A}^3/G)^\wedge$ be a 3-dimensional lrq singularity
 over an algebraically closed field $k$ of characteristic $p>0$.
 Then,
 \begin{enumerate}
     \item $x\in X$ is infinitesimally rigid if and only if 
     $G$ is \'etale or if $G$ is not \'etale and $\chi_{\rm ad}\neq\chi_{\rm det}$.
     \item If $x\in X$ is not infinitesimally rigid, then
     $G^\circ=\bmu_{p^n}$ for some $n\geq1$ and the miniversal deformation space is isomorphic to
   $$
    \Def_X \,\cong\, \Spec W(k)[\varepsilon]/(\varepsilon^2, p^n \varepsilon),
   $$
 where $W(k)$ denotes the Witt ring.
 \end{enumerate}
\end{Theorem}

In particular, miniversal deformation spaces of 3-dimensional lrq singularities
can be arbitrarily non-reduced.
We refer to Section \ref{subsec: nonrigidity} for precise statements, applications,
corollaries, and explicit examples.

We will discuss deformations of 2-dimensional
lrq singularities in Section \ref{subsec:introdeflrqdim2}.

\subsection{Characterisation of F-regular singularities in dimension 2}
Over the complex numbers and in dimension two, it is well-known 
that the \emph{canonical singularities} (resp. \emph{klt singularities}) 
are precisely the quotient singularities
by finite subgroups of $\SL_2$ (resp. $\GL_2$).

Generalising results of Hashimoto \cite{Hashimoto},
Liedtke--Satriano \cite{LiedtkeSatriano},
and using the classification of finite linearly reductive
group schemes of $\SL_2$ and $\GL_2$,
we show the following.

\begin{Theorem}[Theorem \ref{thm: FregularRDP}] \label{thm: Characterization of F-regular}
  For a surface singularity over an algebraically closed field of characteristic $p>0$, 
  the following are equivalent
  \begin{enumerate}
   \item it is an lrq singularity with respect to a finite and very small subgroup
   scheme of $\GL_2$ (resp. $\SL_2$)
   \item it is F-regular (resp. F-regular and Gorenstein).
  \end{enumerate}
  Moreover, if $p\geq7$, these are equivalent to
  \begin{enumerate}
   \setcounter{enumi}{2}
   \item it is a normal klt singularity (resp. a rational double point).
  \end{enumerate}
\end{Theorem}

We note that, if $p\geq7$, then the F-regular (resp. F-regular and Gorenstein)
surface singularities are precisely 
the klt singularities (resp. canonical singularities), 
as follows from Hara's classification \cite{Hara}. 
Hence, if $p\geq7$, then Theorem \ref{thm: Characterization of F-regular} 
establishes a similar picture as over the complex numbers.
We refer to Section \ref{sec: classificationofverysmall} for the subgroup schemes showing 
up and to Section \ref{sec: Fregular} for further details, especially in small characteristics.

As an application, we obtain the following
characteristic $p$ analog of a
theorem of Gurjar \cite{Gurjar} and Kumar \cite{Kumar},
a former conjecture of Wall \cite{Wall}.

\begin{Theorem}[Theorem \ref{thm: Wall}]
Let $x\in X$ be a two-dimensional singularity over an algebraically 
closed field $k$ of characteristic $p>0$.
Assume that $x\in X$ is isomorphic to $(\mathbb{A}^d/G)^\wedge$,
where $G$ is an affine group scheme acting linearly on $\mathbb{A}^d$
and where $G$ is an extension
$$
  1 \,\to\, G_1 \,\to\, G \,\to\, G_2 \,\to\, 1
$$
where $G_2$ is a linearly reductive group scheme 
(not necessarily finite) and where 
$G_2$ is connected reductive algebraic, such that
the action of $G_2$ on $\mathbb{A}^d$ is good in the sense of
\cite{HashimotoGoodfiltrations}.
Then, $x\in X$ is an lrq singularity.
\end{Theorem}

\subsection{Deformations of lrq singularities in dimension 2}
\label{subsec:introdeflrqdim2}
Riemenschneider \cite{Riemenschneider} conjectured
that finite quotient singularities over the complex numbers deform 
to finite quotient singularities.
In dimension $d\geq3$ this is true by Schlessinger's rigidity theorem
\cite{Schlessinger}.
In dimension 2, this conjecture was established by 
Esnault and Viehweg \cite{EsnaultViehweg2}.
Moreover, Koll\'ar and Shepherd-Barron \cite{KSB}
showed that cyclic quotient singularities deform 
to cyclic quotient singularities.
In the final section, we study the following analog for lrq singularities:

\begin{Conjecture}[Conjecture \ref{conj: Riemenschneider}]
  Let $B$ be the spectrum of a DVR with closed, generic, and geometric generic points
  $0$, $\eta$, and $\overline{\eta}$, respectively.
  Let $\mathcal{X}\to B$ be a flat family of $d$-dimensional
  singularities with special and geometric 
  generic fibre $\mathcal{X}_0$ and $\mathcal{X}_{\overline{\eta}}$, respectively. Assume that the non-smooth locus of $\mathcal{X} \to B$ is proper over $B$.
  \begin{enumerate}
    \item If $\mathcal{X}_0$ is an lrq singularity, 
      then $\mathcal{X}_{\overline{\eta}}$ contains at worst 
      lrq singularities.
    \item Let $G_0$ be the group scheme associated to $\mathcal{X}_0$
     and let $G_\eta$ be the group scheme associated to an lrq singularity
     on $\mathcal{X}_{\overline{\eta}}$. 
     Then, we have $|G_0|\geq|G_\eta|$, where $|-|$ denotes the
     length of a group scheme.
    \item If $\mathcal{X}_0$ is a cyclic lrq singularity, 
      then $\mathcal{X}_{\overline{\eta}}$
      contains at worst cyclic lrq singularities.
  \end{enumerate}
\end{Conjecture}

If $B$ is of equal characteristic zero, then Parts (1) and (3) are
true by the already mentioned results of \cite{EsnaultViehweg2} and \cite{KSB}.
As further evidence for this conjecture, we show the following.
\begin{enumerate}
  \item We establish it in dimension $d\geq3$ 
    (Proposition \ref{prop: conjecture dimension 3}).
  \item We establish it if ${\mathcal X}_0$ is Gorenstein
   (Corollary \ref{cor: riemenschneider gorenstein}).
  \item We establish Part (1) if $\mathcal{X}$ is $\QQ$-Gorenstein
   (Proposition \ref{prop: riemenschneider}).
   In recent work of Sato and Takagi \cite{SatoTakagi}, 
   they removed the $\QQ$-Gorenstein 
   condition, that is, Part (1) is now a theorem, see 
   Remark \ref{rem: SatoTakagi}.
   \item We establish Part (2) if 
    $B$ has equal characteristic zero and $\mathcal{X}_0$ is cyclic,
    as well as in some cases where $B$ has equal characteristic $p$
    (Proposition \ref{prop: riemenschneider length}).
  \item We give examples and counter-examples that illustrate that some
  more naive versions of this conjecture are not true.
\end{enumerate}

\subsection{Beyond lrq singularities}
The assumption on linear reductivity in the previous results is crucial -
rigidity in dimension $\geq3$ fails, deformations of quotient singularities
need no longer be quotient singularities, etc.
We will see this in examples in this article, as well as in the companion article
\cite{RDP}, where we study canonical surface singularities 
in positive characteristic.

\begin{VoidRoman}[Acknowledgements]
 We thank Fabio Bernasconi, Michael Collins, Brian Conrad, Leo Herr, Hiroyuki Ito, Gregor Kemper, Shravan Patankar, Quentin Posva, Kenta Sato,
 Matt Satriano, Karl Schwede, Nick Shepherd-Barron, and Claudia Stadlmayr
 for discussion. 
 We also thank the referees for their many comments.
 Part of the research on this article was done whilst the first named author was on sabbatical at the University of Oxford and he thanks the Mathematical Institute for kind hospitality during his stay.
 The first named author was supported by the ERC Consolidator Grant 681838 (K3CRYSTAL).
 The second named author was supported by the DFG Research Grant MA 8510/1-1.
 The third named author was supported by JSPS KAKENHI Grant Numbers JP16K17560
 and JP20K14296.
\end{VoidRoman}

\section{Linearly reductive group schemes}
\label{sec: linred}
In this section, we recall a couple results about finite group schemes over
algebraically closed fields $k$ of positive characteristic $p$.

\subsection{Generalities on (linearly reductive) group schemes}
\label{subsec: generalitieslinred}
Let $G$ be a finite group scheme over an algebraically closed 
field $k$ of characteristic $p\geq0$. 
Since $k$ is perfect, the canonical short exact sequence
of finite group schemes over $k$
\begin{equation}
    \label{eq:connected-etale}
   1\,\to\,G^\circ\,\to\,G\,\to\, G^{\et} \,\to\,1   
\end{equation}
splits,
where $G^\circ$ is the connected component of the identity, and $G^{\et}$
is an \'etale group scheme over $k$.
Since $k$ is algebraically closed, $G^{\et}$ is the constant group scheme associated
to a group and we will not distinguish between these two objects.
Thus, we have a canonical isomorphism $G\cong G^\circ\rtimes G^\et$.
Moreover, $G^\circ$ is an infinitesimal group scheme of length equal to 
some power of $p$.
In particular, if $p=0$ or if the length of $G$ is prime to $p$, 
then $G^\circ$ is trivial and $G$ is \'etale. 

If $M$ is a finitely generated abelian group, then the group algebra
$k[M]$ carries a Hopf algebra structure, and the associated commutative
group scheme is denoted $D(M):=\Spec k[M]$.
By definition, such group schemes are called \emph{diagonalisable}.
For example, we have $D(\ZZ)\cong\GG_m$ and
$D(\C_n)\cong\bmu_n$, where $\C_n$ denotes the cyclic group of order $n$.
Recall that every diagonalisable group scheme can be embedded
into $\GG_m^N$ for some $N\geq1$.
Moreover, $\bmu_n$ is \'etale over $k$ if and only if
$p\nmid n$.

A finite group scheme $G$ over $k$ is said to be \emph{linearly reductive} 
if every (finite-dimensional) representation of $G$ is semisimple.
If $p=0$, then all finite group schemes over $k$ are \'etale and
linearly reductive.
If $p>0$, then, by a theorem of  Nagata \cite[Theorem 2]{Nagata61}
(but see also \cite[Proposition 2.10]{AOV} and \cite[Section 2]{Hashimoto}), 
a finite group scheme over $k$ is
linearly reductive if and only if it is an extension of a finite and \'etale
group scheme, whose length is prime to $p$, 
by a diagonalisable group scheme. 

In particular, diagonalisable group schemes are examples of linearly reductive group schemes.
To obtain more examples of linearly reductive group schemes, we note the following. For a field $k$ and an abstract group $G$, we denote by $\underline{G}_k$ the constant group scheme associated to $G$ over $k$.

\begin{Lemma} \label{lem: lrequivalence}
For every algebraically closed field $k$ of characteristic $p\geq0$, the functor
$$
G^\circ \,\rtimes\, G^{\et} \,\mapsto\, 
\left(\underline{((G^\circ)^{D}(k))}_{\CC}\right)^{D}(\CC) \,\rtimes\, G^\et(k)
$$
is an equivalence of categories between the category of finite and linearly reductive group schemes over $k$ 
and the category of finite abstract groups with a unique abelian $p$-Sylow subgroup.
\end{Lemma}

For $p=0$, this reduces to the equivalence of categories of finite groups and 
finite constant group schemes over $k$. 

\begin{proof}
First, we note that if $G$ is a group with a unique $p$-Sylow subgroup $G_p$,
then $G_p$ is a normal subgroup.
Moreover, the quotient $G':=G/G_p$ is of order prime to $p$ and hence, 
$G$ is isomorphic to the semidirect product $G_p\rtimes G'$ 
by the Schur--Zassenhaus theorem.

An essential inverse to the functor of the lemma is given by the functor that 
maps a group $G = G_p \rtimes G'$, where $G_p$ is the unique $p$-Sylow subgroup 
and $p \nmid |G'|$, to the linearly reductive group scheme  $(\underline{((\underline{G_p}_{\CC})^D(\CC))}_k)^D \rtimes \underline{G'}_k$.
\end{proof}

In this lemma and the proof, the two functors can also be written as follows
$$
\begin{array}{lcl}
 G^\circ \rtimes G^{\et} &\mapsto& 
   \Hom((G^\circ)^{D}(k), \CC^{\times})\rtimes G^\et(k) \\
 G_p \rtimes G' &\mapsto&
  (\underline{(\Hom(G_p, \CC^{\times}))}_k)^D \rtimes \underline{G'}_k
\end{array}
$$
but this is a matter of taste.

\begin{Definition}\label{def: associatedlrgroupscheme}
Let $k$ be an algebraically closed field.
\begin{enumerate}
    \item If $G$ is a linearly reductive group scheme over $k$, we let $G_{\mathrm{abs}}$ be the image of $G$ under the equivalence of Lemma \ref{lem: lrequivalence} and call it the \emph{abstract group associated to $G$}.
    \item If $G$ is a finite and abstract group with a unique abelian $p$-Sylow subgroup, we let $G_{\mathrm{lr}}$ be the preimage of $G$ under the equivalence of Lemma \ref{lem: lrequivalence} and call it the \emph{linearly reductive group scheme associated to $G$} (over $k$).
    \end{enumerate}
\end{Definition}

Given a finite and infinitesimal group scheme $G$ over $k$,
the $k$-linear Frobenius morphism $F$ yields a canonical decomposition 
of $G$ 
$$
   1\,\unlhd\,G[F]\,\unlhd\,G[F^2]\,\unlhd\,\ldots\,\unlhd\,G[F^n]\,=\,G
$$
for some sufficiently large $n$.
By definition, the minimal $n$ for which we have $G=G[F^n]$ is called
the \emph{height} of $G$.
Each subquotient in this decomposition series is of height one.

\begin{Lemma}\label{lem: lrcriterion}
Let $G$ be a finite $k$-group scheme. 
Then, $G$ is linearly reductive if and only if it does not contain $\balpha_p$ or $\C_p$.
\end{Lemma}

\begin{proof}
By Nagata's classification, a linearly reductive group scheme does not contain 
$\balpha_p$ or $\C_p$, so we only have to prove the converse.

First, we recall from the exact sequence \eqref{eq:connected-etale} 
that we have a semidirect decomposition $G = G^\circ \rtimes G^{\et}$. 
The group $G^{\et}$ is linearly reductive if and only if it is of order
prime to $p$ if and only if $G$ does not contain a subgroup isomorphic to 
$\C_p$.
Thus, the claim is clear for $G^{\et}$ and we may assume $G = G^\circ$. 
We will prove the claim by induction on the height $n$ of $G$. 
Consider the short exact sequence
\begin{equation}\label{seslrqcriterion}
1 \,\to\, G[F^{n-1}] \,\to\, G[F^n] \,\to\, G[F^n]/G[F^{n-1}] \,\to\,1,
\end{equation}
where $G[F^{n-1}]$ is linearly reductive by the induction hypothesis. 
Since extensions of linearly reductive groups are linearly reductive, 
it suffices to show that $G[F^n]/G[F^{n-1}]$ is linearly reductive.

Assume that this is not the case. 
Then, either 
\begin{itemize}
    \item $G[F^n]/G[F^{n-1}]$ is not abelian.
    In this case its Lie algebra contains a non-zero vector $v$ such that $v^{[p]} = 0$ 
    (by Chwe's theorem \cite{Chwe}) and
    we can integrate $v$ to a subgroup scheme $\balpha_p \subseteq G[F^n]/G[F^{n-1}]$.
    \item $G[F^n]/G[F^{n-1}]$ is abelian, connected and not linearly reductive.
    In particular, the $p$-operation on its Lie algebra is a semi-linear endomorphism which is not semi-simple, hence not injective (see \cite[Exp. XXII, Section 1]{SGA72}. In particular, as in the previous case, we find a subgroup scheme $\balpha_p \subseteq G[F^n]/G[F^{n-1}]$.
\end{itemize} 
In any case, we find a subgroup scheme $\balpha_p \subseteq G[F^n]/G[F^{n-1}]$. 
Then, we can restrict the sequence \eqref{seslrqcriterion} to this $\balpha_p$ and obtain 
an extension of $\balpha_p$ by the linearly reductive group scheme $G[F^{n-1}]$. 
Since $k$ is perfect, all such extensions split by 
\cite[Exp. XVII, Th\'eor\`eme 6.1.1. B)]{SGA32}, so we find $\balpha_p \subseteq G$, 
contradicting our hypothesis.
Hence, $G[F^n]/G[F^{n-1}]$ is linearly reductive and thus so is 
$G[F^n]$.
\end{proof}

\subsection{Deformations of linearly reductive group schemes} \label{subsec: Deformations of linearly reductive group schemes}
The following result, which should be more or less well-known to the experts, 
states that finite linearly reductive group schemes lift to characteristic 0
and that any two lifts are geometrically isomorphic.

\begin{Proposition}\label{prop: liftinggroupscheme}
 Let $k$ be an algebraically closed field of characteristic $p>0$
 and let $R$ be a complete DVR with residue field $k$.
 Let $G = \prod \bmu_{p^{n_i}} \times G^\et$ be a finite and linearly reductive group scheme over $k$.
 Then, 
 \begin{enumerate}
    \item $\widetilde{G} := \prod \bmu_{p^{n_i},R} \rtimes \underline{(G^\et (k))}_R$ is a deformation of $G$ to $R$,
    \item its automorphism scheme satisfies $\Aut_{\widetilde{G}}^\circ \cong (\widetilde{G}^{\circ})/(\widetilde{G}^{\circ})^{\widetilde{G}^{\et}}$,
    \item every other deformation $\widetilde{G}'$ of $G$ to $R$ is a twisted form of $\widetilde{G}$ that splits over a totally ramified, finite, and flat extension $S$ of $R$ of degree $p^i \leq |G^\circ/(G^\circ)^{G^\et}|$ for some $i \geq 0$.
    Moreover, if ${\rm char}(R) = p$, then the extension $R \subseteq S$ can be chosen to be purely inseparable.   
  \end{enumerate}
\end{Proposition}

\begin{proof}
Claim (1) is clear, since $\widetilde{G}$ is flat over $R$ with special fibre isomorphic to $G$. 

Next, consider Claim (2).
Since $\widetilde{G}^\circ = \prod \bmu_{p^{n_i},R}$ is a characteristic subgroup scheme of $\widetilde{G}$, we have an exact sequence
$$
0 \,\to\, E \,\to\, \Aut_{\widetilde{G}} \,\overset{\varphi}{\to}\, 
\Aut_{\widetilde{G}^\circ} \times \Aut_{\widetilde{G}^\et},
$$
where $E := \Ker(\varphi)$ and ${\rm Im}(\varphi)$ is a subgroup scheme of the \'etale group scheme $\Aut_{\widetilde{G}^\circ} \times \Aut_{\widetilde{G}^\et}$. 
On the other hand, the conjugation action of $\widetilde{G}^\circ$ on $\widetilde{G}$ induces a homomorphism 
$\psi: \widetilde{G}^\circ \to E$,
which induces an isomorphism of the closed fibre of $\widetilde{G}^\circ/(\widetilde{G}^\circ)^{\widetilde{G}^{\et}}$ with the closed fibre of $E$ by \cite[Lemma 2.24]{AOV}. 
By the fibrewise criterion for flatness, we deduce that $E$ is flat over $R$. 
Since $G^\circ/(G^\circ)^{G^\et}$ is diagonalisable, hence rigid, both $E$ and $\widetilde{G}^\circ/(\widetilde{G}^\circ)^{\widetilde{G}^{\et}}$ 
coincide with the unique deformation of $G^\circ/(G^\circ)^{G^\et}$ to $R$. 
Note that $E \cong \Aut_{\widetilde{G}}^\circ$, since ${\rm Im}(\varphi)$ is \'etale and 
$E_k = \Aut_G^{\circ}$.

Finally, let us prove Claim (3). 
By \cite[Lemma 2.14]{AOV} there is a $(\widetilde{G},\widetilde{G}')$-bitorsor 
$I \to \Spec R$. 
In particular, $\widetilde{G}'$ is a twisted form of $\widetilde{G}$. 
Thus, $S' := {\rm Isom}(\widetilde{G},\widetilde{G}')$ is an $\Aut_{\widetilde{G}}$-torsor 
over $R$ and, since $R$ is complete, $S'$ is a disjoint union of $|\Aut_{\widetilde{G}}/\Aut_{\widetilde{G}}^\circ|$ connected components, 
each of which is an $\Aut_{\widetilde{G}}^\circ$-torsor over $R$. 
Taking the normalisation of one of these components, we obtain the desired $S$. 
Indeed, $R$ is complete, so $S$ is totally ramified over $R$, the fibres of $\Aut_{\widetilde{G}}^\circ \to \Spec R$ are local group schemes if ${\rm char}(R) = p$, 
so $S$ is purely inseparable over $R$ if ${\rm char}(R) = p$, and the degree of 
$S$ over $R$ divides the length of $\Aut_{\widetilde{G}}^\circ$ over $R$, 
which coincides with $|G^\circ/(G^\circ)^{G^\et}|$ by Claim (2). 
In particular, the degree of $S$ over $R$ is a power of $p$. 
\end{proof}

In particular, Proposition \ref{prop: liftinggroupscheme} provides us with a ``canonical" deformation of $G$ to $R$.

\begin{Definition} \label{def: canonical deformation}
Let $G \cong \prod \bmu_{p^{n_i}} \rtimes G^\et$ be a finite and linearly reductive group scheme over an algebraically closed field $k$ of characteristic $p > 0$. 
Let $R$ be a complete DVR with residue field $k$. 
The \emph{canonical deformation of $G$ to $R$} (or, \emph{canonical lift of $G$ to characteristic $0$}, if $R = W(k)$) 
is the $R$-group scheme $\prod \bmu_{p^{n_i},R} \rtimes \underline{(G^\et (k))}_R$.
\end{Definition}

The following example shows that there exist linearly reductive group schemes 
that admit ``non-canonical" deformations.

\begin{Example} \label{example: explicitnontrivialdeformation}
Assume $p\geq3$, let $R$ be as in Proposition \ref{prop: liftinggroupscheme}, 
and consider the non-trivial semidirect product 
$G := \bmu_{p} \rtimes \C_2$, where $\C_2$ acts as inversion on $\bmu_{p}$.
Let $\widetilde{G} := \bmu_{p,R} \rtimes \C_2$ be the canonical deformation of $G$
to $R$.
By Proposition \ref{prop: liftinggroupscheme}, we have ${\rm Aut}_{\widetilde{G}}^\circ \cong \bmu_{p}$. 
Thus, every non-trivial element of $\Hfl1(R,\bmu_{p}) = R^\times/(R^\times)^{p}$ 
yields a twisted form of $\widetilde{G}$ over $R$ and hence, a non-trivial deformation of 
$G$ to $R$. 
We note that $R^\times/(R^\times)^{p}$ is non-trivial, since it contains the class of 
$1 + \pi$, where $\pi$ is a uniformiser of $R$, and this element is not a $p$-th power.
\end{Example}

\subsection{Deformations of representations}

If $\widetilde{G}$ is a group scheme over a ring $R$, we denote by 
${\rm Rep}^d_R(\widetilde{G})$ the set of isomorphism classes of free $R$-modules 
of rank $d$ together with an action of $G$, 
that is, ${\rm Rep}^d_R(\widetilde{G})$ is the set of homomorphisms in 
${\rm Hom}(\widetilde{G},\GL_{d,R})$ up to conjugation in $\GL_{d,R}$.

\begin{Definition}
\label{def: lambda}
Let $V$ be a vector space over an algebraically closed field $k$ and let $\rho: G \to \GL(V)$ be a representation of a linearly reductive group scheme $G$. We define the \emph{$\lambda$-invariant of $\rho$} as
$$
\lambda(\rho) \,:=\, {\rm max}_{\{\id\} \neq \bmu_n \subseteq G}\:\dim V^{\bmu_n},
$$
where $V^{\bmu_n} \subseteq V$ denotes the subspace of $\bmu_n$-invariant vectors.
\end{Definition}

\begin{Remark}
\label{rem: lambda}
 Although the $\lambda$-invariant may not have been
 studied before, let us note that the representation
 \begin{enumerate}
     \item $\rho$ is faithful if and only if $\lambda(\rho) \neq \dim V$ and that
     \item $\rho$ contains no pseudo-reflections if and only if
     $\lambda(\rho)\leq\dim(V)-2$,
     see also Section \ref{subsec: more general}.
     Such representations are called \emph{small}.
     \item If $G$ is \'etale, then the $\lambda$-invariant can be defined as 
     $$
      \max_{{\id} \neq g \in G}\:{\rm dim}(V_{g,1}),
     $$
     where $V_{g,1}$ is the eigenspace of the $\rho(g)$-action for the eigenvalue $1$.
     
     In particular, we have
     $\lambda(\rho)=0$ (resp. $\lambda(\rho)\leq\dim V-2$)
     if and only if for every $g\in G\backslash\{e\}$, the linear map $\rho(g)$ does not fix a line
     (resp. a hyperplane).
 \end{enumerate}
 We will give a geometric interpretation of $\lambda(\rho)$
 in Proposition \ref{prop: lambda}.
\end{Remark}

\begin{Proposition} \label{prop: liftingrepresentation}
 Let $k$ be an algebraically closed field of characteristic $p>0$
 and let $R$ be a complete DVR with residue field $k$ and field of fractions $K$.
 Let $G$ be a finite and linearly reductive group scheme over $k$ and let $\widetilde{G}$ be a deformation of $G$ to $R$, with generic fibre $\widetilde{G}_\eta$. Then, for every $d \geq 0$, the following hold:
 
\begin{enumerate}
\item Restriction of $\widetilde{G}$-representations to $K$ induces a bijection
$$
 {\rm Rep}^d_{R}(\widetilde{G}) \,\overset{\sim}{\to}\, {\rm Rep}^d_{K}(\widetilde{G}_{\eta})
$$
\item Restriction of $\widetilde{G}$-representations to $k$ induces a bijection
$$
  {\rm Rep}^d_{R}(\widetilde{G}) \,\overset{\sim}{\to}\, {\rm Rep}^d_{k}(G)
$$
\item There is a bijective specialisation map
$$
{\rm sp}_d\,:\, {\rm Rep}^d_{\overline{K}}(\widetilde{G}_{\bar{\eta}}) \,\to\, {\rm Rep}^d_{k}(G),
$$
such that ${\rm sp} := \coprod_{d=0}^{\infty} {\rm sp}_d$ is compatible with direct sums, tensor products, duals, $\lambda$-invariants, and characters. In particular, ${\rm sp} $ maps simple representations to simple representations.
\end{enumerate}
\end{Proposition}

\begin{proof}
We consider the functor $\overline{{\rm Hom}}(\widetilde{G},\GL_{d,R})$
of homomorphisms of group schemes over $R$ up to conjugation in $\GL_{d,R}$. 
This means that 
${\rm Rep}^d_S(\widetilde{G}_S) = \overline{{\rm Hom}}(\widetilde{G},\GL_{d,R})(S)$ 
for every $R$-algebra $S$.

Let us first prove Claim (1). 
By \cite[Lemme 2]{SerreLinRepGroupScheme}, every $G_{\eta}$-representation admits a
$\widetilde{G}$-invariant $R$-sublattice of full rank, that is, the restriction map 
${\rm Hom}(\widetilde{G},\GL_{d,R})(R) \to {\rm Hom}(\widetilde{G},\GL_{d,R})(K)$ 
is surjective. 
Thus, to show (1), we have to show that two $R$-representations $u,v$ of $\widetilde{G}$ 
are conjugate in $\GL_{d,R}$ if and only if their generic fibres $u_K$ and $v_K$ 
are conjugate in $\GL_{d,K}$. 
It is clear that $u_K$ and $v_K$ are conjugate if $u$ and $v$ are. 
For the converse, we use that the transporter functor ${\rm Transp}(u,v)$ of sections of
$\GL_{d,R}$ conjugating $u$ to $v$ is representable by a smooth closed subscheme of 
$\GL_{d,R}$ by \cite[Theorem 4.5, Lemma 4.7]{Margaux}. 
In particular, if the generic fibre of ${\rm Transp}(u,v)$ is non-empty 
(that is, if $u_K$ and $v_K$ are conjugate), then the special fibre of 
${\rm Transp}(u,v)$ is non-empty and Hensel's Lemma provides us with an element of 
$\GL_{d,R}(R)$ conjugating $u$ to $v$.

Next, consider Claim (2). 
By \cite[Lemma 4.4, Theorem 4.5]{Margaux}, the functor 
$\overline{{\rm Hom}}(\widetilde{G},\GL_{d,R})$ is formally \'etale and locally 
of finite presentation. 
Since $R$ is complete, we immediately obtain the claimed bijection.

For Claim (3), we note that Claims (1) and (2) apply to every finite field extension 
$L$ of $K$ with $R$ replaced by the integral closure $R_L$ of $R$ in $L$. 
We can thus define ${\rm sp}_d$ as the following chain of bijections:
\begin{eqnarray*}
{\rm Rep}^d_{\overline{K}}(G_{\bar{\eta}}) &=& 
\overline{{\rm Hom}}(\widetilde{G},\GL_{d,R})(\overline{K}) \\
&=& \varinjlim_{K \subseteq L \text{ finite}} \overline{{\rm Hom}}(\widetilde{G},\GL_{d,R})(L) \\
&=& \varinjlim_{K \subseteq L \text{ finite}} \overline{{\rm Hom}}(\widetilde{G},\GL_{d,R})(R_L) \\
&\to&\overline{{\rm Hom}}(\widetilde{G},\GL_{d,R})(k) 
\,=\, {\rm Rep}_k^d(G).
\end{eqnarray*}
Since ${\rm sp} := \coprod_{d=0}^{\infty} {\rm sp}_d$ is a coproduct of compositions of restriction maps and their inverses, both of which satisfy the stated compatibilities, ${\rm sp}$ satisfies the compatibilities as well. 
\end{proof}

\subsection{Representation theory of linearly reductive group schemes and 
their associated abstract groups}

The following proposition shows that the abstract group $G_{\mathrm{abs}}$ associated to a 
linearly reductive group scheme $G$ over an algebraically closed field $k$ of 
characteristic $p > 0$ appears naturally as the geometric generic fibre of any 
lift of $G$ to characteristic $0$.

\begin{Proposition} \label{prop: geometricgenericfiberisassociated}
Let $G$ be a finite abstract group. 
Let $k$ be an algebraically closed field of characteristic $p > 0$,
let $W(k)$ be the ring of Witt vectors, let $K:={\rm Frac}(W(k))$,
and let $\bar{K}$ be an algebraic closure. 
Then, the constant group scheme $\underline{G}_{\bar{K}}$ 
is the geometric generic 
fibre of a lift of a linearly reductive group scheme $H$ over $k$
to characteristic $0$ if and only if $G$ has a unique abelian $p$-Sylow subgroup and $H \cong G_{\mathrm{lr}}$.
\end{Proposition}

\begin{proof}
Using Nagata's classification, we can write $H = H^\circ \rtimes H^{\et}$, 
where $H^\circ \cong \prod_{i=1}^l \bmu_{p^{n_i}}$ for some $n_i \geq 0$ and 
where $H^{\et}$ is finite \'etale of length prime to $p$. 
Let $\widetilde{H}$ be any lift of $H$ to $W(k)$. By Proposition \ref{prop: liftinggroupscheme}, we have $G = \widetilde{H}(\bar{K}) \cong \prod_{i=1}^l \C_{p^{n_i}} \rtimes H^{\et}(k)$. In particular, $G$ has a unique abelian $p$-Sylow subgroup and we have $G_{\mathrm{lr}} \cong H$, 
which is what we had to prove.

For the converse, it suffices to note that $G_{\overline{K}}$ is the geometric 
generic fibre of the canonical lift of $G_{\mathrm{lr}}$ to characteristic $0$ 
as in Definition \ref{def: canonical deformation}. 
\end{proof}

Since we already know from Proposition \ref{prop: liftingrepresentation} that the representation theory of a linearly reductive group scheme $G_{\mathrm{lr}}$ and of the geometric generic fibre of a lift of $G_{\mathrm{lr}}$ to characteristic $0$ coincide, we obtain the following corollary.

\begin{Corollary}\label{cor: reptheoriesthesame}
Let $k$ be an algebraically closed field of characteristic $p \geq 0$ and let $G$ be a finite and linearly reductive group scheme over $k$. Then, for every $d \geq 0$, there is a bijection
$$
{\rm sp}_d \,:\, {\rm Rep}^d_{\CC}(G_{\mathrm{abs}}) \,\cong\, 
{\rm Rep}^d_{k}(G),
$$
such that ${\rm sp} := \coprod_{d=0}^{\infty} {\rm sp}_d$ is compatible with direct sums, tensor products, duals, $\lambda$-invariants, and characters. In particular, ${\rm sp}$ maps simple representations to simple representations.
\end{Corollary}

\begin{Remark}
 \label{rem: cde}
 This can also be phrased in the language of cde-triangles
 as in \cite[Part III]{SerreLinRep}.
 We denote by $R_F(H)$ (resp. $P_F(H)$) the Grothendieck ring
 of finite (resp. finite projective) $H$-modules over some field $F$.
 Keeping the notations and assumptions of the previous Corollary, we 
 have the following:
 \begin{enumerate}
   \item The Cartan morphism $c:P_k(G)\to R_k(G)$ is an isomorphism
   since $G$ is linearly reductive.
   \item We define the decomposition $d:R_\CC(G_{\mathrm{abs}})\to R_k(G)$
   via ${\rm sp}$.
   \item We define the extension $e:P_k(G)\to R_\CC(G_{\mathrm{abs}})$
   via the inverse of ${\rm sp}$ or using the lifting of representations.
 \end{enumerate}
 We leave the details including commutativity of the cde-triangle to the
 reader.
\end{Remark}

The results of this section lead to the following classification result, which is interesting in its own.

\begin{Theorem}
  \label{thm: linearlyreductivesubgroupschemeGL}
  Let $k$ be an algebraically closed field of
  characteristic $p>0$.
  Lifting induces a canonical  map
 \begin{eqnarray*}
 &&
  \left\{
  \mbox{finite linearly reductive subgroup schemes of }\GL_{d,k}
  \mbox{ up to conjugation} \right\} 
  \\
  &\to&
    \left\{
  \mbox{finite subgroups of }\GL_{d}(\CC)
  \mbox{ up to conjugation}
  \right\}  \,.
  \end{eqnarray*}
 It has the following properties:
 \begin{enumerate}
     \item It is injective.
     \item The image is the set of those 
     conjugacy classes of
  finite subgroups of $\GL_d(\CC)$ that 
  have a unique abelian $p$-Sylow subgroup (possibly trivial).
   \item If $p>2d+1$, then it is surjective.
   \item  It preserves $\lambda$-invariants.
 In particular, small (resp. very small) subgroup group schemes
 get mapped to small (resp. very small) subgroups.
  \item It maps subgroup schemes of $\SL_{d,k}$ to subgroups of $\SL_d(\CC)$.
 \end{enumerate}

\end{Theorem}

\begin{proof}
The existence of the map, injectivity, and compatibility 
with $\lambda$-invariants and with $\SL_2$ 
follow from Corollary \ref{cor: reptheoriesthesame}.
Assertion (2) follows from Proposition \ref{prop: geometricgenericfiberisassociated}.
If $p>2d+1$, then every finite subgroup of $\GL_d(\CC)$ is either of order prime to $p$ or has a unique abelian $p$-Sylow subgroup \cite{FeitThompson}.
This shows Assertion (3).
\end{proof}

\begin{Remark}
 If $p\equiv1\mod4$ (resp. $p\equiv3\mod4$),
 then $\SL_2(\FF_p)$ (resp. $\PSL_2(\FF_p)$)
 is a subgroup of $\GL_d(\CC)$ with $p=2d+1$.
 The $p$-Sylow subgroups of these groups
 are not normal, which 
 shows that the bound in (3) of the theorem 
 is sharp.
\end{Remark}

\section{Very small linearly reductive subgroup schemes of $\GL_d$}
\label{sec: classificationofverysmall}

In this section, we explain how to classify representations of finite and linearly 
reductive group schemes admitting a representation with $\lambda=0$ 
(see Definition \ref{def: lambda}) in arbitrary dimension $d$. 
We will make this classification explicit if $d \leq 3$, 
if $d$ is odd and if $d$ is a power of ${\rm char}(k)$. 
We will see in Section \ref{sec: quotient singularities} that these 
are precisely the representations leading to lrq singularities.
In particular, the results of this section yield a complete classification of lrq singularities.

\subsection{Very small representations}

\begin{Definition} \label{def: verysmallrepresentation}
Let $G$ be a linearly reductive group scheme over an algebraically closed field. 
A $d$-dimensional representation $\rho$ of $G$ is called \emph{very small} if $\lambda(\rho) = 0$. 
In this case, we say that $\rho(G)$ is a \emph{very small subgroup scheme of} $\GL_d$.
\end{Definition}

We note that a very small representation is faithful and that the
name \emph{very small} is motivated by the terminology of
$\rho$ being \emph{small} if $\lambda(\rho)\leq d-2$, 
see Remark \ref{rem: lambda}.
We refer to Proposition \ref{prop: lambda} 
for a geometric interpretation in terms of fixed loci.

Next, we have the following corollary of our analysis
of representations of linearly reductive group schemes 
from Section \ref{sec: linred}, which puts strong restrictions
on them and which is the key to our classification results.

\begin{Theorem}\label{thm: repsofGlr}
Let $G$ be a linearly reductive group scheme over an algebraically closed field of characteristic $p \geq 0$. Then, the following hold:
\begin{enumerate}
 \item The bijection ${\rm sp}_d: {\rm Rep}^d_{k}(G) \to {\rm Rep}^d_{\CC}(G_{\mathrm{abs}})$ identifies the subsets of very small representations.
 \item If $G$ admits a very small representation, then every abelian subgroup of $G_{\mathrm{abs}}$ is cyclic. 
\end{enumerate}
\end{Theorem}

\begin{proof}
Claim (1) is a special case of Corollary \ref{cor: reptheoriesthesame}.

For Claim (2), let $H \subseteq G_{\mathrm{abs}}$ be an abelian subgroup and let 
$\rho: G_{\mathrm{abs}} \to \GL_{d,\CC}(\CC)$ be a very small representation of $G_{\mathrm{abs}}$, 
which exists by Claim (1). 
Since $H$ is abelian, we may conjugate $\rho$ to assume that $H$ acts diagonally. 
Consider the restriction $f: H \to \CC^{\times}$ of the $H$-representation to a 
coordinate axis. 
Then, $\Ker(f)$ is trivial since $\lambda(\rho) = 0$. 
Thus, $H$ is isomorphic to a finite, hence cyclic, subgroup of $\CC^\times$, 
which is what we wanted to show.
\end{proof}

Thus, for every dimension $d$ and every algebraically closed field $k$ 
of characteristic $p \geq 0$, the classification of finite, 
linearly reductive, and very small subgroup schemes of $\GL_{d,k}$ is the same 
(by passing to the associated abstract group) 
as the classification of finite and very small subgroups of $\GL_{d,\CC}(\CC)$ 
that admit a unique cyclic normal $p$-Sylow subgroup. 

Next, the study of finite groups that admit very small representations was initiated by 
Zassenhaus \cite{Zassenhaus} (motivated by questions in near-fields), followed by important
contributions by Suzuki and others, and a complete classification was achieved 
(motivated by questions in 
differential geometry, see Section \ref{subsec: differentialgeometry})
by Milnor, Thomas, Wall, Wolf, and others, see \cite{Wolf}.
Instead of repeating this classification here, we will focus on interesting special cases, 
such as $d = 2$ or $d$ odd. 

Let us also note that the class of finite groups, all of whose abelian subgroups are cyclic, has interesting reformulations and characterisations:
they are precisely the groups with periodic cohomology
(which makes them interesting for topologists)
and they can be characterised by the structure
of their Sylow subgroups.
We refer to \cite[Chapter XII, Theorem 11.6]{CE} 
for details.

\subsection{Very small linearly reductive subgroup schemes of $\GL_2$}

Let us first recall the classification of finite subgroup schemes 
of $\SL_2$ over arbitrary fields from \cite{Hashimoto}, which extends the classification
of finite subgroups of $\SL_{2,\CC}(\CC)$ from \cite{Klein}.
We adapt these results to our setting.

\begin{Theorem}[Hashimoto, Klein]
\label{thm: hashimoto}
 Let $G$ be a finite and linearly reductive subgroup scheme of $\SL_{2,k}$
 over an algebraically closed field $k$ of characteristic $p\geq0$.
 Then, $G$ is conjugate to one of the following, where $\zeta_r$ denotes a primitive 
 $r$th root of unity.
 \begin{enumerate}
 \item $(n\geq1)$ The group scheme $\bmu_{n}$ of length $n$ embedded in $\SL_{2,k}$ as
$$
\left\{
 \left(
 \begin{array}{cc}
   a& 0\\ 0&a^{-1} 
 \end{array}\right)\,, \,a\in\bmu_{n}
 \right\} \,.
$$
 \item $(n\geq2, \,p\geq3)$ The binary dihedral group scheme $\BD_{n}$ of length $4n$
 generated by $\bmu_{2n}\subset\SL_{2,k}$ as in (1)  and
 $$
 \left(
 \begin{array}{cc}
 0 & \zeta_4 \\
 \zeta_4 & 0
 \end{array}
 \right)\,.
 $$
 \item $(p\geq5$) The binary tetrahedral group scheme $\BT_{24}$ of length $24$ generated by 
 $\BD_2\subset\SL_{2,k}$ as in (2) and 
$$
 \frac{1}{\sqrt{2}} \left(
 \begin{array}{cc}
   \zeta_8^7 & \zeta_8^7 \\
   \zeta_8^5 & \zeta_8
 \end{array}
 \right)\,.
$$
\item $(p\geq5)$ The binary octahedral group scheme $\BO_{48}$ of length $48$ generated by 
$\BT_{24}\subset\SL_{2,k}$ as in (3) and $\bmu_8\subset\SL_{2,k}$ as in (1).
\item $(p\geq7$) The binary icosahedral group scheme $\BI_{120}$ of length $120$ generated by 
$\bmu_{10}\subset\SL_{2,k}$ as in (1),
$$
 \left(\begin{array}{cc}
  0 & 1 \\ -1 & 0
 \end{array}\right),
 \mbox{ \quad and \quad }
 \frac{1}{\zeta_5^2-\zeta_5^3}
 \left(\begin{array}{cc}
   \zeta_5+\zeta_5^{-1} & 1 \\ 1 & -(\zeta_5+\zeta_5^{-1})
 \end{array}\right)\,.
$$
 \end{enumerate}
 Conversely, any of the above is a linearly reductive group scheme of $\SL_{2,k}$
 in the indicated characteristics.
\end{Theorem}

\begin{proof}
See \cite[Theorem 3.8]{Hashimoto}. Note also that, using Theorem \ref{thm: repsofGlr}, this follows immediately from the classification of finite subgroups of $\SL_{2,\CC}(\CC)$, which goes back to Klein \cite{Klein}, by passing to the linearly reductive group schemes associated to the finite subgroups of $\SL_{2,\CC}(\CC)$ with a unique abelian $p$-Sylow subgroup.
\end{proof}

To state the classification of very small linearly reductive subgroup schemes of $\GL_{2,k}$, we follow Brieskorn \cite[Section 2.4]{Brieskorn}. 
We let 
$$
  \psi \,:\, \GG_m \times \SL_{2,k} \,\to\, \GL_{2,k}
$$
be the multiplication map, where $\GG_m \subseteq \GL_{2,k}$ is 
the diagonal torus.
Let $H_1 \subseteq \GG_m$ and $H_2 \subseteq \SL_{2,k}$ be finite and linearly reductive subgroup schemes, 
let $N_i \subseteq H_i$ be normal subgroup schemes with projection maps $\pi_i: H_i \to H_i/N_i$, 
and assume that there exists an isomorphism $\varphi: H_2/N_2 \to H_1/N_1$. Then, we define
$$
(H_1,N_1;H_2,N_2)_{\varphi} \,:=\, \psi(H_1 \times_{\pi_1,H_1/N_1,\varphi \circ \pi_2} H_2),
$$
where we consider $H_1 \times_{\pi_1,H_1/N_1,\varphi \circ \pi_2} H_2$ as a subgroup scheme 
of $\GG_m \times \SL_{2,k}$ via its natural embedding into $H_1 \times H_2$. 
In particular, since linear reductivity is stable under taking quotients and subgroups, 
the group scheme $(H_1,N_1;H_2,N_2)_{\varphi}$ is linearly reductive. 
If the length of $H_i/N_i$ is at most $3$, then the conjugacy class of $(H_1,N_1;H_2,N_2)_{\varphi}$ 
does not depend on $\varphi$,
so we will drop $\varphi$ from our notation in these cases.

Moreover, for $1 \leq q < n$ with $(n,q) = 1$, we define $\bmu_{n,q}$ as the group scheme 
$\bmu_n$ of length $n$ embedded in $\GL_{2,k}$ as
$$
\left\{
 \left(
 \begin{array}{cc}
   a& 0\\ 0&a^{q} 
 \end{array}\right)\,, \,a\in\bmu_{n}
 \right\} \,.
$$

After this preparation, the classification of small linearly reductive subgroup schemes of $\GL_{2,k}$ follows from \cite[Satz 2.9]{Brieskorn} and is as follows.

\begin{Theorem}
\label{thm: smallingl2}
 Let $G$ be a finite, very small and linearly reductive subgroup scheme of $\GL_{2,k}$
 over an algebraically closed field $k$ of characteristic $p\geq0$.
 Then, the following hold:
 \begin{itemize}
 \item[(i)]
 $G$ is conjugate to one of the following.
 \begin{enumerate}
     \item $(n \ge 1, (n,q) = 1)$
     
     The cyclic group scheme $\bmu_{n,q}$
     \item[(2a)] $(n \geq 2, (m,2) = (m,n) = 1, p \geq 3)$ 
     
     The group scheme $(\bmu_{2m,1},\bmu_{2m,1};\BD_n,\BD_n)$.
     \item[(2b)] $(n \geq 2, (m,2) = 2, (m,n) = 1, p \geq 3)$ 
     
     The group scheme $(\bmu_{4m,1},\bmu_{2m,1};\BD_n,\bmu_{2n,2n-1})$.
     
     \item[(3a)] $((m,6) = 1, p\geq 5)$
     
     The group scheme $(\bmu_{2m,1},\bmu_{2m,1};\BT_{24},\BT_{24})$.
     
      \item[(3b)] $((m,6) = 3, p\geq 5)$
     
     The group scheme $(\bmu_{6m,1},\bmu_{2m,1};\BT_{24},\BD_2)$.
     
       \item[(4)] $((m,6) = 1, p\geq 5)$
     
     The group scheme $(\bmu_{2m,1},\bmu_{2m,1};\BO_{48},\BO_{48})$.
     
       \item[(5)] $((m,30) = 1, p \geq 7)$
       
      The group scheme $(\bmu_{2m,1},\bmu_{2m,1};\BI_{120},\BI_{120})$.
       
 \end{enumerate}
 \item [(ii)] Two of these group schemes $G_1$ and $G_2$ are conjugate if and only if they are equal 
 or $G_1 = \bmu_{n,q}$ and $G_2 = \bmu_{n,q'}$ with $qq' = 1$ modulo $n$.
 \end{itemize}
 Conversely, any of the above is a small linearly reductive subgroup scheme of $\GL_{2,k}$
 in the indicated characteristics. 
\end{Theorem}

\begin{proof}
As already mentioned before, any of these subgroup schemes is linearly reductive 
and one can check that they are very small.

Conversely, assume that $G$ is a finite, very small and linearly reductive subgroup 
scheme of $\GL_{2,k}$. 
By Theorem \ref{thm: repsofGlr}, the associated abstract group $G_{\mathrm{abs}}$ is a finite and 
very small subgroup of $\GL_{2,\CC}(\CC)$. 
By \cite[Satz 2.9]{Brieskorn}, $G_{\mathrm{abs}}$ is the abstract group associated to one of the 
linearly reductive group schemes listed in Theorem \ref{thm: smallingl2}, 
where the condition on $p$ comes from the condition that $G_{\mathrm{abs}}$ admits a 
unique cyclic $p$-Sylow subgroup. 
One can check that the the embeddings of $G$ into $\GL_{2,k}$ 
given in Cases (1),...,(5) correspond to the embeddings of 
$G_{\mathrm{abs}}$ into $\GL_{2,\CC}(\CC)$ given by Brieskorn in \cite[Satz 2.9]{Brieskorn}. 
Thus, Claim (i) and (ii) follow from \cite[Satz 2.9]{Brieskorn}.
\end{proof}

\subsection{Very small linearly reductive subgroup schemes of $\GL_d$ in higher dimensions}
The simplest examples of finite and linearly reductive group schemes that 
admit very small representations are the following metacylic group schemes.

\begin{Definition}\label{def: metacyclic}
Let $k$ be an algebraically closed field $k$ of characteristic $p \geq 0$.
Let $m,n \geq 1$ and $0 < r \leq m$ be integers with $(n(r-1),m) = 1$ and $r^n = 1$ mod $m$. Let $e := \ord(r)$ be the multiplicative order of $r$ mod $m$ and assume $p \nmid e$. 
\begin{enumerate}
  \item The \emph{split metacyclic group scheme} with parameters $m,n,$ and $r$ is defined as $\bmu(m,n,r) := \bmu_m \rtimes \bmu_n$ such that $\bmu_n$ acts on $\bmu_m$ through an epimorphism $\bmu_n \to \langle r \rangle \subseteq \Aut(\bmu_m)$.
  \item We say that $\bmu(m,n,r)$ is \emph{very small} if and only if every prime divisor of $e$ divides $n/e$.
\end{enumerate}
\end{Definition}

\begin{Remark}
Note that the ``cyclic" linearly reductive group schemes $\bmu_n$ are examples of split metacyclic group schemes, since $\bmu_n \cong \bmu(1,n,1)$.
\end{Remark}

\begin{Remark} \label{rem: metacyclic}
The very small split metacyclic group schemes defined as above are precisely those linearly reductive group schemes whose associated abstract group is split metacyclic in the sense of \cite[Type I in Theorem 6.1.11]{Wolf} and 
that admits very small representations over $\CC$.
\end{Remark}

\begin{Theorem} \label{thm: verysmallgroupsinsmalldimension}
Let $G$ be a very small, finite, and linearly reductive subgroup scheme of $\GL_{d,k}$ over an algebraically closed field $k$ of characteristic $p \geq 0$.
 \begin{enumerate}
      \item If $d$ is odd or $p = 2$, then $G \cong \bmu(m,n,r)$ such that $G$ is very small and $\ord(r) \mid d$.
 
    \item \label{item: very small 3-dim} If $d = 3$, then $G$ is conjugate to one of the following subgroups:
    \begin{enumerate}
        \item \label{item: cyclic 3-dim} $G \cong \bmu_m$ embedded as   
    $$
\left\{
 \left(
 \begin{array}{ccc}
   a& 0 & 0\\ 0&a^{q_1} & 0 \\
   0 & 0 & a^{q_2}
 \end{array}\right)\,, \,a\in\bmu_{m}
 \right\} \,
$$
    for some $q_1,q_2 \geq 1$ with $(m,q_1) = (m,q_2) = 1$.        
    \item \label{item: metacyclic 3-dim}
    $p \neq 3$, $G \cong \bmu(m, 3^f N, r) \cong \bmu_m \rtimes (\C_{3^f} \times \bmu_N)$ with $\ord(r) = 3$, $3 \nmid N$, $f \geq 2$, $\bmu_m$ is embedded as in (a) with $q_1 = r$ and $q_2 = r^2$, $\bmu_{N}$ is embedded as in (a) with $q_1 = q_2 = 1$, and a generator of $\C_{3^f}$ is mapped to
        \[
 \begin{pmatrix}
   0 & 1 & 0 \\ 
   0 & 0 & 1 \\
   \zeta_{3^{f-1}} & 0 & 0
 \end{pmatrix},
\]
where $\zeta_{3^{f-1}}$ is a primitive $3^{f-1}$-th root of unity.
     \end{enumerate}

    \item If $d = p^i$ for some $i \geq 0$, then $G \cong \bmu_n$ for some $n \geq 1$.
 \end{enumerate}
\end{Theorem}

\begin{proof}
Let us first prove Claim (1). 
Since $G$ is a very small linearly reductive subgroup scheme of $\GL_{d,k}$, the associated abstract group $G_{\mathrm{abs}}$ is a very small subgroup of $\GL_{d,\CC}(\CC)$. 
First, assume that $d$ is odd. 
By \cite[Theorem 7.2.18]{Wolf}, the only very small subgroups of $\GL_{d,\CC}(\CC)$ 
for $d$ odd are the split metacyclic groups with $\ord(r) \mid d$.
Hence, $G$ is metacyclic and very small with $\ord(r) \mid d$. 
Next, assume that $p = 2$. 
Then, the $2$-Sylow subgroup of $G_{\mathrm{abs}}$ is cyclic. 
By \cite[Theorem 6.1.11, Theorem 7.2.18]{Wolf}, the only very small subgroups 
of $\GL_{d,\CC}(\CC)$ with a cyclic $2$-Sylow subgroup are the split metacyclic 
groups with $\ord(r) \mid d$.

In particular, in the situation of Claim (2), we know that $G \cong \bmu(m,n,r)$ 
is very small with $\ord(r) \mid 3$. 
If $\ord(r) = 1$, then $G$ is cyclic and the embedding $G \to \GL_{3,k}$ 
can be diagonalised as in Case (a). 
If $\ord(r) = 3$, then $3 \mid n$ and $(3,m) = 1$, 
so that $\bmu(m,n,r) = \bmu(m,3^fN,r)$ for some $f \geq 2$ and $3 \nmid N$. 
The embeddings of $G$ into $\GL_{3,k}$ given in Case (b) correspond to the embedding 
of $G_{\mathrm{abs}}$ into $\GL_{3,\CC}(\CC)$ given in \cite[Section 7.5]{Wolf}.

Finally, for Claim (3), note that $G \cong \bmu(m,n,r)$ with $\ord(r) \mid p^i$ by Claim (1).
Recall that $p \nmid \ord(r)$ by the definition of $\bmu(m,n,r)$. 
Hence, $\ord(r) = 1$, so that $(n(r-1),m) = 1$ implies that $m = 1$ and 
$\bmu(m,n,r) = \bmu(1,n,1) = \bmu_n$.
\end{proof}

Similarly, it is straightforward to determine which of the 
above representations factors through $\SL_{d,k}$, thereby 
giving a classification of finite, very small, and linearly reductive 
subgroup schemes of $\SL_{d,k}$. 
This is particularly interesting if $d = 3$, where it implies the following.

\begin{Corollary}
\label{cor: SL3}
Let $G$ be a very small, finite, and linearly reductive subgroup scheme of $\SL_{3,k}$ over an algebraically closed field $k$ of characteristic $p \geq 0$.
Then, $G$ is conjugate to $\bmu_m$ embedded as
    $$
\left\{
 \left(
 \begin{array}{ccc}
   a& 0 & 0\\ 0&a^{q_1} & 0 \\
   0 & 0 & a^{q_2}
 \end{array}\right)\,, \,a\in\bmu_{m}
 \right\} \,
$$
    for some $q_1,q_2 \geq 1$ with $(m,q_1) = (m,q_2) = 1$ and $q_1 + q_2 + 1 = 0$ mod $m$.
\end{Corollary}

In the case of very small subgroup schemes
of $\GL_{d,k}$, we can strengthen
Theorem \ref{thm: linearlyreductivesubgroupschemeGL}.

\begin{Theorem}
Let $k$ be an algebraically closed field of
  characteristic $p\geq7$.
  Assume that $p \nmid d$. 
  Then, lifting induces a canonical and bijective  map
 \begin{eqnarray*}
 &&
  \left\{
  \begin{array}{l}
  \mbox{ finite, linearly reductive, and very small}\\
  \mbox{ subgroup schemes of }\GL_{d,k}
  \mbox{ up to conjugation} 
  \end{array} \right\} 
  \\
  &\to&
    \left\{
    \begin{array}{l}
  \mbox{ finite and very small}\\
  \mbox{ subgroups of }\GL_{d}(\CC)
  \mbox{ up to conjugation}
  \end{array}
  \right\}  \,.
  \end{eqnarray*}
  This bijection maps subgroup schemes of $\SL_{d,k}$ to subgroups of $\SL_d(\CC)$.
\end{Theorem}

\begin{proof}
By Theorem \ref{thm: linearlyreductivesubgroupschemeGL}, it suffices 
to show that every finite and very small subgroup 
$G \subseteq \GL_d(\CC)$ has a unique abelian $p$-Sylow subgroup. 
By \cite[Theorem 6.1.11 and Theorem 6.3.1]{Wolf}, $G$ is an extension 
of a group of order $2^a 3^b 5^c$, for certain $a,b,c \geq 0$, 
by a meta-cyclic group $H$. 
Since normal $p$-Sylow subgroups are characteristic subgroups and 
since we assumed $p \geq 7$, it follows that $G$ admits a 
unique abelian $p$-Sylow subgroup if and only if $H$ does. 
Hence, we may assume without loss of generality that $G$ is metacyclic.

Since $G$ is very small, we have 
$G \cong \bmu(m,n,r)(\CC) = \bmu_m(\CC) \rtimes \bmu_n(\CC) \cong \C_m \rtimes \C_n$ 
(see Definition \ref{def: metacyclic} and Remark \ref{rem: metacyclic})
with $m,n \geq 1$, $0 < r \leq m$, $(n(r-1),m) = 1$, $r^n = 1$ mod $m$,
and $e = {\rm ord}(r)$ is the multiplicative order of $r$ mod $m$. 

Now, let $H \subseteq G$ be a $p$-Sylow subgroup and assume that 
$H$ is non-trivial. 
If $p \mid m$, then $H \subseteq \C_m$ and hence, $H$ is cyclic 
and normal in $G$, as desired. 
If $p \mid n$, then we can write $n = n'q$ with $(n',p) = 1$ and 
$q$ a power of $p$ such that $G$ is an extension
$$
1 \to \C_m \rtimes \C_q \to G \to \C_{n'} \to 1.
$$
By \cite[Section 7.2]{Wolf}, $G$ admits a very small representation of
dimension $d$ if and only if $e \mid d$. 
Thus, our assumption that $p \nmid d$ implies that $p \nmid e$. 
In turn, this implies that $\C_m \rtimes \C_q$ is in fact a direct product $\C_m \times \C_q$. 
The group $H$ is a $p$-Sylow subgroup in $\C_m \times \C_q$, 
hence abelian and characteristic, hence normal in $G$, as desired.
\end{proof}

Finally, we can classify finite, very small, and abstract groups that have 
``everywhere good reduction".

\begin{Corollary} \label{cor: everywheregoodreductiongroup}
Let $G$ be a finite abstract group. Assume that, for every $p > 0$ there exists an algebraically closed field $k$ of characteristic $p$ and a finite and linearly reductive group scheme $G_p$ over $k$ such that $G_p$ admits a very small representation and such that $G$ is the abstract group associated to $G_p$. 
Then, $G$ is cyclic.
\end{Corollary}

\begin{proof}
By Theorem \ref{thm: repsofGlr}, $G$ admits a unique cyclic $p$-Sylow subgroup for every $p > 0$. In particular, $G$ is the direct product over its cyclic $p$-Sylow subgroups. Hence, $G$ itself is cyclic. 
\end{proof}

\subsection{Connection to differential geometry}
\label{subsec: differentialgeometry}
Let $\rho:G\to\GL_{d,\CC}(\CC)$ be a very small representation of a finite group.
After fixing a $G$-invariant inner product, we may assume that $\rho$ is unitary.
Then, $\rho(G)$ acts on the $(2d-1)$-dimensional (dimension over $\RR$)
unit sphere $\Sph^{2d-1}\subset\CC^d$ by unitary transformations.
Since $\rho$ is very small, $G$ acts without fixed points on $\Sph^{2d-1}$
and thus, the quotient $\Sph^{2d-1}/\rho(G)$
is a $(2d-1)$-dimensional differential manifold that
admits a metric of constant positive curvature.
In fact, the classification of complete $n$-dimensional
Riemannian manifolds of constant positive curvature
(the so-called \emph{spherical space form problem})
is equivalent to the classification
of finite and very small subgroups of ${\rm O}(n)$,
see \cite[Chapter 5]{Wolf}.

\section{Generalities on F-singularities}
In this section, we recall some results about F-injectivity,
F-regularity, and F-signature of singularities, 
as well as tautness in dimension two.

\subsection{F-regularity and related properties}
Let $k$ be an algebraically closed field of positive characteristic $p > 0$ and let
$(R,\mathfrak{m})$ be a local, complete, and Noetherian $k$-algebra of dimension $d\geq2$.
Recall the definition of tight closure from \cite[Section 1]{TightClosure}.

\begin{enumerate}
\item $R$ is \emph{F-injective} if the action of Frobenius on 
the local cohomology groups $H^i_{\idealm}(R)$ is injective for all $i$.
 \item $R$ is \emph{weakly F-regular}  $\iff$ all ideals of $R$ are tightly closed.
 \item $R$ is \emph{strongly F-regular} $\iff$ for all $c \in R$ that is not in a minimal prime there exists $e > 0$ such that the map $R \to R^{1/p^e}, 1 \mapsto c^{1/p^e}$ splits.
\item For each $e\geq1$, we define $a_{p^e}$
to be the maximal rank of a free summand of $R$, considered as a module over itself 
via the $e$-fold Frobenius.
Then, the \emph{F-signature} of $R$ is defined to be
$$
 s(R) \,:=\, \lim_{e\to\infty} \frac{a_{p^e}}{p^{ne}}
$$
and we refer to \cite{Tucker} for the existence of this limit.
\item The \emph{Hilbert--Kunz multiplicity} of $(R,\idealm)$ is defined
to be
$$
 e_{\rm HK}(R) \,:=\, e_{\rm HK}(\idealm,R) \,:=\,
 \lim_{e\to\infty}\,\frac{{\rm length}(R/\idealm^{[p^e]})}{p^{ed}}
$$
and we refer to \cite{Huneke} for the existence of this limit.
\end{enumerate}

\begin{Remark}
Strong F-regularity implies weak F-regularity by \cite[Theorem 3.1 (d)]{HochsterHunekeFregular}  
and they are conjectured to coincide in general.
We will see in Proposition \ref{prop: lrq} below 
that they are both satisfied for lrq singularities, which is why we will
simply write F-regular instead of strongly F-regular.
\end{Remark}

We recall the following implications between the above notions
$$
s(R) > 0 \,\Leftrightarrow\, R \text{ is F-regular}  \,\Rightarrow\, R \text{ is F-injective}.
$$

\subsection{F-regular surface singularities}
Next, we recall the following theorem of Hara \cite[Theorem (1.1)]{Hara} 
that gives a classification of F-regular singularities in dimension 2.
Here, the notion \emph{star-shaped of type $(a,b,c)$}
for a graph means that it is star shaped, 
that there are three components once the central vertex is removed, and that these
components span lattices of discriminant $a$, $b$, and $c$, respectively.

\begin{Theorem}\label{thm: Fregularclassification}
Let $(R,\mathfrak{m})$ be a local, complete, Noetherian $k$-algebra of dimension $2$ with $R/\mathfrak{m} = k$. Then, the following are equivalent:
\begin{enumerate}
\item $R$ is F-regular,
\item $R$ has rational singularities and the graph $\Gamma$ of the minimal resolution is one of the following:
\begin{enumerate}
\item $\Gamma$ is a chain.
\item $\Gamma$ is star shaped of type $(2,2,d)$, $d \geq 2$, and $p \neq 2$,
\item $\Gamma$ is star shaped of type $(2,3,3)$ or $(2,3,4)$ and $p \neq 2,3$,
\item $\Gamma$ is star shaped of type $(2,3,5)$ and $p \neq 2,3,5$.
\end{enumerate}
\end{enumerate}
In particular, if $R$ is F-regular, then $R$ is normal and klt, 
and the converse holds if $p \geq 7$.
\end{Theorem}

Moreover, F-regular surface singularities are taut by work of Tanaka \cite[Theorem 1.3]{Tanaka}, 
that is, their formal isomorphism class is uniquely determined by the exceptional locus of 
their minimal resolution.

\begin{Theorem}\label{thm: Fregulartaut}
Let $(R,\mathfrak{m})$ be a local, complete, and Noetherian F-regular $k$-algebra 
of dimension $2$ with $R/\mathfrak{m} = k$. 
Then, $R$ is taut.
\end{Theorem}

\section{Local fundamental groups and class groups}

In this section, we recall a couple of results concerning the local \'etale fundamental group
and the class group of a singularity.

Let $k$ be an algebraically closed field of characteristic $p\geq0$ and let
$(R,\mathfrak{m})$ be a local, complete, and Noetherian $k$-algebra of dimension $d\geq2$.
We denote the closed point of $X:=\Spec R$ corresponding to $\mathfrak m$ by $x$ and 
set $U := X \setminus \{x\}$. 
We will simply write $x \in X$ in this situation.

Let $G$ be a finite group scheme over $k$. 
Recall that isomorphism classes of $G$-torsors over $X$ (resp. $U$) are in bijection with 
the flat cohomology $\Hfl{1}(X,G)$ (resp. $\Hfl{1}(U,G)$).
In general, this is a pointed set and if $G$ is abelian, then $\Hfl{1}(-,G)$ 
is an abelian group.
Next, the canonical inclusion $\imath:U\to X$ induces a pullback map
\begin{equation}
 \label{eq: imath}
   \imath^* \,:\, \Hfl{1}(X,G) \,\to\, \Hfl{1}(U,G)
\end{equation}
of pointed sets and of abelian groups if $G$ is abelian.
We will now describe the structure of the cokernel of $\imath^*$ 
in some cases.

\subsection{The local \'etale fundamental group}
\label{subsec: etale}
Let $G$ be a finite and \'etale group scheme over $k$. 
Since $k$ is assumed to be algebraically closed, $G$ is the constant group scheme associated 
to a finite group. 
By Hensel's lemma \cite[Proposition 5]{Nagata50} and since $k$ is algebraically closed, 
$X$ admits no non-trivial finite \'etale covers by \cite[Proposition 18.8.1]{EGA4}. 
In particular, $\Hfl{1}(X,G)$ is trivial.

By \cite{SGA1}, there exist \emph{Galois categories} classifying torsors
under \emph{all} finite and \'etale group schemes over $k$ over $X$ and $U$,
which leads to the \emph{\'etale fundamental groups} $\piet(X,x)$ and $\piet(U,u)$.
Here, $x\in X$ and $u\in U$ are base-points that have to be chosen
to begin with.
For example, if we choose the geometric generic points
$\overline{\eta}$ of $U$ and $X$ as base-points, then the
inclusion $\imath:U\to X$ induces a map
$$
 \imath_*\,:\,\piet(U,\overline{\eta}) \,\to\, \piet(X,\overline{\eta}),
$$
which is a continuous homomorphism of profinite groups.
Since there are no non-trivial torsors under finite and \'etale group schemes
over $X$, we have $\piet(X,\overline{\eta})=\{e\}$, that is, $X$ is algebraically simply connected.
The group 
$$
  \pietloc(X) \,:=\, \piet(U,\overline{\eta})
$$
is called the \emph{local (\'etale) fundamental group} of $X$. 
Then, the following is well-known and follows immediately from the definition of $\pietloc(X)$.

\begin{Proposition}\label{prop: torsfinet}
If $G$ is finite and \'etale, then there is a canonical 
bijection
     $$
     \Hfl{1}(U,G)/\Hfl{1}(X,G) \,=\, \Hfl{1}(U,G)
      \,\leftrightarrow\, 
      \Hom(\pietloc(X),G)/\sim,
     $$
where the right hand side denotes homomorphisms of profinite groups
modulo inner automorphisms of $G$.
\end{Proposition}

\begin{proof}
See \cite[Expos\'e V, Remarque 5.11]{SGA1} 
and \cite[Expos\'e XI, \S 5]{SGA1}.
\end{proof}

\subsection{The class group}
Now, let $G$ be a finite and abelian group scheme over $k$, such that $G^D$ is \'etale.
By Boutot \cite[III Lemme 4.2., III Corollaire 4.9.]{Boutot}, the class group of 
$X$ parametrises equivalence classes of $G$-torsors over $U$ in the following sense.

\begin{Proposition} \label{prop: torsfinabcl}
If $G$ is finite and abelian and $G^D$ is \'etale, then there is a canonical 
bijection
     $$
     \Hfl{1}(U,G)/\Hfl{1}(X,G)
      \,\leftrightarrow\, 
      \Hom(G^D,{\rm Cl}(X)),
     $$
where the right hand side denotes morphisms of abelian groups.
\end{Proposition}

\begin{Example}
If in Proposition \ref{prop: torsfinabcl}, we assume additionally that $G$ is \'etale, 
then $\Hfl{1}(U,G)/\Hfl{1}(X,G) \,=\, \Hfl{1}(U, G)$ by  Proposition \ref{prop: torsfinet}
and thus, $\Hom(G^D,{\rm Cl}(X))$ parametrises $G$-torsors over $U$. 
However, this is not true in general: for example, if $G = \bmu_{p^n}$, 
then $\Hfl{1}(X,G) = R^\times/(R^\times)^{p^n}$, so $\Hfl{1}(U,G)$ is usually much 
bigger than $\Hom(G^D,{\rm Cl}(X))$ in positive characteristics.
\end{Example}

\section{Quotient singularities and linearisation}
\label{sec: quotient singularities}

In this section, we introduce quotient singularities by
finite group schemes, that is, quotients by a finite group scheme $G$ with respect to actions on the spectrum of
formal a power series ring $k[[u_1,...,u_d]]$, 
such that the $G$-action
is free outside the closed point.
The main result is that such $G$-actions are linearisable
if and only if $G$ is linearly reductive, which leads
to the notion of a linearly reductive quotient singularity
(lrq singularity for short).

\subsection{Quotient singularities}
Recall that if $G$ is a finite group scheme over an algebraically closed field $k$,
if $Y=\Spec S$ is an affine scheme over $k$, and if we have an \emph{action}
$G\times_k Y\to Y$ over $k$, then the \emph{geometric quotient} $Y/G$ exists, and it
is isomorphic to $\Spec S^G$, where $S^G\subseteq S$ denotes the ring 
of invariants, see \cite[Theorem 12.1]{MumfordAV}.
The $G$-action on $Y$ is said to be \emph{free}, if the quotient morphism $\pi:Y\to Y/G$
is a $G$-torsor over $Y/G$. 
Recall that a ($k$-valued) \emph{fixed point} of the $G$-action is a point in $Y$, 
whose stabiliser is all of $G$.

\begin{Definition} \label{def: verysmallaction}
Let $G$ be a finite group scheme over $k$. 
A faithful action of $G$ on $k[[u_1,...,u_d]]$ with $d \geq 2$ is called \emph{very small} 
if it is free outside the closed point and the closed point is a fixed point of the action.
\end{Definition}

We will see in Proposition \ref{prop: lambda} that this notion is compatible with the notion of very small representations in the case where $G$ is linearly reductive.

\begin{Definition}
 \label{def: quotient singularity}
 A \emph{quotient singularity} $x\in X$ by a finite group scheme $G$ over $k$ is a pair $(X,x)$,
 where $X = \Spec~R$ is the spectrum of a local $k$-algebra of dimension $d \geq 2$ with closed point $x$ such that
 $$
   R \,\cong\, k[[u_1,...,u_d]]^G,
 $$
 where $G$ acts on $k[[u_1,...,u_d]]$ via a very small action.
\end{Definition}

\begin{Remark}
 If $k$ is the field $\CC$ of complex numbers, 
 then there exists a natural equivalence of categories between 
 finite group schemes over $\CC$, finite and linearly reductive
 group schemes over $\CC$, and finite groups,
 see Section \ref{subsec: generalitieslinred}.
 Actions of finite groups on $\Spec \CC[[u_1,...,u_d]]$
 can be linearised, see Section \ref{subsec: linearization} below.
 In positive characteristic, these equivalences and implications
 are no longer true and this leads to 
  different natural possibilities to extend
 the notion of a quotient singularity to positive characteristic.
 Definition \ref{def: quotient singularity} is fairly general,
 since we allow not necessarily linear actions of finite group schemes,
 but see also the discussion in \cite[Section 9.7]{RDP}.
\end{Remark}

\subsection{Linearisation}
\label{subsec: linearization}
If the action of a group scheme $G$ over $k$ on $k[[u_1,...,u_d]]$ factors over 
the obvious $\GL_d$-action via a homomorphism $\rho:G\to\GL_d$ of group schemes over $k$, 
then the $G$-action on $k[[u_1,...,u_d]]$ is called \emph{linear}.
If the $G$-action becomes linear after some change of coordinates,
it is called \emph{linearisable}.

Over the complex numbers, all finite group actions as in Definition
\ref{def: quotient singularity} are linearisable by a lemma of Cartan \cite[Lemma 1]{Cartan}.
Using the correspondence between $\bmu_p$-actions and multiplicative 
vector fields in characteristic $p>0$, linearisation of $\bmu_p$-actions
as in Definition \ref{def: quotient singularity} follows from a theorem
of Rudakov and Shafarevich \cite[Theorem 2]{RudSha}.
On the other hand, the action of a finite group as in Definition
\ref{def: quotient singularity}, where the characteristic $p$ divides
the order of the group cannot be linearised \cite[Proposition 2.1]{Peskin}.
The following result puts these (non-)linearisation results
into perspective.

\begin{Proposition}
   \label{prop: linearizable}
   In the situation of Definition \ref{def: quotient singularity},
   the $G$-action on $Y = \Spec k[[u_1,\hdots,u_d]]$ is linearisable 
   if and only if $G$ is linearly reductive.
\end{Proposition}

\begin{proof}
A proof of the fact that linearly reductive group scheme actions can be 
linearised can be found, for example in  \cite[proof of Corollary 1.8]{Satriano}.

For the converse, assume that $G$ is not linearly reductive. 
Seeking a contradiction, we assume that the $G$-action on $Y$ can be linearised, 
and thus, $G \subseteq \GL_d$ compatible with the action of the latter on $Y$.
Then, we have $\C_p \subseteq G \subseteq \GL_d$ or $\balpha_p \subseteq G \subseteq \GL_d$ 
by Lemma \ref{lem: lrcriterion}. 
Since both of these group schemes are unipotent, they are conjugate in $\GL_d$ 
to subgroup schemes of the group scheme of upper triangular matrices 
with ones on the diagonal. 
In particular, their fixed loci are not isolated, because they contain $V(u_2,\hdots,u_d)$, and therefore, the $G$-action on $Y$ 
is not free outside the closed point, a contradiction.
\end{proof}

\begin{Definition} \label{def: lrq}
  A \emph{linearly reductive quotient singularity}, or, \emph{lrq singularity} for short,
  is a quotient singularity by a finite and linearly reductive group scheme.
\end{Definition}

The following proposition bridges the gap between very small actions in the sense of 
Definition \ref{def: verysmallaction} and very small representations in the sense of 
Definition \ref{def: verysmallrepresentation} for finite and linearly 
reductive group schemes.

\begin{Proposition}
\label{prop: lambda}
Let $k$ be an algebraically closed field and let $\rho: G \to \GL_{d,k}$ be a representation of a finite and linearly reductive group scheme $G$ over $k$. Then, $\lambda(\rho)$ coincides with the dimension of the non-free locus of the induced $G$-action $\rho'$ on $\Spec k[[u_1,\hdots,u_d]]$.
\end{Proposition}

\begin{proof}
Let $m$ be the dimension of the non-free locus of $\rho'$ and let $V$ 
be a $k$-vector space of dimension $d$ so that 
$({\rm Sym}(V))^\wedge =  k[[u_1,\hdots,u_d]]$. 
We have $\lambda(\rho) \leq m$, since 
$\Spec ({\rm Sym}(V^{\bmu_n}))^\wedge \subseteq \Spec ({\rm Sym}(V))^\wedge$ is contained 
in the non-free locus of $\rho'$ for every $\{\id\} \neq \bmu_n \subseteq G$.
Conversely, let $\eta$ be the generic point of an irreducible component of the non-free locus 
of $\rho'$ such that $\eta$ has height $d-m$. 
By our choice of $\eta$, the stabiliser ${\rm Stab}_{\eta}$ is non-trivial. 
Since $G_{\eta}$ is linearly reductive, the inclusion 
${\rm Stab}_{\eta} \subseteq G_{\eta}$ descends to $k$ and we find some 
$\bmu_n \subseteq G$ that fixes $\eta$. 
Thus, we have 
$\eta \subseteq \Spec ({\rm Sym}(V^{\bmu_n}))^\wedge \subseteq \Spec ({\rm Sym}(V))^\wedge$, 
so $m \leq \lambda(\rho)$, which is what we wanted to show.
\end{proof}

\begin{Corollary} \label{cor: verysmallactionvsverysmallrep}
Let $k$ be an algebraically closed field and let $G$ be a finite and linearly 
reductive group scheme over $k$. 
An action of $G$ on $\Spec~k[[u_1,\hdots,u_d]]$ is very small if and only 
if, after a change of coordinates, it coincides with an 
action induced by a very small linear representation of $G$.
\end{Corollary}

\subsection{Very small actions of connected group schemes}
To understand torsors over the punctured spectrum of lrq singularities, we also need some information on very small actions of not necessarily linearly reductive group schemes.
It turns out that the class of connected group schemes admitting very small actions is rather restricted. The key observation is the following.

\begin{Lemma} \label{lem: cyclic}
If $G$ is a non-trivial, connected, and finite group scheme of height $1$ 
that admits a very small action, then ${\rm length}(G) = p$. 
\end{Lemma}

\begin{proof}
Seeking a contradiction, assume that ${\rm length}(G) > p$. We write $G = \Spec~T$ with $T = k[t_1, \dots, t_l] / (t_1^p, \dots, t_l^p)$, where $l \geq 2$ and $S := k[[u_1, \dots, u_d]]$. 
Write the very small action as  $a\colon S\to T \otimes_k S$, $b\mapsto a(b)$ with
$$
a(b) \,\equiv\, 1 \otimes b + \sum_{h=1}^{l} t_h \otimes D_h(b)  \pmod{(t_1, \dots, t_l)^2}.
$$
Then, the $D_h$ are (not necessarily commuting) derivations on $S$.
Now let $C$ be the vanishing locus of the ideal $I_C \subseteq S$ generated by the $2 \times 2$ minors of $M$ (so that $\rank M \leq 1$ on $C$), where $M$ is the $d \times 2$ matrix 
\[
 M = 
 \begin{pmatrix}
  D_1(u_1) & D_2(u_1) \\
  \vdots   & \vdots \\
  D_1(u_d) & D_2(u_d)
 \end{pmatrix}.
\]
Note that all the $D_i(u_j)$ must vanish at $0$, since $0 \in  \widehat{\mathbb{A}}_k^d$ 
is a fixed point of the $G$-action, so $C$ is non-empty. 
Therefore, by the formula for the expected dimension of degeneracy loci 
(see, for example, \cite[Section 14]{Fulton}), 
we have $\dim(C) \geq d - (d-1)(2-1) = 1$. 
We claim that this implies that the action of $G$ on $U$ is not free, where $U$ 
is the complement of the origin.

The map $(a, i_2) : S \otimes_k S \to T \otimes_k S$ induces a map 
$(a, i_2)|_U : \OO_U \otimes_k \OO_U \to T \otimes_k \OO_U$.
Then, for every local section $\beta \in \OO_U \otimes \OO_U$ 
we have $c_1(\beta) D_2(u_j) - c_2(\beta) D_1(u_j) \in I_C$ for all $j = 1, \dots,d$, 
where $c_h(\beta)$ is the coefficient of $t_h$ in $(a, i_2)(\beta)$ and $I_C$ 
is the ideal defining $C$. 
Indeed, we can reduce to the case $\beta = b_1 \otimes b_2 = u_k \otimes 1$, and 
this case is clear.
Since $C$ has positive dimension, the ideal $I_C|_U$ is non-trivial. 
If the ideal $I_1 \subset \OO_U$ generated by the image of $D_1$ is equal to $\OO_U$, then it follows 
that $(a, i_2)|_U$ is not surjective. 
If $I_1 \subsetneq \OO_U$, then again it follows that $(a, i_2) |_U$ is not surjective.
Hence, the $G$-action on $U$ is not free, contradicting our assumption.
\end{proof}

\begin{Proposition} \label{prop: linearlyreductive or unipotent}
 Let $G$ be a finite group scheme admitting a very small action.
 Then the scheme underlying $G^{\circ}$ is isomorphic to $\Spec k[t]/(t^{p^n})$ 
 for some $n \geq 0$.
 In particular, if $n\geq1$, then
 \begin{enumerate}
     \item either $G^{\circ}$ is linearly reductive of height $n$, hence isomorphic to $\bmu_{p^n}$,
     \item or else $G^{\circ}$ is unipotent of height $n$, hence $G^{\circ}[F^{i+1}]/G^{\circ}[F^{i}] \cong \balpha_p$ for all $0 \leq i < n$.
 \end{enumerate}
\end{Proposition}

\begin{proof}
By Lemma \ref{lem: cyclic} and our definition of very small actions, 
the subgroup scheme $G^{\circ}[F] \subseteq G^\circ$ of height $1$ is of length 
$\leq p$.
From this, the first assertion follows. 
Note that the height of $G^{\circ} = \Spec k[t]/(t^{p^n})$ is $n$ and that, if 
$n \geq 1$, we have $G^{\circ}[F] \in \{\balpha_p,\bmu_p\}$, since these are the 
only connected group schemes of length $p$ over $k$.

If $G^{\circ}[F] \neq \balpha_p$, then $G^\circ$ is linearly reductive by 
Lemma \ref{lem: lrcriterion} and thus, $G^{\circ} \cong \bmu_{p^n}$ 
by Nagata's classification. 
If $G^{\circ}[F] \neq \bmu_p$, then $G^\circ$ is unipotent by 
\cite[Th\'eor\`eme 4.6.1. iv), Exp. XVII]{SGA32}. 
Since $G^\circ$ has height $n$, the subquotient $G^{\circ}[F^{i+1}]/G^{\circ}[F^{i}]$ 
is unipotent of length $p$ for every $i < n$, hence 
$G^{\circ}[F^{i+1}]/G^{\circ}[F^{i}] \cong \balpha_p$. 
\end{proof}

\begin{Remark}
Both cases occur: the two-dimensional rational double points
of type $A_{p^n-1}$ are lrq singularities by $G=G^\circ=\bmu_{p^n}$.
For examples with $G=G^\circ=\balpha_p$, we refer to Section
\ref{subsec: deformation non-lrq}, and examples with
$G=G^\circ=\bM_2$ (a non-split extension of $\balpha_p$ by $\balpha_p$)
can be found in \cite{RDP}.
\end{Remark}

The local Picard scheme and local fundamental group scheme of arbitrary quotient singularities 
can be very complicated and we refer to \cite{RDP} for examples.
However, in the case of linearisable actions, the situation becomes simpler, as we will see
in the next section.

\section{Properties of lrq singularities}
\label{sec: properties lrq}

In this section, we will compute some basic
invariants of lrq singularities and contrast them with properties of more general quotient singularities in Section \ref{subsec: CM} and Section \ref{subsec: more general}.

\subsection{Basic properties of lrq singularities}
Given a $d$-dimensional (very small) representation
$\rho:G\to \GL_{d,k}$ of a finite $k$-group scheme $G$, 
we obtain an induced $G$-action on the symmetric algebra ${\rm Sym}^*(k^d)$, 
and passing to spectra on $d$-dimensional affine space $\mathbb{A}^d$.
By Proposition \ref{prop: linearizable}, every lrq singularity is formally isomorphic 
to some $({\mathbb{A}^d/G})^\wedge$.

The following proposition is a collection of well-known results, 
or results that are easily deduced from known results.

\begin{Proposition} \label{prop: lrq}
 Let $k$ be an algebraically closed field of characteristic $p>0$,
 let $G$ be a finite and linearly reductive group scheme over $k$, and
 let $\rho:G\to{\rm GL}_{d,k}$ be a very small representation of $G$.
 Let 
 $$
  x \in X \,:=\, 0 \in ({\mathbb{A}_k^d/G})^\wedge
 $$
 be the associated lrq singularity.
 \begin{enumerate}
  \item \label{item: F-regular Q-Gorenstein} 
  The singularity is F-regular and $\QQ$-Gorenstein. 
  In particular, it is Cohen--Macaulay, normal and log terminal.
  \item \label{item: class group} The class group ${\rm Cl}(X)$ is finite and there exists an isomorphism of abelian groups
  $$
  {\rm Cl}(X) \,\cong\, (G^{{\rm ab}})^D,
  $$
  where $G^{{\rm ab}}$ is the abelianisation of $G$.
  In particular, the $p$-primary part of the class group ${\rm Cl}(X)$ is cyclic of order $p^m \leq |G^\circ|$.
  \item \label{item: F-signature} The F-signature $s(X)$ is finite and satisfies
  $$
   s(X) \,=\, \frac{1}{|G|}.
  $$
  \item \label{item: Hilbert-Kunz multiplicity} The Hilbert--Kunz multiplicity $e_{\rm HK}(X)$ satisfies
  $$
    e_{\rm HK}(X)\,=\, \frac{1}{|G|}\, {\rm length}(S/\idealm_R S),
  $$
  where $(R,\idealm_R)$ denotes the local ring of $x\in X$
  and $S=k[[u_1,...,u_d]]$ is the local ring of $(\mathbb{A}^d)^\wedge$.
  In particular, it is a rational number, whose denominator
  divides the length of $G$.
  \item \label{item: pietloc} The \'etale local fundamental group satisfies
  $$
  \pietloc(X) \,\cong\, G^\et.
  $$
 \end{enumerate}
\end{Proposition}

\begin{proof}
Let $S=k[[u_1,...,u_d]]$ and $R:=S^G\subseteq S$.
Since $G$ is linearly reductive and since the action on $S$ is linear, 
the inclusion $R\subseteq S$ is split, that is, $R$ is a direct summand, 
see also \cite[Section 6.5]{BrunsHerzog}.
Being a direct summand of an F-regular ring, $R$ is F-regular, 
see \cite[Theorem 3.1(e)]{HochsterHunekeFregular}. 

The assertion that ${\rm Cl}(X) \,\cong\, (G^{{\rm ab}})^D$ is shown in 
\cite[Theorem 3.9.2]{Benson}. 
In fact, it is stated there for finite groups rather than group schemes, but the 
proof also works in our case. 
By Proposition \ref{prop: linearlyreductive or unipotent}, we have 
$G^\circ \cong \bmu_{p^n}$ for some $n$ and since $G^\circ$ is normal in $G$, 
the group scheme $\C_{p^n} = \bmu_{p^n}^D$ surjects onto the $p$-primary part 
of ${\rm Cl}(X)$. 
In particular, the latter is cyclic of order $p^m \leq |G^\circ|$.

Since the class group is finite, the class of a canonical Weil divisor has finite order 
in ${\rm Cl}(X)$, that is, $X$ is $\QQ$-Gorenstein.

The assertion on the F-signature follows from \cite[Theorem B]{Carvajal-RojasSchwedeTucker} and \cite[Theorem C]{Carvajal-Rojas}.

Assertion (\ref{item: Hilbert-Kunz multiplicity}) in the case of dimension 2 and if $G$ is \'etale
is \cite[Theorem 5.4]{WYHK}. 
However, as explained in \cite[Example 18 in Section 3]{Huneke}, this formula also
holds if $G$ is finite and \'etale and using these arguments,
one can check that the formula also holds in our case.

Finally, consider Assertion (\ref{item: pietloc}). The \'etale fundamental group of $Y = \Spec k[[u_1,\hdots,u_d]]$ is trivial.
Since $G^\circ$ is connected, the quotient map $Y \to Y/G^\circ$
is purely inseparable, which implies that the local fundamental group
of $Y/G^\circ$ is trivial, see also \cite[Expos\'e IX, Th\'eor\`eme 4.10]{SGA1}.
Since $G^\et$ is \'etale and acts freely on the pointed space $Y/G^\circ-\{0\}$,
which is simply connected (since the dimension is $\geq 2$),
where $\{0\}$ denotes the closed point, the local fundamental group
of $Y/G$ is isomorphic to $G^\et$.
\end{proof}

\begin{Remark} \label{rmk: Hashimoto}
Note that, in the setting of Proposition \ref{prop: lrq}, 
we have $(G^{\rm ab})^D = \chi(G)$, where $\chi(G)$ is the group of characters of $G$.
By \cite[Lemma 3.13, Example 4.15]{HashimotoEquivariantI}, the
identification ${\rm Cl}(X) \cong \chi(G)$ extends to small actions of
smooth finite type group schemes $G$.
\end{Remark}

As an application of Proposition \ref{prop: linearlyreductive or unipotent} and
Proposition \ref{prop: lrq}, we deduce that lrq singularities are precisely the F-regular quotient singularities in positive characteristic.

\begin{Proposition} \label{prop: lrq=Fregular}
 Let $x\in X=\Spec R$ be a quotient singularity in characteristic $p \geq 0$. Then, the following are equivalent:
 \begin{enumerate}
     \item \label{item:lrq} $x \in X$ is an lrq singularity.
     \item \label{item:Finjective} Either $p = 0$ or $p > 0$, $p \nmid |\pietloc(X)|$, and the integral closure $\widetilde{x} \in \widetilde{X}$ of the universal \'etale cover of $X-x$ is F-injective.
     \item \label{item:Fregular} Either $p = 0$ or $p > 0$ and $X$ is F-regular.
 \end{enumerate}
\end{Proposition}
\begin{proof}
Since every quotient singularity is lrq if $p = 0$, we may assume that $p > 0$.

The implication (\ref{item:lrq}) $\Rightarrow $ (\ref{item:Fregular}) is 
Proposition \ref{prop: lrq} (\ref{item: F-regular Q-Gorenstein}).

Next, for (\ref{item:Fregular}) $\Rightarrow $ (\ref{item:Finjective}), note that 
$p \nmid |\pietloc(X)|$ by \cite[Theorem A]{Carvajal-RojasSchwedeTucker}. 
The F-signature of $\widetilde{X}$ is positive by 
\cite[Theorem B]{Carvajal-RojasSchwedeTucker}, so $\widetilde{X}$ is F-regular and, 
in particular, $\widetilde{X}$ is F-injective.

Finally, to see (\ref{item:Finjective}) $\Rightarrow $ (\ref{item:lrq}), 
let $G$ be a finite group scheme, let $S := k[[u_1,\hdots,u_d]]$, and assume that 
$X$ is a quotient singularity via a very small action of $G$ on $\Spec S$. 
Note that $\Spec S \to \Spec S^{G^\circ}$ is a homeomorphism in the \'etale
topology, so $\Spec S^{G^\circ} \cong \widetilde{X}$ and $G^{\et} \cong \pietloc(X)$. 
In particular, we have $p \nmid |G^\et|$, so $G^\et$ is linearly reductive. 
By Proposition \ref{prop: linearlyreductive or unipotent}, we know that $G^\circ$ 
is either linearly reductive or unipotent. 
If $G^\circ$ is unipotent, then it admits a quotient isomorphic to $\balpha_p$ 
and the associated $\balpha_p$-torsor over $\widetilde{X}-\widetilde{x}$ 
induces a non-trivial class in $H^1(\widetilde{X}-\widetilde{x},\OO_{\widetilde{X}-\widetilde{x}}) = H^2_{\{\widetilde{x}\}}(\widetilde{X},\OO_{\widetilde{X}})$,
which is killed by Frobenius. 
By the definition of F-injectivity, this is impossible, 
so $G^\circ$ is linearly reductive, hence so is $G = G^\circ \rtimes G^\et$. 
\end{proof}

\subsection{Relation to toric singularities}
Toric singularities that are $\mathbb{Q}$-factorial are examples of lrq singularities and we have the following relation 
between these classes of singularities.

\begin{Proposition} \label{prop: toric}
Let $x\in X$ be an isolated singularity of dimension $d$ over an algebraically closed
field $k$.
Then, the following are equivalent:
\begin{enumerate}
    \item $x\in X$ is toric and $\mathbb{Q}$-factorial,
    \item $x\in X$ is an lrq singularity by an abelian group scheme,
    \item $x\in X$ is an lrq singularity by $\bmu_n$ for some $n > 0$,
    \item $x\in X$ is an lrq singularity by $\bmu_n = {\rm Cl}(X)^D$, 
    where $\bmu_n$ is embedded into $\GL_d$ as $\zeta \mapsto {\rm diag}(\zeta^{q_1},\hdots,\zeta^{q_d})$ 
    with $1 = q_1 \leq \hdots \leq q_d < n$.
\end{enumerate}
Moreover, the $q_i$ in $(4)$ are uniquely determined by $X$ up to automorphisms of $\bmu_n$
and every linear action of $\bmu_n = {\rm Cl}(X)^D$ on $Y = \Spec k[[u_1,\hdots,u_d]]$ 
with quotient $X$ is conjugate in $\GL_d$ to an action as in $(4)$.
\end{Proposition}

\begin{proof}
Clearly, we have $(4) \Rightarrow (3) \Rightarrow (2)$. 
Conversely, the implications $(2) \Rightarrow (3) \Rightarrow (4)$ follow from 
Theorem \ref{thm: repsofGlr} and the representation theory of cyclic groups 
over the complex numbers.

Let us prove $(4) \Rightarrow (1)$. Since $\bmu_n$ acts diagonally, the diagonal torus action descends to $X$,
showing that $X$ is toric. Moreover, $x \in X$ is $\mathbb{Q}$-factorial by Proposition \ref{prop: lrq} (\ref{item: F-regular Q-Gorenstein}).

As for $(1) \Rightarrow (2)$, assume that $X = (\Spec k[M])^\wedge$ 
for some affine semigroup $M$.
Since the claim is clear for smooth $X$, we may assume that $X$ is singular. 
Since an affine toric variety has trivial Picard group and $x$ is the only singular point of $\Spec k[M]$, the natural map 
${\rm Cl}(k[M]) \to {\rm Cl}(\mathcal{O}_{\Spec k[M],x})$ 
is injective. 
Moreover, by \cite[Tag 0CDY]{StacksProject}, the map ${\rm Cl}(\mathcal{O}_{\Spec k[M],x}) \to {\rm Cl}(X)$ 
is also injective. 
Thus, the assumption that $X$ is $\mathbb{Q}$-factorial guarantees that ${\rm Cl}(k[M])$ is finite.
Since $X$ is singular, the toric variety $\Spec k[M]$ has no torus factors, so the Cox construction
(see, for example, \cite[Section 3.1]{GeraschenkoSatriano}) realises $\Spec k[M]$ as a 
quotient of $\mathbb{A}^d$ by a diagonal action of the finite abelian group scheme 
$G = {\rm Cl}(k[M])^D$ 
with fixed locus of codimension at least $2$. 
In fact, by the Chevalley--Shephard--Todd Theorem for 
linearly reductive group schemes \cite{Satriano}, the $G$-action on $\mathbb{A}^d$ 
is very small, since $x \in X$ is an isolated singularity.
Hence, $X$ is an lrq singularity by the finite abelian group 
scheme $G$.

Finally, note that a realisation of $X$ as a $\bmu_n$-quotient as in (4) yields an 
identification $X \cong (\Spec k[M])^\wedge$, where the embedding of $\bmu_n$ into the 
diagonal torus $\GG_m^d$ defines the affine semigroup
$$
M \,:=\, \Ker\left(\ZZ^d = (\GG_m^d)^D(k) \to (\bmu_n)^D(k) = \C_n\right) \cap \NN^d.
$$
If $M' \subseteq \NN^d$ is another affine semigroup such that $X \cong (\Spec k[M'])^\wedge$ 
via the above construction, then there is an isomorphism $\varphi: M \cong M'$ of submonoids 
of $\NN^d$ by Demushkin's Theorem (see, for example, \cite[Theorem 4.6.1]{Wl}). 
It is elementary to check that every such $\varphi$ is induced by a permutation of 
the generators of $\NN^d$. 
Hence, the two embeddings $\bmu_n \to \GG_m^d$ corresponding to $M$ and $M'$ coincide 
up to permutation of the coordinates in $Y$ and up to automorphisms of $\bmu_n$. 
\end{proof}

\begin{Remark}
It is well-known that an affine toric variety $Y$ associated to a 
cone $\sigma$ is $\mathbb{Q}$-factorial if and only if $\sigma$ is simplicial \cite[Lemma 14-1-1.]{Matsuki}. 
A similar statement holds if $x \in X$ is the toric singularity 
obtained by completing an affine toric variety $Y$ with a single
isolated singularity at its singular point.

Indeed, let $\sigma$ be the cone corresponding to $Y$. 
If $\sigma$ is simplicial, then ${\rm Cl}(Y)$ is finite. 
As explained in the proof of Proposition \ref{prop: toric}, the Cox construction 
realises $X$ as an lrq singularity by ${\rm Cl}(Y)$, 
so ${\rm Cl}(X) = {\rm Cl}(Y)$ by Proposition \ref{prop: lrq}. 
In particular, $X$ is $\mathbb{Q}$-factorial. 
If $\sigma$ is not simplicial, then ${\rm Cl}(Y)$ is infinite. 
As in the proof of Proposition \ref{prop: toric}, the natural map 
${\rm Cl}(Y) \to {\rm Cl}(X)$ is injective, 
so ${\rm Cl}(X)$ is also infinite, that is, $X$ is not
$\mathbb{Q}$-factorial.
\end{Remark}

\begin{Remark}
\label{rem: HK}
 By Proposition \ref{prop: lrq}, the F-signature of an lrq singularity 
 by $\bmu_n$ is equal to $\frac{1}{n}$, regardless of the embedding
 $\bmu_n\to \GL_d$.
 On the other hand, the Hilbert--Kunz multiplicity
 depends on the embedding of $\bmu_n$ inside $\GL_d$, that is,
 the integers $1=q_1\leq \dots \leq q_d<n$, as the following examples
 in dimension $d=2$ show:
 \begin{enumerate}
      \item If $q_1=q_2=1$, then $x\in X$ is the cone over a rational normal curve of degree $n$ and $e_{HK}(X)=\frac{n+1}{2}$
    by \cite[Example 2.8]{WYHK}.
     \item If $q_1=1$ and $q_2=n-1$, then $x\in X$ is a rational double point of type $A_{n-1}$ and $e_{HK}(X)=2-\frac{1}{n}$ by
     \cite[Theorem 5.4]{WYHK}.
 \end{enumerate}
\end{Remark}

Combining Corollary \ref{cor: everywheregoodreductiongroup} and Proposition \ref{prop: toric}, we conclude that the only lrq singularities with ``lrq reduction" modulo every prime are the toric $\mathbb{Q}$-factorial singularities. More precisely:

\begin{Definition} \label{def: goodreduction} 
Let $X$ be an lrq singularity over an algebraically closed field $K$ of characteristic $0$. 
Let $p \in \ZZ$ be a prime. 
We say that \emph{$X$ has good reduction modulo $p$} if there exists a discrete valuation 
ring $R$ of mixed characteristic $(p,0)$, a finite and flat $R$-group scheme $\mathcal{G}$ 
such that all geometric fibres of $\mathcal{G}$ are linearly reductive, 
and an action of $\mathcal{G}$ on $\mathcal{Y} := \Spec R[[u_1,\hdots,u_d]]$ 
that is very small on every geometric fibre and such that the generic fibre of the quotient $\mathcal{X} := \mathcal{Y}/\mathcal{G}$ coincides with $X$ over a 
common field extension of ${\rm Frac}(R)$ and $K$. 
\end{Definition}

\begin{Proposition} \label{prop: goodreductionimpliestoric}
Let $X$ be an lrq singularity by a finite group $G$ over an algebraically closed 
field $K$ of characteristic $0$. 
Then, $X$ has good reduction modulo every prime $p \in \ZZ$ if and only if $X$ is 
toric and then, $X$ can be defined over $\ZZ$ (as an lrq singularity).
\end{Proposition}

\begin{proof}
Let $G$ be the finite group by which $X$ is a quotient singularity.

First, assume that $X$ has good reduction modulo every prime $p$ and choose a model as 
in Definition \ref{def: goodreduction}. 
Then, the geometric special fibre $G_p$ of $\mathcal{G}$ is a linearly reductive group 
scheme over an algebraically closed field of characteristic $p$ that 
admits a very small action. 
Moreover, by Proposition \ref{prop: geometricgenericfiberisassociated}, 
we have $(G_p)_{\mathrm{abs}} \cong G$. 
By Corollary \ref{cor: everywheregoodreductiongroup}, this implies that $G$ is cyclic, 
hence $X$ is toric by Proposition \ref{prop: toric}.

Conversely, assume $X$ is toric. 
By Proposition \ref{prop: toric}, we can write $X$ as a quotient of 
$K[[u_1,\hdots,u_d]]$ by a diagonal action of $\bmu_n$ with weights $(q_1,\hdots,q_d)$. 
Now, both $\bmu_n$ and this very small diagonal action descend to $\ZZ$. 
This shows that $X$ can be defined over $\ZZ$ and has good reduction modulo every prime. 
\end{proof} 

\begin{Remark}
If, in the setting of Definition \ref{def: goodreduction}, 
we assume $\dim(X) \geq 3$, then $X$ has good reduction modulo $p$ if and only 
if we find a DVR $R$ of mixed characteristic $(p,0)$ and a 
flat model $\mathcal{X}$ of $X$ over $R$ whose geometric special fibre is an lrq singularity. 
This will follow from Corollary \ref{cor: canonicallift}.
\end{Remark}

\subsection{Cohen--Macaulayness of quotient singularities}
\label{subsec: CM}
We have seen in Proposition \ref{prop: lrq}, that lrq singularities, being direct 
summands of regular local rings, are Cohen--Macaulay 
(see also \cite[Proposition 12]{HochsterEagon}). 
In this section, we will show that this fails for many, but not all, 
quotient singularities that are not lrq singularities. 
The case of wild quotient singularities by abstract groups is due to 
Fogarty \cite{Fogarty}.

\begin{Lemma}
\label{lem: Cohen--Macaulay}
 Let $x\in X$ be a quotient singularity by a finite group scheme $G$ 
 in dimension $d\geq 2$. If $G$ admits a quotient isomorphic to $\balpha_p$ or $\C_p$ 
 (for example, if $G$ is unipotent), then $\depth(X) = 2$.
 In particular, if additionally $d\geq3$, then $X$ is not Cohen--Macaulay.
\end{Lemma}
\begin{proof}
Since $X$ is normal, we have $\depth(X) \geq 2$ and thus, 
it remains to show $\depth(X) \leq 2$.
First, assume that $G$ admits a normal subgroup scheme $N$ such that $G/N \cong \C_p$. 
By \cite[Proposition 4(a)]{Fogarty} and using that $Y/N$ is normal, 
we have $\depth(X) = \depth((Y/N)/\C_p) \leq 2$.
Next, suppose that $G$ admits a normal subgroup scheme $N$ such that $G/N \cong \balpha_p$. 
Then $Y/N \to X$ is an $\balpha_p$-torsor over $X-x$ that does not extend to $X$.
As in the proof of Proposition \ref{prop: lrq=Fregular}, this shows that 
$H^2_{\{x\}}(X,\OO_{X})$ is non-trivial, which implies $\depth(X) \leq 2$.

To finish the proof, it suffices to recall that if $G$ is unipotent, then it admits a decomposition series with quotients isomorphic to $\C_p$ or $\balpha_p$, 
see \cite[Exp. XVII, Th\'eor\`eme 3.5]{SGA32}.
\end{proof}

\begin{Example} \label{ex: Cohen--Macaulay} 
There exists an example of a $3$-dimensional quotient singularity by a finite
and non linearly reductive group 
scheme that is Cohen--Macaulay. 

Suppose $p = 2$ and let $H = \langle h \rangle$ be 
a cyclic group of odd order $N \geq 5$. 
Consider the action on $A = k[[u_1, u_2, u_3]]$ of $\balpha_2$ 
corresponding to the derivation $D(u_i) = u_i^2$ and the
action of $H$ on $A$ given by $h(u_i) = \zeta u_i$, 
where $\zeta$ is a primitive $N$-th root of unity.
This induces an action of the semidirect product $G = \balpha_2 \rtimes H$
on $A$.
We note that $G$ is neither linearly reductive nor unipotent and that it does not admit
a quotient isomorphic to $\balpha_2$ or $\C_2$.
We claim that $A$ is Cohen--Macaulay.
\end{Example}
\begin{proof}
Since $A$ is 3-dimensional, it suffices to show $\depth(A^G)\geq3$.
First, we observe that $B := A^D$ is generated by $1,y_1,y_2,y_3,z$ as a module 
over $A^{(p)} = k[[x_1^2, x_2^2, x_3^2]]$, 
where 
$$
 y_i \,:=\, u_{i+1}^2 u_{i+2} + u_{i+1} u_{i+2}^2
 \mbox{\quad and \quad}
 z := u_1 u_2 u_3 (u_1 + u_2 + u_3)
$$
(consider the indices modulo $3$), 
subject to the single relation $\sum_i u_i^2 y_i = 0$, 
that is, 
$$
 B \,\cong\, \Coker\left(A^{(p)} \,\xrightarrow{(u_1^2, u_2^2, u_3^2, 0, 0)}\, (A^{(p)})^{\oplus 5}\right).
$$
Then we have $H^2_{\idealm}(B) \cong \Ker(H^3_{\idealm}(A^{(p)}) \to H^3_{\idealm}(A^{(p)})^{\oplus 5}) \cong H^3_{\idealm}(A^{(p)})[\idealm_{A^{(p)}}]$, which is $1$-dimensional over $k$. 
With respect to the regular sequence $u_1^2,u_2^2,u_3^2$ of $A^{(p)}$, this 
cohomology group is generated by the class $[u_1^{-2} u_2^{-2} u_3^{-2}]$.
Since $h$ acts on this space by $\zeta^{-6} \neq 1$, we have 
$H^2_{\idealm}(B^H) = H^2_{\idealm}(B)^H = 0$, 
hence the depth of $A^G = B^H$ is $\geq 3$.
\end{proof}

\subsection{More general quotients by linearly reductive group schemes}
\label{subsec: more general}
Let $G\subset\GL_d$ be a finite and linearly reductive group scheme over 
the algebraically closed field $k$ of characteristic $p\geq0$.
In Definition \ref{def: lrq}, we considered quotient singularities
$x \in X:= 0 \in (\Aff^d/G)^\wedge$ such that the $G$-action is very small. 
We have seen in Corollary \ref{cor: verysmallactionvsverysmallrep} that 
these correspond to $G$-representations with $\lambda$-invariant $0$.

Now, one could also consider $G$-actions with larger fixed locus, 
(or, equivalently, $G$-representations with non-trivial $\lambda$-invariant)
and refer to the singularities considered in Definition \ref{def: lrq}
and in this article as \emph{isolated lrq singularities}.
In this case, not all conclusions of Proposition \ref{prop: lrq}
may hold for $x\in X$. 
In fact, $x\in X$ might even be a smooth point, in particular, 
$s(X)=1$ and $\pietloc(X)=\{e\}$,
even if $G^{\et}$ is non-trivial.

To put this into perspective, we recall that a subgroup scheme of $G$ acts 
via \emph{pseudo-reflections} if its fixed locus has codimension $1$
in $\Aff^d$ - here, we follow Satriano \cite[Definition 1.2]{Satriano} 
since this definition also works in the case where $G$ 
is not necessarily \'etale.
In particular, a pseudo-reflection fixes a hyperplane in $\Aff^d$.
Using the representation $\rho$ associated to the linear $G$-action on $\Aff^d$, 
Proposition \ref{prop: lambda} shows that $G$ contains a subgroup scheme 
that acts via pseudo-reflections if and only if $\lambda(\rho) = d-1$, 
see also Remark \ref{rem: lambda}.

Pseudo-reflections generate a normal subgroup scheme
$N\unlhd G$.
By the theorem of Chevalley--Satriano--Shephard--Todd \cite{Satriano}, 
the quotient $\Aff^d/G$ is smooth if and only
if $G$ is generated by pseudo-reflections, that is, if and only if $N=G$.
In particular, after replacing $G$ by $G/N$ and $\Aff^d$ by $\Aff^d/N\cong\Aff^d$,
we may assume that $G$ acts without pseudo-reflections and that
the $G$-action is free outside a closed subset of codimension $\geq2$.

In particular, if $d=2$, then every singularity $x \in X = 0 \in (\Aff^2/G)^\wedge$
arising in this more liberal sense is an lrq singularity in the sense of 
Definition \ref{def: lrq}. For higher $d$, the singularities that arise in this more general sense form a strictly bigger class than the lrq singularities we study in this article. However, these singularities are all still strongly F-regular, so the interested reader will find many interesting properties of these singularities in \cite{Carvajal-Rojas}.

\section{Uniqueness of the quotient presentation} \label{sec: uniqueness}

If $x\in X$ is an lrq singularity with respect to the 
finite and linearly reductive group scheme $G$, 
then Proposition \ref{prop: lrq} shows that the length of $G$,
the maximal \'etale quotient $G^\et$ of $G$, and the character group
$\Hom(G,\GG_m)$ are invariants of $x\in X$.
This suggests that an lrq singularity determines the 
linearly reductive group scheme $G$ together
with a linear action on a smooth $k$-algebra.
This is indeed the case by the following uniqueness result, which is the main
result of this section.

\begin{Theorem} \label{thm: lrqmain}
 Let $x \in X = 0 \in ({\mathbb{A}^d/G_1})^\wedge$ be a $d$-dimensional lrq singularity 
 with notations 
 and assumptions as in Proposition \ref{prop: lrq}. 
 Let $G_2$ be a finite $k$-group scheme such that $x \in X$ is a quotient singularity by $G_2$
 in the sense of Definition \ref{def: quotient singularity}.
 Then,
 \begin{enumerate}
  \item $G_2$ is isomorphic to $G_1$ as finite $k$-group scheme,
  \item the $G_2$-action on $\widehat{\mathbb{A}}^d$ is linearisable, and
  \item with respect to linearisations, $G_1$ and $G_2$ are conjugate as subgroup schemes of $\GL_d$, 
 \end{enumerate}
\end{Theorem}

\begin{proof}
For the proof, we set $S := k[[u_1,\hdots,u_d]]$ with 
$\widehat{\mathbb{A}}^d := \Spec S$ and assume that $G_2$ acts 
on $\widehat{\mathbb{A}}^d$, such that the action of $G_2$ is very small and
with quotient $\widehat{\mathbb{A}}^d/G_2$ isomorphic to $X$.

Since $X$ is an lrq singularity, we know from Proposition \ref{prop: lrq} (\ref{item: F-regular Q-Gorenstein}) that $X$ is F-regular. 
Then, the same argument as in the proof of Proposition \ref{prop: lrq=Fregular} 
shows that $G_2$ is linearly reductive.
In particular, $G_1^\circ$ and $G_2^\circ$
both are diagonalisable group schemes.

To prove the theorem, it suffices to prove Claim (3) and to do this, 
we may assume that $G_2$ acts linearly on $\widehat{\mathbb{A}}^d$.
By Proposition \ref{prop: lrq} (\ref{item: pietloc}) we have 
$
 G_1^{\et} \,\cong\, \pietloc(X) \,\cong\, G_2^{\et}.
$
Thus, $\widehat{\mathbb{A}}^d/G_1^{\circ}$ and $\widehat{\mathbb{A}}^d/G_2^{\circ}$ 
both induce the universal \'etale cover of $X-x$, 
from which we obtain a $\pietloc(X)$-equivariant isomorphism 
$\phi \colon \widehat{\mathbb{A}}^d/G_1^{\circ} \cong \widehat{\mathbb{A}}^d/G_2^{\circ}$.
By Proposition \ref{prop: lrq} (\ref{item: class group}), we have 
$$
 \Hom(G_1^{\circ}, \GG_m) \,=\, {\rm Cl}(\widehat{\mathbb{A}}^d/G_1^{\circ}) \,\cong\, 
 {\rm Cl}(\widehat{\mathbb{A}}^d/G_2^{\circ}) \,=\, \Hom(G_2^\circ, \GG_m),
$$
which implies $ G_1^{\circ} \cong G_2^{\circ}$. 
Considering the action of $\pietloc(X)$ on the class groups, 
it follows that the two actions $G_i^{\et} \to \Aut(G_i^{\circ})$ coincide.
But these actions determine the semidirect product structure
$G_i\cong G_i^\circ\rtimes G_i^\et$, which implies that 
$G_1 \cong G_2$ as $k$-group schemes.
Now, we apply Lemma \ref{lem: toricautomorphisms} and Claim (3) follows.
\end{proof}

It remains to establish the following technical result, which is slightly stronger than 
what we need, since we do not assume that $G^\et$ acts with isolated fixed locus.

\begin{Lemma} \label{lem: toricautomorphisms}
Let $G$ be a finite and linearly reductive group scheme over $k$, 
let $\widehat{\mathbb{A}}^d := \Spec k[[u_1,\hdots,u_d]]$, and assume that we have two actions 
$\rho_h$, $h = 1,2$, of $G$ on $\widehat{\mathbb{A}}^d$ with quotient morphisms $f_h: \widehat{\mathbb{A}}^d \to \widehat{\mathbb{A}}^d/\rho_h(G^\circ)$ 
such that the $\widehat{\mathbb{A}}^d/\rho_h(G^\circ)$ are lrq singularities by $G^\circ$. 

If there is a $G^\et$-equivariant
isomorphism $\widehat{\mathbb{A}}^d/\rho_1(G^\circ) \cong \widehat{\mathbb{A}}^d/\rho_2(G^\circ)$, 
then there is a commutative diagram of $G$-equivariant 
morphisms
$$
\xymatrix{
\widehat{\mathbb{A}}^d \ar[r]^{\psi} \ar[d]^{f_1} & \widehat{\mathbb{A}}^d \ar[d]^{f_2} \\
\widehat{\mathbb{A}}^d/\rho_1(G^\circ) \ar[r]^{\varphi} & \widehat{\mathbb{A}}^d/\rho_2(G^\circ),
}
$$
where $\psi$ and $\varphi$ are isomorphisms. 
In particular, the actions $\rho_1$ and $\rho_2$ are conjugate via $\psi$.
\end{Lemma}

\begin{proof}
First, since $\widehat{\mathbb{A}}^d/\rho_h(G^\circ)$ is an lrq singularity by $G^\circ$, 
we have $G^0 = \bmu_{p^n}$ for some $n \geq 0$ by Proposition \ref{prop: toric}. 
Next, we note that we can conjugate the $\rho_h$ by automorphisms of $\widehat{\mathbb{A}}^d$ without 
changing the assertion. 
Thus, we may assume that $\rho_1$ is a linear action, that $\rho_1(G^\circ)$ acts diagonally, 
that $X := \widehat{\mathbb{A}}^d/\rho_1(G^\circ) = \widehat{\mathbb{A}}^d/\rho_2(G^\circ) = \Spec k[[M]]$ for an affine semigroup $M$, 
that the two $G$-actions on $X$ coincide, and that $f_1$ is induced by an inclusion 
of monoids $M \subseteq N = \mathbb{N}^d$ with cokernel $N/M \cong {\rm Cl}(X) = (G^\circ)^D$. 
In particular, we can write $(f_1)_* \mathcal{O}_{\widehat{\mathbb{A}}^d} = k[[N]] = \bigoplus_{i \in N/M} I_i$, 
where the $I_i$ are the eigenspaces for the action of $G^\circ$ on $k[[u_1,\hdots,u_d]]$.
Note that we have $g(I_i) = I_{g(i)}$ for all $g \in G^{\et}$, since $G^\et$ 
normalises $G^\circ$.

Since $f_2: \widehat{\mathbb{A}}^d \to X$ is an integral $G^\circ$-torsor over the smooth locus of $X$ and 
$\{I_i\}_{i \in N/M}$ is a full set of representatives for ${\rm Cl}(X)$, 
by Proposition \ref{prop: toric} we have an isomorphism $(f_2)_* \mathcal{O}_{\widehat{\mathbb{A}}^d} \cong \bigoplus_{i \in N/M} I_i$ of $\mathcal{O}_X$-modules with $G$-action 
(but not necessarily of $\mathcal{O}_X$-algebras). 
Denote by $\chi_{h,i_1,i_2}: I_{i_1} \otimes I_{i_2} \to I_{i_1+i_2}$ the morphisms induced by multiplication on $(f_h)_* \mathcal{O}_{\widehat{\mathbb{A}}^d}$. To prove the lemma, we have to construct $k$-linear automorphisms $\psi_i: I_{i} \to I_i$ such that
\begin{enumerate}
    \item $\chi_{2,i_1,i_2} \circ (\psi_{i_1} \otimes \psi_{i_2}) = \psi_{i_1+i_2} \circ \chi_{1,i_1,i_2}$ for all $i_1,i_2 \in {\rm Cl}(X)$, and
    \item $g \circ \psi_i = \psi_{g(i)} \circ g$ for all $g \in G^\et$.
\end{enumerate}
\noindent The first condition ensures that the morphism 
$\psi := \bigoplus_{i \in {\rm Cl}(X)} \psi_i: (f_1)_* \mathcal{O}_{\widehat{\mathbb{A}}^d} \to (f_2)_* \mathcal{O}_{\widehat{\mathbb{A}}^d}$ 
is an isomorphism of $k$-algebras and the second condition ensures that $\psi$ is $G$-equivariant 
(as $\psi$ is compatible with the $G^\circ$-action by construction). 
The sought automorphism of $X$ is then induced by $\varphi := \psi_0$.

Consider the maps $\phi_{h,i} \colon I_i^{\otimes p^n} \to \mathcal{O}_X$ induced by
$\chi_{h,i,i}$. 
Then, since the cokernels of the $\phi_{h,i}$'s are supported on $\{ \idealm \}$, there exist units 
$c_i \in \mathcal{O}_X^{\times}$ with $\phi_{2,i} = c_i \phi_{1,i}$.
This implies that we have 
$\chi_{2,i_1,i_2} = (c_{i_1} c_{i_2} c_{i_1 + i_2}^{-1})^{1/p^n} \chi_{1,i_1,i_2}$.
In particular, we have $c_{i+j} \equiv c_i c_j \pmod{(\mathcal{O}_X^{\times})^{p^n}}$.
Since the $f_h$ are $G^\et$-equivariant, the family $(c_i)$ is $G^\et$-equivariant,
that is, $c_{g(i)} = g(c_i)$ for $g \in G^\et$.

Let $e_1, \dots, e_d$ be the canonical basis of $N = \NN^d$.
Define a homomorphism of monoids $\delta \colon N \to \OO_X^{\times}$ by $\delta(e_j) = c_{[e_j]}$,
where $[e_j] \in N/M$ denotes the class of $e_j \in N$.
By construction, we have $\delta(\nu) \equiv c_{[\nu]} \pmod{(\OO_X^{\times})^{p^n}}$ for all $\nu \in N$.

Now, observe that every element of $I_i$, when considered as an element of $k[[N]]$ via $f_1$, 
is a (possibly infinite) sum of monomials in $I_i$, because $G^\circ$ acts diagonally. 
Hence, we can define $k$-linear morphisms $\psi_i: I_i \to I_i$ by mapping a monomial 
$\nu_i \in I_i \subseteq k[[N]]$ to $(\delta(\nu_i) c_{i}^{-1})^{1/p^n} \nu_i$ and 
extending linearly. 
We claim that $\psi := \bigoplus_{i \in {\rm Cl}(X)} \psi_i$ satisfies properties (1) and (2) above. 
Both properties can be checked on monomials $\nu_{i_j} \in I_{i_j}$. 

For (1), we compute
\begin{eqnarray*}
&\chi_{2,i_1,i_2} \circ (\psi_{i_1} \otimes \psi_{i_2})(\nu_{i_1} \otimes \nu_{i_2})  \\
=& (c_{i_1} c_{i_2} c_{i_1 + i_2}^{-1})^{1/p^n} \chi_{1,i_1,i_2} ((\delta(\nu_{i_1}) c_{i_1}^{-1})^{1/p^n} \nu_{i_1} \otimes (\delta(\nu_{i_2}) c_{i_2}^{-1})^{1/p^n}) \nu_{i_2})  \\
=& (\delta(\nu_{i_1} + \nu_{i_2})c_{i_1 + i_2}^{-1})^{1/p^n} \chi_{1,i_1,i_2}(\nu_{i_1} \otimes \nu_{i_2}) \\
=& \psi_{i_1 + i_2} \circ \chi_{1,i_1,i_2}(\nu_{i_1} \otimes \nu_{i_2}).
\end{eqnarray*}
As for (2), we use that for any monomial $\lambda$ appearing in $g(\nu)$ we have $\delta(\lambda) = g(\delta(\nu))$ (it suffices to check the case where $\nu \in \{e_1, \dots, e_d\}$, which is clear). Setting $\delta(g(\nu)) := \delta(\lambda)$ and using $g(c_i) = c_{g(i)}$, we obtain
\begin{eqnarray*}g \circ \psi_i (\nu_i) = g ((\delta(\nu_i)c_i^{-1})^{1/p^n} \nu_i) 
= (\delta(g(\nu_i)) c_{g(i)}^{-1})^{1/p^n} g(\nu_i) 
= \psi_i (g(\nu_i)).
\end{eqnarray*}
Therefore, (1) and (2) hold for $\psi$ and thus, $\psi$ and $\varphi:=\psi_0$
induce the stated commutative diagram.
\end{proof}

\begin{Remark}
If $G$ is connected, that is, $G=G^\circ$, and using the terminology of local torsors 
as in \cite{RDP}, Lemma \ref{lem: toricautomorphisms} 
can be rephrased as follows: 
Let $X$ be an lrq singularity by a finite, connected, and linearly reductive group scheme $G$. 
Then, the pullback action of $\Aut(X)$ on the set of strictly local $G$-torsors 
over $X$ is transitive.

However, we warn the reader that this does \emph{not} imply that any two strictly 
local $G$-torsors over $X$ are isomorphic over $X$, 
but merely that they are $G$-equivariantly isomorphic as $k$-schemes.
\end{Remark}

\begin{Corollary} \label{cor: lrq<->verysmall}
For every algebraically closed field $k$, there is a bijection between
\begin{enumerate}
    \item the set of isomorphism classes of (Gorenstein) lrq singularities of dimension $d$, and
    \item the set of conjugacy classes of finite and very small
    subgroup schemes of $\GL_{d,k}$ (resp. $\SL_{d,k}$).
\end{enumerate}  
\end{Corollary}

\begin{proof}
The statement for $\GL_{d,k}$ follows from Theorem \ref{thm: lrqmain}. 
Thus, it suffices to note that $x \in X = 0 \in ({\mathbb{A}^d/G})^\wedge$ is Gorenstein 
if and only if $G$ preserves the $d$-form $\omega = du_1 \wedge \hdots \wedge du_d$. 
Since $G$ acts on $\langle \omega \rangle$ through the determinant 
$G \to \GL_{d,k} \to \GG_m$, 
we see that $X$ is Gorenstein if and only if $G \subseteq \SL_{d,k}$.
\end{proof}

We will use Corollary \ref{cor: lrq<->verysmall} to give an explicit description of 
all lrq singularities in dimension $d = 2$ in Section \ref{sec: Fregular} 
and show that, in this case, the lrq singularities are precisely 
the F-regular singularities. 

\begin{Remark}
Let $x \in X = 0 \in (\mathbb{A}^d/G)^\wedge$ be a quotient singularity over an algebraically closed field $k$. If $ k =\mathbb{C}$, then the group $G$ can be recovered as the local \'etale fundamental group $G = \pietloc(X)$ and the action of $G$ on $(\mathbb{A}^d)^\wedge$ can be linearised and two actions of $G$ yield the same quotient $X$ if and only if the actions are conjugate in $\GL_{d,\mathbb{C}}$ (see \cite{Prill}).

This cannot be true in positive characteristic, since $G$ is not necessarily \'etale, and one has to take into account infinitesimal torsors over $X \setminus \{x\}$. However, we remark that in \cite{EsnaultViehweg}, Esnault and Viehweg used Nori's fundamental group scheme to define an analogue $\piNloc(X)$ of $\pietloc(X)$ that also takes into account torsors under finite group schemes over $X \setminus \{x\}$. However, at least if $G$ is not linearly reductive, the obvious guess that $G = \piNloc(X)$ holds is not true in positive characteristic, even if $X$ is a rational double point and $G$ is \'etale. For a more in-depth discussion of $\piNloc(X)$ for surface singularities, including counterexamples to $G = \piNloc(X)$ and geometric reasons for the failure of this equality, we refer the reader to \cite{RDP}. 
\end{Remark}

Thus, Theorem \ref{thm: lrqmain} begs for the following question:
\begin{Question}
Let $x \in X = 0 \in (\mathbb{A}^d/G)^\wedge$ be an lrq singularity. Does $\piNloc(X) = G$ hold?
\end{Question}

\section{Rigidity of lrq singularities in dimension $d \geq 4$} 
\label{sec: rigidity}

In \cite{Schlessinger}, Schlessinger 
proved infinitesimal rigidity of isolated quotient singularities by finite groups 
of order prime to the characteristic and in dimension $\geq3$.
We will show in this section that his proof can be adjusted to 
lrq singularities in dimension $d\geq4$
and we will deal with $d=3$ in Section \ref{sec: dimension 3}.

\begin{Convention} \label{def: rigid2}
We say that an lrq singularity $X$ is \emph{rigid} (resp. \emph{infinitesimally rigid}) 
if all deformations of $X$ over equicharacteristic complete DVRs $(R,m)$
(resp. equicharacteristic Artinian local rings) are trivial modulo $m^n$ for every $n > 0$ (resp. are trivial).

Here, a deformation of an lrq singularity $X$ over a DVR $R$ is a flat 
family $\mathcal{X} \to \Spec R$ together with an identification of the special 
fibre of $\mathcal{X}$ with $X$ and such that the non-smooth locus of 
$\mathcal{X} \to \Spec R$ is proper over $\Spec R$.
\end{Convention}

Indeed, we already know from Proposition \ref{prop: liftinggroupscheme} 
that lrq singularities are \emph{not} arithmetically rigid 
(that is, rigid in mixed characteristic).

\subsection{Infinitesimal rigidity in dimension $\mathbf{\geq4}$}

\begin{Proposition}\label{prop: rigidity}
Let $x \in X = 0 \in ({\mathbb{A}^d/G})^\wedge$ be a $d$-dimensional lrq singularity. 
If $d \geq 4$, then $X$ is infinitesimally rigid.
\end{Proposition}

\begin{proof}
We follow the proof given in Schlessinger \cite{Schlessinger} 
for the case where $G$ is a finite group of order prime to the characteristic. 
Setting $S = k[[u_1,\hdots,u_d]]$ and $R = S^G$ with punctured spectra $V$ and $U$, respectively, 
Schlessinger's article shows that it suffices to prove that $h^1(U,T_U) = 0$.

Since $T_{V/G^\circ}^{G^\et} = T_U$ by \cite[Section 2, Proposition]{Schlessinger}, 
we may assume that $G$ is infinitesimal, that is, $G = G^\circ$. 
Moreover, since the $G$-action on $\mathbb{A}^d$ has an isolated fixed locus, 
we may assume that $G^\circ = \bmu_{p^n}$ for some $n \geq 0$ 
by Proposition \ref{prop: linearlyreductive or unipotent}.

We let $f:V\to U$ be the quotient map and denote by $T_V$ and $T_U$ their
tangent sheaves.
By \cite[Corollary 3.4]{Ekedahl} (see also \cite[Lemma 2.11]{Matsumotoalpha}), 
there is an exact sequence
$$
 0 \,\to\, T_{V/U} \,\to\, T_V \,\to\, f^* T_U \,\to\, T_{V/U}^{p^n} \,\to\, 0
$$
of $G$-equivariant sheaves on $V$, where $T_{V/U}$ is locally free 
and where $T_{V/U}^{p^n}$ is the pullback of $T_{V/U}$ along the 
$n$-th iterate of the absolute Frobenius.
More precisely, by loc.cit., ${\rm log}_p \deg(f)$ is equal to 
the height of the purely inseparable morphism $f$ times the rank
of $T_{V/U}$, which implies that $T_{V/U}$ is of rank 1.

Since $\Spec S$ is smooth over $k$, the reflexive hull of $T_{V/U}$ on $\Spec S$ is invertible, 
hence $T_{V/U}$ is free. 
Since $G$ is linearly reductive, applying $f_*$ and taking $G$-invariants is exact.
Thus, we obtain an exact sequence 
\begin{equation} \label{eq: invariants}
0 \,\to\, (f_* T_{V/U})^G \,\to\, (f_* T_V)^G \,\to\, T_U \,\to\, (f_* T_{V/U}^{p^n})^G \,\to\, 0. 
\end{equation}

Using once more that $G$ is linearly reductive, we know that $(f_* T_{V/U})^G$, $(f_* T_{V/U}^{p^n})^G$, and $(f_* T_V)^G$ 
are direct summands of $f_* T_{V/U}$, $f_* T_{V/U}^{p^n}$, and $f_* T_V$, respectively. 
Since $S$ is regular and the latter three sheaves are pushforwards of free sheaves on $V$, 
we deduce that 
$$
  h^i\left(U,(f_* T_{V/U})^G\right) \,=\, 
  h^i\left(U,(f_* T_{V/U}^{p^n})^G\right) \,=\,  
  h^i\left(U,(f_* T_V)^G\right) \,=\, 0
$$
for $0 < i < \dim(V) - 1$.
We assumed $\dim(V) = \dim(X) = d \geq 4$ and plugging this into the exact sequence \eqref{eq: invariants}, 
we obtain the desired result.
\end{proof}

\begin{Remark}
 As one of the referees pointed out to us, 
 this proposition was already known to the experts:
 Proposition \ref{prop: rigidity} 
 follows directly from the proof of
 \cite[Theorem 4.9]{SatrianoDeRham}
 and we will now  give the details of because 
 some of the arguments
 are somewhat difficult to trace.
\end{Remark}

\begin{proof}
(A proof of Proposition \ref{prop: rigidity} using tame stacks.)
The singularity $x \in X = 0 \in ({\mathbb{A}^d/G})^\wedge$
lifts to $W_2(k)$ - in fact, we can even lift it as an
lrq singularity over $W(k)$, see 
Proposition \ref{prop: liftinggroupscheme}.

Since $d\geq4$, it follows from the proof of
\cite[Theorem 4.9]{SatrianoDeRham} that
every deformation of $X=({\mathbb{A}^d/G})^\wedge$ 
is induced by a deformation of the stack
$\mathcal{X}:=[\mathbb{A}^d/G]$.
More precisely, 
the proof shows that $\Ext^i(L_{\mathcal{X}/X},\calO_X)=0$
if $i<d-2$ and in this proof, this boils down 
to a local statement, that is,
the properness assumption of \cite[Theorem 4.9]{SatrianoDeRham} 
is not needed for this.
In particular, if $d\geq4$, 
this implies that every deformation of $X$ is
induced by a deformation of $\mathcal{X}$.

The assertion now follows from the fact that
$\mathcal{X}$ is infinitesimally rigid.
Again, this is known to the experts
and follows, for example, from combining
\cite[Proposition 5.2]{Satriano} and
\cite[Lemma 5.3]{LiedtkeSatriano}.
\end{proof}

\subsection{Deformations of non-lrq quotient singularities}
\label{subsec: deformation non-lrq}
To complete the picture, let us note that rigidity fails for quotient singularities by group schemes that 
are not linearly reductive in dimension $d \geq 3$. 

Let $B:=\Spec \FF_p[[t]]$ with closed point $0$ and generic point $\eta$
and let ${\cal G}\to B$ be a finite and flat group scheme of length $p$
given by Oort--Tate parameters $a=t^{p-1}$ and $b=0$ as described in 
\cite[Theorem 2]{TateOort}.
Thus, the fibre of $\cal G$ over $0$ (resp. $\eta$) is $\balpha_p$ (resp. $\C_p$).
In \cite[Part II, Chapter 7]{ItoMiyanishi}, this group scheme is called
the \emph{unified $p$-group scheme}. 
Explicitly, we have
$$
{\cal G} \,=\, \Spec \FF_p[[t]][X]/(X^p - t^{p-1}X)
$$
with co-multiplication induced by the ambient $\mathbb{G}_{a,B} = \Spec \FF_p[[t]][X]$.

In particular, $\mathcal{G}$ admits a natural embedding $\iota$ into $\mathbb{G}_{a,B}$. 
Consider any faithful action $\rho$ of $\mathbb{G}_a$ on $\mathbb{P}^1$ that fixes $\infty$. 
Now, the action of $\mathcal{G}$ on $(\mathbb{P}^1_B)^d$ given via
$$
\mathcal{G} \,\overset{\iota}{\to}\, \mathbb{G}_{a,B} 
\,\overset{\Delta}{\to}\, (\mathbb{G}_{a,B})^d 
\,\overset{\rho^d_B}{\to}\, ({\rm PGL}_{2,B})^d
$$
fixes $\infty_B^d$ and is free on a punctured neighbourhood. 
Assuming that formation of quotients commutes with restriction to the special fiber 
in this setting (as pointed out to us by Posva, this is not automatic in general) 
and passing to the completion at $\{\infty\}^d$, we obtain a family of 
quotient singularities whose generic (resp. special) fibre is a quotient singularity by 
$\C_p$ (resp. $\balpha_p$) and with local \'etale fundamental group $\C_p$ (resp. trivial). 
 
In the following, we verify that taking quotients and fibers commutes in this setting for 
some small values of $p$ and $d$: 

\begin{Example}[$p = 2, d = 3$] \label{ex: nonrigidexample}
Consider the co-action
$$\begin{array}{ccccc}
\rho^{\sharp} &: & S = \mathbb{F}_2[[t,u_1,u_2,u_3]] & \to & 
S \otimes_{\mathbb{F}_2[[t]]} \mathbb{F}_2[[t]][X]/(X^2 + tX) \\
&& u_i & \mapsto & \frac{u_i}{1 + u_i X }.
\end{array}$$
For $s \in S$, we let ${\rm tr}(s) = \frac{1}{t}(s + \rho^{\sharp}(s)(t))$, 
where we consider $\rho^{\sharp}(s)$ as a polynomial in $X$.
Define $S^{\rm tr} \subseteq S$ to be the $\mathbb{F}_2[[t]]$-subalgebra 
$$
S^{\rm tr} \,=\, \mathbb{F}_2[[t,{\rm tr}(u_1),{\rm tr}(u_2),{\rm tr}(u_3),
{\rm tr}(u_1u_2),{\rm tr}(u_1u_3),{\rm tr}(u_2u_3),{\rm tr}(u_1u_2u_3)]].
$$
Then, we have inclusions $S^{\rm tr} \subseteq S^{\mathcal{G}} \subseteq S$. 
Denote by $(-)_0$ the image of a subring of $S$ in the reduction $S_0$ of $S$ modulo $t$. 
Then, we obtain a chain of subrings
$$
S^{\rm tr}_0 \,\subseteq\, S^{\mathcal{G}}_0 \,\subseteq\, (S_0)^{\mathcal{G}_0} \,\subseteq\, S_0
$$
and we compute
$$
S^{\rm tr}_0 \,=\,
\mathbb{F}_2[[t,u_1^2,u_2^2,u_3^2,u_1u_2(u_1+u_2),u_1u_3(u_1+u_3),u_2u_3(u_2+u_3),u_1u_2u_3(u_1+u_2+u_3)]].
$$
One checks that this is a normal ring and that the extension $S_0^{\rm tr} \subseteq S_0$ 
has degree $2$, hence the birational inclusions 
$S_0^{\rm tr} \subseteq S_0^{\mathcal{G}} \subseteq (S_0)^{\mathcal{G}_0}$ 
are all equalities, which, in particular, implies the desired equality 
$S_0^{\mathcal{G}} = (S_0)^{\mathcal{G}_0}$.
\end{Example}

\begin{Remark}[$p = 2, d = 2$]
For the co-action
$$\begin{array}{ccccc}
\rho^{\sharp} &: & S = \mathbb{F}_2[[t,u_1,u_2]] & \to & S \otimes_{\mathbb{F}_2[[t]]} \mathbb{F}_2[[t]][X]/(X^2 + tX) \\
&& u_i & \mapsto &\frac{u_i}{1 + u_i X },
\end{array}$$
and with $\mathrm{tr}(s)$ as in the previous example,
we can directly compute the invariant ring as
$$
S^{\mathcal{G}} \,=\, \mathbb{F}_2[[t,{\rm tr}(u_1),{\rm tr}(u_2),{\rm tr}(u_1u_2)]] 
\,\cong\, \mathbb{F}_2[[t,x,y,z]]/(z^2 + txyz + x^2y + xy^2),
$$
which yields a family of $\cal G$-quotients
  with generic (resp. special) fibre a rational double point singularity 
  of type $D_4^1$ (resp. $D_4^0$).
\end{Remark}

\section{Deformation theory of lrq threefolds}
\label{sec: dimension 3}
\label{subsec: nonrigidity}

In this section, we establish a necessary and sufficient criterion for
a 3-dimensional lrq singularity to be infinitesimally rigid.
Moreover, we also compute the miniversal deformation spaces of 3-dimensional
lrq singularities.

Recall that we have classified very small subgroup schemes of $\GL_3$
in Theorem \ref{thm: verysmallgroupsinsmalldimension} \eqref{item: very small 3-dim}. 
To simplify the statements of this section, we will introduce the following notation 
for the associated lrq singularities.

\begin{Definition}
Let $x \in X = 0 \in ({\mathbb{A}^3/G})^\wedge$ be a $3$-dimensional lrq singularity by a linearly reductive
group scheme $G$.
\begin{enumerate}
    \item If $G \cong \bmu_n$ and $G$ acts as in Theorem \ref{thm: verysmallgroupsinsmalldimension} \eqref{item: cyclic 3-dim}, we say that
    $X$ is of type $\frac{1}{n}(1,q_1,q_2)$.
    \item If $G \cong \bmu(m,3^fN,r)$ and $G$ acts as in Theorem \ref{thm: verysmallgroupsinsmalldimension} \eqref{item: metacyclic 3-dim}, we say that $X$ is of type $\frac{1}{m}(1,r,r^2) \rtimes_{\zeta_{3^{f-1}}} \bmu_{3^fN}$. 
\end{enumerate}
\end{Definition}

\subsection{The infinitesimal rigidity criterion}
To state this criterion, we introduce
two linear representations associated to 
a $d$-dimensional lrq singularity $x \in X$:
let $G$ be the linearly reductive group scheme associated to $x\in X$
and let $\rho:G\to\GL_d$ be a very small linear action on $\mathbb{A}^d$ such 
that $x \in X = 0 \in ({\mathbb{A}^d/G})^\wedge$.
Then, 
\begin{enumerate}
\item $G$ acts via conjugation on $G$ and $G^\circ$, which gives rise to the
 \emph{adjoint representation} $\chi_{{\rm ad}}$ of $G$ on the Lie algebras 
 ${\rm Lie}(G)={\rm Lie}(G^\circ)$.
 We note that this representation depends on $G$ only.
 \begin{enumerate}
     \item If $G$ is \'etale, then ${\rm Lie}(G)$ is zero.
     \item If $G$ is not \'etale, then $G^\circ=\bmu_{p^n}$ for some $n\geq1$
 by Proposition \ref{prop: linearlyreductive or unipotent} and thus, $\chi_{{\rm ad}}$ is a one-dimensional
 representation. In particular, $\chi_{{\rm ad}}$ can be regarded as a homomorphism from $G$ to $\Aut({\rm Lie}(G)) = \GG_m$.
 \end{enumerate}
 
 \item By composing the action with the determinant, we obtain the one-dimensional representation
 $\chi_{{\rm det}}:=\det\circ\rho:G\to\GL_d\to\GG_m$,
 which depends on $G$ as well as on its action $\rho$ on $\mathbb{A}^d$.
\end{enumerate}
We note that by Theorem \ref{thm: lrqmain}, $G$ is unique up to isomorphism
and $\rho$ is unique up to conjugation, that is, the two representations
$\chi_{{\rm ad}}$ and $\chi_{{\rm det}}$
depend only on the lrq singularity $x \in X$ and not on any choices.
We refer to Example \ref{ex: lrq in dimension 3} 
for explicit computations of these representations.

\begin{Theorem} 
 \label{thm: non-rigid 3foldclassification}
Let $x \in X = 0 \in ({\mathbb{A}^3/G})^\wedge$ be a $3$-dimensional lrq singularity by a linearly reductive
group scheme $G$.
Then, $X$ is infinitesimally rigid if and only if
\begin{enumerate}
  \item $G^\circ$ is trivial or
  \item $G^\circ$ is non-trivial and $\chi_{{\rm ad}}\neq\chi_{{\rm det}}$.
\end{enumerate}
\end{Theorem}

We note that the first part is precisely Schlessinger's result \cite{Schlessinger}
that a 3-dimensional isolated quotient singularity by a group of order prime 
to the characteristic is infinitesimally rigid.

\begin{proof}
By Nagata's classification, we have $G=G^\circ\rtimes G^\et$, where $G^\et$ is of order
prime to $p$.
In particular, if $G^\circ$ is trivial, then $X$ is infinitesimally
rigid by Schlessinger's theorem \cite{Schlessinger}.
We may now assume that $G^\circ$ is non-trivial.

We keep the notations of the proof of Proposition \ref{prop: rigidity},
and let $Y = \Spec S$ and let $y \in Y$ be the closed point.
Then, we have an identification 
\begin{align*}
  H^2(V, -) &\cong H^0(Y, -) \otimes H^2(V, \OO_V) \\
 &\cong H^0(Y, -) \otimes H^3_y(Y, \OO_Y) \\
 &\cong  H^0(Y, -) \otimes \frac{S[(u_1 u_2 u_3)^{-1}]}{S[(u_2 u_3)^{-1}] + S[(u_1 u_3)^{-1}] + S[(u_1 u_2)^{-1}]}
\end{align*}
of free sheaves on $Y$,
where $u_1, u_2, u_3$ are coordinates of $Y$.
Next, let $D$ be a $p$-closed generator of $T_{V/U}$.
Then, we have a surjection
\begin{equation}\label{eq: generatorofH1}
 \Ker \left( H^2(V, T_{V/U}) \to H^2(V, T_V) \right) \,\to\, H^1(V, f^* T_U),
\end{equation}
where the left-hand side is one-dimensional with basis $D \otimes (u_1 u_2 u_3)^{-1}$ 
under the above identification. 
The surjection in \eqref{eq: generatorofH1} is in fact an isomorphism and its 
injectivity follows from the surjectivity of 
$H^0(V, f^* T_U) \to H^0(V, T_{V/U}^{p^n})$, which we will prove now. 
We may assume the action of $G^0 = \bmu_{p^n}$ is given by $\diag(q_1,q_2,q_3)$, $q_i \in \ZZ$,
and let $D_0 \in H^0(U, T_U)$ be the element characterised by the property that
$D_0(u^e) = ((\sum q_i e_i)/p^n)u^e$.
Then the image of $D_0$ is a generator.

Now, observe that $\GL_3$ acts on the one-dimensional subspace 
$$
  \langle (u_1 u_2 u_3)^{-1} \rangle \,\cong\, H^2(V,\OO_V)[\idealm] \,\subseteq\, H^2(V,\OO_V)
$$
via the inverse of the determinant: $\GL_3 \overset{\det}{\to} \GG_m \overset{{-1}}\to \GG_m$. 
Indeed, it suffices to observe that the $\GL_3$-action preserves this subspace and that 
diagonal matrices act via $\det^{-1}$. 

Therefore, $D \otimes (u_1 u_2 u_3)^{-1}$ is $G$-invariant if and only if the action of 
$G$ on 
$H^0(k,G^\circ) \cong \langle D \rangle \subseteq H^0(Y,T_Y)$
coincides with $\det$.
Identifying $H^0(k,G^\circ)$ with ${\rm Lie}(G)={\rm Lie}(G^\circ)$,
this implies that $D \otimes (u_1 u_2 u_3)^{-1}$ is $G$-invariant if and only if
$\chi_{{\rm ad}}=\chi_{{\rm det}}$.
In particular, $D \otimes (u_1 u_2 u_3)^{-1}$ descends to a cohomology class 
in $H^1(U,T_U)$ if and only if $\chi_{{\rm ad}}=\chi_{{\rm det}}$.

Putting all these observations together, we see that the cohomology group
$H^1(U,T_U) = H^1(V,f^*T_U)^G$ 
is zero (resp. one-dimensional) if and only if 
$\chi_{{\rm ad}}\neq\chi_{{\rm det}}$ (resp. $\chi_{{\rm ad}}=\chi_{{\rm det}}$).
In particular, $X$ is infinitesimally rigid if and only if 
$\chi_{{\rm ad}}\neq\chi_{{\rm det}}$. 
\end{proof}

\begin{Remark}
 We note that the proof of Theorem \ref{thm: non-rigid 3foldclassification} shows more 
 generally that, for $d \geq 3$, the vector space $H^{d-2}(U,T_U)$ is non-trivial 
 if and only if $\chi_{{\rm ad}} = \chi_{\det}$, in which case it is one-dimensional. 
\end{Remark}

\begin{Example}
\label{ex: lrq in dimension 3}
    Assume that $X$ is an lrq singularity in dimension $3$ by a group scheme $G$ with non-trivial $G^\circ$.
    
    \begin{enumerate}
        \item  If $X$ is of type $\frac{1}{n}(1,q_1,q_2)$, then
         $\chi_{{\rm ad}}$ is trivial. 
         The character $\chi_{\det}: \bmu_n \to \GG_m$ 
         is induced by the $(1+q_1+q_2)$-th power map on $\bmu_n$, 
         hence it is trivial if and only if $n \mid (1 + q_1 + q_2)$.
        \item  If $X$ is of type 
        $\frac{1}{m}(1,r,r^2) \rtimes_{\zeta_{3^{f-1}}} \bmu_{3^fN}$, then
        \begin{eqnarray*}
        \chi_{{\rm ad}}|_{\bmu_m}(a) = \chi_{\det} |_{\bmu_m}(a) &=& 1, \\
        \chi_{{\rm ad}}|_{\bmu_N}(a) &=& 1, \\
        \chi_{\det} |_{\bmu_N}(a) &=& a^3, \\
        \chi_{{\rm ad}}|_{\C_{3^f}}(1) &=& \bar{r}, \\
        \chi_{\det}|_{\C_{3^f}}(1) &=& \zeta_{3^{f-1}}.
        \end{eqnarray*} 
       Here, $\bar{r}$ is the image of $r$ under 
       $(\ZZ/m\ZZ)^\times \to (\ZZ/p\ZZ)^\times \cong \bmu_{p-1} \subseteq \GG_m$.
    \end{enumerate}
\end{Example}

Applying Theorem \ref{thm: non-rigid 3foldclassification} to the above example, 
we obtain the following classification of non-rigid $3$-dimensional lrq singularities.

\begin{Corollary} \label{cor: non-rigid 3foldclassification}
Let $x \in X$ be a 3-dimensional lrq singularity.
Then, $X$ is not infinitesimally rigid if and only if one of the following conditions hold.
\begin{enumerate}
  \item $X$ is of type $\frac{1}{n}(1,q_1,q_2)$ with $p \mid n$ and $n \mid (1+q_1+q_2)$.
  \item $X$ is of type $\frac{1}{m}(1,r,r^2) \rtimes_{\zeta_{3}} \bmu_{9}$ with $p \mid m$ 
  and such that $\bar{r} = \zeta_3$. 
  In particular, $p \equiv 1 \pmod{3}$ in this case.
\end{enumerate}
\end{Corollary}
 
Quite surprisingly, this corollary 
shows that 3-dimensional lrq singularities in characteristic $2$ 
are infinitesimally rigid.

\begin{Corollary} \label{cor: non-rigid 3fold p=0,2}
Let $X$ be a 3-dimensional lrq singularity. 
If $p \in \{0,2\}$, then $X$ is infinitesimally rigid.
\end{Corollary}

\begin{proof}
The case $p = 0$ is a special case of Schlessinger's result. 
If $p = 2$ and $X$ is not infintesimally rigid, then
Corollary \ref{cor: non-rigid 3foldclassification} shows that 
$X$ must be of type $\frac{1}{n}(1,q_1,q_2)$ with $2 \mid n$ and 
$n \mid (1 + q_1 + q_2)$,
which is impossible since $2 \nmid q_i$.
\end{proof}

\begin{Remark}
Since $G$ acts on $\omega := du_1 \wedge \hdots \wedge du_d$ via $\chi_{\det}$, 
the lrq singularity $X$ is Gorenstein if and only if $\chi_{\det}$ is trivial.
Hence, we obtain the following table for 3-dimensional lrq singularities
$(\mathbb{A}^3/G)^\wedge$,
where IR stands for infinitesimally rigid, Gor for Gorenstein,
and \checkmark\ and - denote respectively the existence and non-existence of 
such a singularity.

\begin{center}
\begin{tabular}{c|c||cccc}
    $p$ & $G^{\circ}$ & IR,\,Gor & IR,\,not Gor & not IR,\, Gor & not IR,\, not Gor \\
    \hline
    any & $1$ & \checkmark & \checkmark & - & - \\
    $2$ & $\neq 1$ & - & \checkmark & - & - \\
    $3,5 \pmod{6}$ & $\neq 1$ & - & \checkmark & \checkmark & - \\
    $1 \pmod{6}$   & $\neq 1$ & - & \checkmark & \checkmark & \checkmark \\
\end{tabular}
\end{center}
\end{Remark}

\subsection{Deformation spaces}
By Corollary \ref{cor: non-rigid 3foldclassification}, we know that there exist lrq singularities in dimension $3$ that are not infinitesimally rigid.
Being isolated singularities, their deformation functors have a pro-representable hull,
the \emph{miniversal deformation space}, which
is represented by a local complete algebra over the Witt ring,
see \cite[Proposition 3.10]{SchlessingerArtinrings}.
This begs for the question for the deformation spaces in the cases where
infinitesimal rigidity fails.

\begin{Theorem} \label{thm: deformation space}
Let $x \in X = 0 \in ({\mathbb{A}^3/G})^\wedge$ be a $3$-dimensional lrq singularity by a linearly 
reductive group scheme $G$.
Assume $\length{G^\circ} = p^n$.
If $X$ is not infinitesimally rigid,
then the miniversal deformation space $\Def_X$ is given by
$$
  \Def_X \,\cong\, \Def_{{\mathbb{A}^3/G^\circ}} 
   \,\cong\, \Spec W(k)[\varepsilon]/(\varepsilon^2, p^n \varepsilon),
$$
where $W(k)$ denotes the Witt ring.
\end{Theorem}

In particular, the deformation space of $X$ in equicharacteristic $p$ is 
isomorphic to the spectrum of
$k[\varepsilon]/(\varepsilon^2)$, which is of length $2$.
Before giving the proof, we
combine the results of this section to obtain the following analog
of Schlessinger's theorem for lrq singularities.

\begin{Corollary}
  \label{cor: rigidity}
  Let $x \in X = 0 \in ({\mathbb{A}^d/G})^\wedge$ be a $d$-dimensional lrq singularity 
  by a linearly reductive group scheme $G$. If $d \geq 3$, then $X$ is rigid 
  (but not necessarily infinitesimally rigid if $d=3$).
\end{Corollary}

\begin{proof}
If $d\geq4$, then this is Proposition \ref{prop: rigidity}.
If $d=3$, then this is Theorem \ref{thm: non-rigid 3foldclassification}
and Theorem \ref{thm: deformation space}.
\end{proof}

As another corollary to Theorem \ref{thm: deformation space}, we obtain that for an lrq singularity of dimension $d \geq 3$, there is a \emph{canonical lift} of $X$ to characteristic $0$:

\begin{Corollary}
\label{cor: canonicallift}
  Let $x \in X = 0 \in ({\mathbb{A}^d/G})^\wedge$ be a $d$-dimensional lrq singularity 
  by a linearly reductive group scheme $G$ over an algebraically closed field $k$ of characteristic $p > 0$. 
  If $d \geq 3$, then the following hold:
  \begin{enumerate}
      \item We have $(\Def_{X})_{{\rm red}} \cong \Spec W(k)$.
      \item The restriction of a miniversal deformation of $X$ to $(\Def_{X})_{{\rm red}}$ is isomorphic to 
      the quotient $\mathcal{X}$ of $\hat{\Aff}^d_{W(k)}$ by the natural linear action of the canonical deformation $\tilde{G}$ of $G$ (Definition \ref{def: canonical deformation}).
  \end{enumerate}
\end{Corollary}

\begin{proof}
By Proposition \ref{prop: liftingrepresentation}, there is a unique linear $\tilde{G}$-action 
on $\hat{\Aff}^d_{W(k)}$ lifting the linear action of $G$ on $\hat{\Aff}^d$. 
The quotient $\mathcal{X}$ of $\hat{\Aff}^d_{W(k)}$ by this $\tilde{G}$-action is a deformation of 
$X$ over $W(k)$, so $X$ lifts to $W(k)$.

Now Claim (1) follows from Proposition \ref{prop: rigidity} and Theorem \ref{thm: deformation space}. 
Indeed, these results show that the special fibre of $(\Def_{X})_{{\rm red}} $ has length 
$1$ over $k$. 
Since $X$ lifts to characteristic $0$, every irreducible component of $(\Def_{X})_{{\rm red}} $ dominates $W(k)$, 
which shows $(\Def_{X})_{{\rm red}} \cong \Spec W(k)$.

As for Claim (2), 
the deformation $\mathcal{X}$ over $W(k)$ is induced by a map $\Spec~W(k) \to \Def_X$,
which factors through $(\Def_{X})_{{\rm red}} \cong \Spec W(k)$ since $W(k)$ is reduced.
\end{proof}

\begin{Remark}
 The reader should compare this result with the fact that a lift of $G$ over the
 Witt ring $W(k)$ is, in general, \emph{not} unique,  
 see Example \ref{example: explicitnontrivialdeformation}.
 More precisely, any two lifts $\widetilde{G}_i$, $i=1,2$
 of $G$ over $K:={\rm Frac}(W(k))$ become isomorphic over $\overline{K}$
 (Proposition \ref{prop: liftinggroupscheme}).
 An embedding $G\to\GL_{d,k}$ lifts uniquely to an embedding
 $\widetilde{G}_i\to\GL_{d,K}$ 
 (Proposition \ref{prop: liftingrepresentation})
 and thus, we obtain two quotient singularities
 $\Aff_K^d/\widetilde{G}_i$ over $K$, which 
 become isomorphic over $\overline{K}$.
 It is therefore surprising that Corollary \ref{cor: canonicallift}
 ensures that these two quotient singularities are already isomorphic
 over $K$.
 In particular, the uniqueness result of Theorem \ref{thm: lrqmain}
 does \emph{not} hold over fields that are not algebraically closed.
 In fact, Example \ref{example: explicitnontrivialdeformation} shows that
 it does not hold over $K$, which is of characteristic zero and perfect.
\end{Remark}

\begin{proof}[Proof of Theorem \ref{thm: deformation space}]
By Corollary \ref{cor: non-rigid 3fold p=0,2}, we may assume that $p\geq3$.
Next, we note that $H^i(U,T_U) = H^i(V/G^\circ,T_{V/G^\circ})^{G^\et}$ for $i = 1,2$ follows 
from \cite[Section 2, Proposition]{Schlessinger}. 
The natural maps $\Def_X\to\Def_U$ and $\Def_{V/G^\circ}\to\Def_{\mathbb{A}^3/G^\circ}$
are isomorphisms, because the total spaces $X$ and $\mathbb{A}^3/G^\circ$ 
are of depth $\geq3$, see Proposition \ref{prop: lrq}.
In particular, the map in the middle of 
$$
 \Def_X \,=\, \Def_U \,=\, (\Def_{V/G^\circ})^{G^\et} \,\to\, \Def_{V/G^\circ} 
 \,=\, \Def_{\mathbb{A}^3/G^\circ}
$$
is an isomorphism, as it is formally smooth and since it induces 
an isomorphism on tangent spaces. 
Hence, we may assume that $X$ is of type $\frac{1}{p^n}(q_1,q_2,q_3)$ with $p^n \mid (q_1 + q_2 + q_3)$ 
and we have to show that $\Def_X \cong \Spec W(k)[\varepsilon]/(\varepsilon^2, p^n \varepsilon)$. 
Here, we use a symmetric parametrisation $q_1,q_2,q_3$ instead of $1,q_1,q_2$.

Next, note that we already know that $X$ admits a lift $\widetilde{X}$ to $W(k)$. 
Indeed, we can set $\widetilde{G} = \bmu_{p^n,W(k)}$ and consider the quotient of $W(k)[[u_1,u_2,u_3]]$ 
by the diagonal action of $\widetilde{G}$ with weights $(q_1,q_2,q_3)$. 
Describing $X$ and $\widetilde{X}$ as spectra of invariant rings $R$ and $\widetilde{R}$, we obtain
\begin{eqnarray*}
R &=& k[[u^e \mid q \cdot e = 0 \pmod{p^n}]], \text{ and } \\
\widetilde{R} &=& W(k)[[u^e \mid q \cdot e = 0 \pmod{p^n}]],
\end{eqnarray*}
where $e$ and $q$ are multi-indices. 
We will use the open affine cover of $\widetilde{U} := \Spec \widetilde{R} - V(u_1,u_2,u_3)$ 
given by the $\widetilde{U}_i := \Spec \widetilde{R}_i$, 
where $\widetilde{R}_i := \widetilde{R}[(u_i^{p^n})^{-1}]$. 
Since $\Def_X=\Def_U$, it suffices to study deformations of $U$.

First, we construct a deformation $\mathcal{U}_{\mathrm{univ}}$ of $U$ over
$W(k)[\varepsilon]/(\varepsilon^2,p^n\varepsilon)$, which will turn out to be a miniversal deformation of $U$. 
For this, consider indices modulo $3$ and consider the alternating $1$-cochain in $C^1(\tilde{U}_i, T)$ 
defined by
$$
f_{i+1,i+2} := \frac{q_i u_i}{u_1u_2u_3} \frac{\partial}{\partial{u_i}} \in T_{\tilde{R}_{i+1,i+2}/W(k)}.
$$
One easily checks that
$$
f_{12} + f_{23} + f_{31} = p^n \frac{D_0}{u_1u_2u_3} \in T_{\tilde{R}_{123}/W(k)},
$$
where $D_0$ is defined by the property $D_0(u^e) = ((q \cdot e)/p^n)u^e$. 
In particular, we can use the $1 + \varepsilon f_{ij}$ to glue the 
$\widetilde{U}_i \otimes W(k)[\varepsilon]/(\varepsilon^2,p^n \varepsilon)$ and obtain a 
deformation $\mathcal{U}_{\mathrm{univ}}$ of $U$ over $W(k)[\varepsilon](\varepsilon^2,p^n\varepsilon)$. 
By the proof of Theorem \ref{thm: non-rigid 3foldclassification}, the induced $1$-cocycle $(f_{ij} \pmod{p})$ 
generates $H^1(U,T_U)$, so $\mathcal{U}_{\mathrm{univ}}$ induces the universal deformation of 
$U$ over $k[\varepsilon]/(\varepsilon^2)$. 

Now, we show that $\cal{U}_{\mathrm{univ}}$ is indeed a miniversal deformation of $U$. 
Write $S = W(k)[[\varepsilon]] $ and $I =(\varepsilon^2, p^n \varepsilon) \subset W(k)[[\varepsilon]]$. 
We know by \cite[Proposition 3.10]{SchlessingerArtinrings} that there is a miniversal deformation space, 
which is of the form $S/I'$ for some $I'$. 
There is a $W(k)$-homomorphism $S/I' \to S/I$ that induces a $k$-isomorphism 
$k[\varepsilon]/(\varepsilon^2) \to k[\varepsilon]/(\varepsilon^2)$. 
By replacing $\varepsilon$, we may assume that $S/I' \to S/I$ is a $S$-homomorphism and hence, $I' \subset I$. 
To show that it is an isomorphism, we have to show that $I/I' = 0$ or, equivalently by Nakayama's lemma, 
that $I' + \idealm I = I$.
Therefore, it suffices to show the following:

\bigskip
\noindent \textbf{Claim 1:} Suppose $I' \subseteq S$ is an ideal satisfying $\idealm I \subset I' \subset I$ 
and such that there exists a deformation $\cal{U}' \to \Spec S/I'$ of $\cal{U}_{\mathrm{univ}}$. 
Then, $I' = I$. 
\bigskip

\textit{Proof of Claim 1:} 
Let $\mathcal{U}'_{i} \subset \mathcal{U}'$ be the open subschemes corresponding to the underlying 
spaces of $U_{i} \subset U$.
For simplicity, we write $A' = S/I'$.
By the assumption that $\mathcal{U}'$ extends $\mathcal{U}_{\mathrm{univ}}$, we can choose 
isomorphisms $\alpha_i \colon \mathcal{U}_i \to \tilde{U}_i \otimes A'$, such that the cocycle 
defined by $\beta_{ij} := \alpha_i \alpha_j^{-1} \in \Aut(\tilde{U}_{ij} \otimes A')$
satisfies 
$
\beta_{ij} \equiv 1 + \varepsilon f_{ij} \pmod{I}
$.

Thus, we can write $\beta_{ij} = \exp(\varepsilon f_{ij}) + \theta_{ij}$, with 
$\theta_{ij} \in \Der(\tilde{R}_{ij} \otimes A',  I (\tilde{R}_{ij} \otimes A')) \cong \Der(R_{ij},R_{ij}) \otimes I/I'$.
A straightforward calculation shows that
\begin{equation} \label{eq: obstruction}
\id \,=\, \beta_{12} \beta_{23} \beta_{31} \,=\, \id + \varepsilon  p^n \frac{D_0}{u_1 u_2 u_3}+ \varepsilon^2 \frac{D_1}{(u_1u_2u_3)^2} + \theta_{12} + \theta_{23} + \theta_{31},
\end{equation}
where 
$$
D_1 \,=\, q_1q_2 \left(u_2\frac{\partial}{\partial u_2} - u_1\frac{\partial}{\partial u_1}\right).
$$
Hence, the obstruction to extending $\mathcal{U}_{\mathrm{univ}}$ to $A'$ is 
$$
  \varepsilon p^n \frac{D_0}{u_1 u_2 u_3} + \varepsilon^2 \frac{D_1}{(u_1u_2u_3)^2} \,\in\, H^2(U,T_U) \otimes I/I'.
$$
Thus, to show that $I' = I$, it suffices to show that $\frac{D_0}{u_1 u_2 u_3}$ and $\frac{D_1}{(u_1u_2u_3)^2}$ 
are $k$-linearly independent, because then Equation \eqref{eq: obstruction} shows that
$p^n\varepsilon = \varepsilon^2 = 0 \in I/I'$, hence $I = I'$, which is precisely Claim 1. 
Therefore, let us finish the proof by showing the following:

\bigskip
\noindent \textbf{Claim 2:} $\frac{D_0}{u_1 u_2 u_3}$ and $\frac{D_1}{(u_1u_2u_3)^2}$ are $k$-linearly independent in $H^2(U,T_U)$. 
\bigskip

\textit{Proof of Claim 2:} 
To see this, recall that the image of $D_0$ under the map $H^0(U,T_U) \to H^0(V,T_{V/U}^{p^n})^G$ is a generator, 
so that $\frac{D_0}{u_1 u_2 u_3}$ is non-trivial in $H^2(V,T_{V/U}^{p^n})^G$  
under the identification given in the proof of Theorem \ref{thm: non-rigid 3foldclassification}. 
Next, $D_1$ is a $G$-invariant derivation on $k[[u_1,u_2,u_3]]$ that is not contained in 
${\rm Lie}(G) \subseteq H^0(Y,T_Y)$. 
Hence, $\frac{D_1}{(u_1u_2u_3)^2}$ is a non-trivial element in 
${\rm Coker}(H^2(V,T_{V/U})^G \to H^2(V,T_V)^G)$. 
Therefore, the exact sequence
\[
H^2(V,T_{V/U})^G \,\to\, H^2(V,T_V)^G \,\to\, H^2(U,T_U) \,\to\, H^2(V,T_{V/U}^{p^n})^G \,\to\, 0
\]
from the proof of Proposition \ref{prop: rigidity} shows Claim 2 and thus, finishes the proof. 
\end{proof}

\subsection{Concerning arithmetic rigidity}
\label{subsec: arithmetic rigidity}
For an lrq singularity in positive characteristic,
there exists a canonical lift over the Witt ring.
In particular, lrq singularities
are \emph{not} arithmetically rigid, that is, rigid in mixed characteristic.
On the other hand, the canonical
lift together with Theorem \ref{thm: deformation space} 
and Corollary \ref{cor: canonicallift}
could be viewed as a sort of (infinitesimal) 
rigidity result in mixed characteristic.

\section{F-regular surfaces are lrq surfaces} 
\label{sec: Fregular}

Over the complex numbers and in dimension two, it is classical
that the klt (resp. canonical) singularities are precisely
the quotient singularities by finite subgroups
of $\GL_2$ (resp. $\SL_2$).
In this section, we extend this result to positive characteristic
and show that F-regular (resp. F-regular and Gorenstein) surface singularities
are precisely the quotient singularities by finite and linearly
reductive subgroup schemes of $\GL_2$ (resp. $\SL_2)$.

\subsection{F-regular Gorenstein surface singularities}
We recall that the F-regular Gorenstein surface singularities are precisely 
the F-regular rational double points. 
All of them have been independently realised as lrq singularities 
by Hashimoto and Liedtke--Satriano.

\begin{Theorem}
 \label{thm: LS}
 The following rational double points over an algebraically closed field $k$ of 
 characteristic $p\geq0$
  $$\begin{array}{c|cccl}
           & \mbox{group scheme}  & \mbox{length} & \mbox{\'etale} & \mbox{characteristic}\\
       \hline
      A_n & \bmu_{n+1}  & n+1 &\mbox{if }p\nmid(n+1) &  \mbox{all $p$} \\
      D_n & \BD_{n-2}    & 4(n-2) & \mbox{if }p\nmid(n-2) & p\neq 2 \\
      E_6 & \BT_{24}      & 24   & \mbox{yes} & p\neq 2,3\\
      E_7 & \BO_{48}     & 48   & \mbox{yes} & p\neq 2,3 \\
      E_8 & \BI_{120}     & 120 & \mbox{yes} & p\neq 2,3,5
  \end{array}$$
  are quotient singularities by the indicated linearly reductive subgroup schemes 
  of $\SL_2$ in the indicated characteristics.
\end{Theorem}

\begin{proof}
See \cite[Corollary 3.10]{Hashimoto} or 
\cite[Proposition 4.2]{LiedtkeSatriano}.
\end{proof}

Combining Theorem \ref{thm: LS} with the classification lists of Artin \cite{ArtinRDP} and 
Hara \cite{Hara} and the classification of finite, very small, and linearly reductive subgroup 
schemes of $\SL_2$ (see Theorem \ref{thm: hashimoto}), we have the following characteristic $p$ 
analog of Klein's theorem \cite{Klein}.

\begin{Theorem}\label{thm: FregularRDP}
  For a Gorenstein surface singularity $x\in X$ over an algebraically closed field $k$ of 
  characteristic $p > 0$, the following are equivalent
  \begin{enumerate}
   \item $x\in X$ is an lrq singularity,
   \item $x\in X$ is an lrq singularity by a subgroup scheme of $\SL_{2,k}$,
   \item $x\in X$ is F-regular.
  \end{enumerate}
  Moreover, if $p\geq7$, then these are equivalent to
  \begin{enumerate}
   \setcounter{enumi}{3}
   \item $x\in X$ is a rational double point.
  \end{enumerate}
\end{Theorem}

\begin{proof}
First, note that the equivalence $(3) \Leftrightarrow (4)$ in characteristic $p \geq 7$ 
follows from Theorem \ref{thm: Fregularclassification}, so that we can focus 
on the equivalence of the first three properties.

Clearly, we have $(2) \Rightarrow (1)$. 
To prove $(1) \Rightarrow (2)$, note first that the singularity is an lrq singularity 
by a subgroup scheme of $\GL_2$ by Proposition \ref{prop: linearizable}. 
Now, it suffices to note that a linearly reductive subgroup scheme of $\GL_2$ that 
is not contained in $\SL_2$ acts non-trivially on a regular $2$-form, 
so the quotient will not be Gorenstein.

The implication $(2) \Rightarrow (3)$ is contained in Proposition \ref{prop: lrq}. 
Conversely, for $(3) \Rightarrow (2)$ we use that F-regular surface singularities 
are rational by Theorem \ref{thm: Fregularclassification}. 
Since $X$ was assumed to be Gorenstein, this implies that $X$ is a rational double point. 
Thus, the claim follows by noting that the rational double points in 
Theorem \ref{thm: LS} are precisely the F-regular rational double points.
\end{proof}

\begin{Remark}
If $p=0$, then the equivalence $(1)\Leftrightarrow(4)$ is classical, 
see, for example \cite{Brieskorn} or \cite{Durfee}.
\end{Remark}

\begin{Remark}
It would be interesting to have a proof of these equivalences without using 
classification lists.
\end{Remark}

\subsection{F-regular surface singularities}

Next, we study also non-Gorenstein surface singularities and consider 
quotients by finite and linearly reductive subgroup schemes of $\GL_2$.

\begin{Theorem}\label{thm: Fregularklt}
  For a surface singularity $x\in X$ over an algebraically closed field $k$ of characteristic $p > 0$, 
  the following are equivalent
  \begin{enumerate}
   \item $x\in X$ is an lrq singularity,
   \item $x\in X$ is an lrq singularity by a subgroup scheme of $\GL_{2,k}$,
   \item $x\in X$ is F-regular.
  \end{enumerate}
  Moreover, if $p\geq7$, then these are equivalent to
  \begin{enumerate}
   \setcounter{enumi}{3}
   \item $x\in X$ is a normal klt singularity,
  \end{enumerate}
\end{Theorem}

\begin{proof}
First, note that the equivalence $(3) \Leftrightarrow (4)$ in characteristic $p\geq7$ 
follows from Theorem \ref{thm: Fregularclassification}, so that we can focus on the 
equivalence of the first three properties.

The implication $(2) \Rightarrow (1)$ is clear and the converse follows 
from Proposition \ref{prop: linearizable}.

Note that $(2) \Rightarrow (3)$ follows from Proposition \ref{prop: lrq}, so that
we can focus on the implication $(3) \Rightarrow (2)$. 
We thus have to show that every F-regular surface singularity is an lrq singularity. 
It suffices to show that the lrq singularities by the group schemes that occur 
in Theorem \ref{thm: smallingl2} exhaust all F-regular surface singularities. 
Since F-regular surface singularities are taut by Theorem \ref{thm: Fregulartaut}, 
it suffices to show that all dual graphs allowed by Theorem \ref{thm: Fregularclassification} 
occur as dual graphs of the resolution of $X = (\mathbb{A}^2/G)^\wedge$ 
for some $G$ in the list of Theorem \ref{thm: smallingl2}.

The resolution graph $\Gamma$ of $X$ is the same as the resolution of the 
corresponding quotient singularity by $G_{\mathrm{abs}}$ over the complex numbers.
This is well-known if $G = \bmu_{n,q}$ (and thus, $X$ is toric). 
In the general case, one studies the induced action of $G$ on the blow-up $Y$ of 
the origin in $\mathbb{A}^2$. 
Let $E \cong \PP^1$ be the exceptional divisor of this blow-up and let $K$ (resp. $H$) 
be the kernel (resp. image) of the induced homomorphism $G \to \Aut_E$. 
By the Chevalley--Shephard--Todd theorem \cite{Satriano}, $Y' := Y/K$ 
is smooth and the image $E'$ of $E$ in $Y'$ is still isomorphic to $\PP^1$. 
The action of $H$ on $Y'$ preserves $E'$ and acts faithfully on it. 
Since $H$ is linearly reductive and $E'$ is $H$-stable, the action of every stabiliser 
can be linearised on $E'$, hence every such stabiliser is isomorphic to some $\bmu_n$. 
This reduces the calculation of $\Gamma$ to the toric case. 
We refer the reader to \cite[Satz 2.11]{Brieskorn} for a list and the description of $\Gamma$,
which carries over without change to positive characteristic. 
Comparing this list with the list of resolution graphs of F-regular surface 
singularities given in Theorem \ref{thm: Fregularclassification}, 
one easily checks that all dual graphs are realised. 
As explained in the previous paragraph, this finishes the proof.
\end{proof}

\begin{Remark}
If $p=0$, then the equivalence $(1)\Leftrightarrow(4)$ is classical, see 
\cite[Theorem 9.6]{Kawamata}.
\end{Remark}

\begin{Remark}
In \cite[Satz 2.10]{Brieskorn}, Brieskorn showed that 2-dimensional quotient singularities 
over the complex numbers are taut.
Since $2$-dimensional lrq singularities in characteristic $p>0$
are F-regular by Proposition \ref{prop: lrq}, it follows from Tanaka's result 
in Theorem \ref{thm: Fregulartaut} that they are taut, which generalises 
Brieskorn's theorem to positive characteristic.
\end{Remark}

\begin{Remark}
In characteristic $2$, Theorem \ref{thm: Fregularklt} together with the list in 
Theorem \ref{thm: smallingl2} and Proposition \ref{prop: toric} as well as the fact that normal toric surface singularities are rational \cite[Theorem 11.4.2]{Cox}, hence $\mathbb{Q}$-factorial, shows that 
F-regular surface singularities are the same as toric surface singularities. 
This emphasises the fact that there is a huge difference between the 
notions of klt and F-regular singularities in small characteristics.
\end{Remark}

\begin{Remark}
The proof of Theorem \ref{thm: Fregularklt} shows that the dual resolution graph of 
an lrq surface singularity $X$ uniquely determines the very small subgroup 
scheme $G \subseteq \GL_{2,k}$ by which $X$ is a quotient singularity. 
We note that this observation, together with the rigidity of lrq singularities 
in dimension $d \geq 3$ that we proved in Proposition \ref{prop: rigidity} 
and Theorem \ref{thm: deformation space} can be used to give another proof of 
Theorem \ref{thm: lrqmain} (3) by reducing the problem to characteristic $0$, 
where the classical topological proof applies.
\end{Remark}

\subsection{Wall's conjecture for lrq singularities}
Let $x\in X$ be a two-dimensional singularity over the complex numbers 
that is isomorphic to the quotient $(\mathbb{A}^d/G)^\wedge$, 
where $G$ is a reductive algebraic group 
acting linearly on the affine space $\mathbb{A}^d$.
Then, C.T.C. Wall \cite{Wall} conjectured that
$x\in X$ is a quotient singularity in the sense of this article,
that is, isomorphic to $(\mathbb{A}^2/\Gamma)^\wedge$, where
$\Gamma$ is a finite group acting linearly on $\mathbb{A}^2$.
This conjecture was proven by Gurjar \cite{Gurjar}
and Kumar \cite{Kumar}.
Using Theorem \ref{thm: Fregularklt}, we obtain the following
analog for lrq singularities.

\begin{Theorem}
\label{thm: Wall}
Let $x\in X$ be a two-dimensional singularity over an algebraically
closed field $k$ of characteristic $p>0$.
Assume that $x\in X$ is isomorphic to $(\mathbb{A}^d/G)^\wedge$,
where $G$ is an affine group scheme acting linearly on $\mathbb{A}^d$
and where $G$ is an extension
$$
  1 \,\to\, G_1 \,\to\, G \,\to\, G_2 \,\to\, 1
$$
where $G_2$ is a linearly reductive group scheme 
(not necessarily finite) and where 
$G_2$ is connected reductive algebraic, such that
the action of $G_2$ on $\mathbb{A}^d$ is good in the sense of
\cite{HashimotoGoodfiltrations}.
Then, $x\in X$ is an lrq singularity.
\end{Theorem}

\begin{proof}
By \cite[Corollary 2]{HashimotoGoodfiltrations}, $x \in X$ is $F$-regular. Thus, by Theorem \ref{thm: Fregularklt}, $x \in X$ is an lrq singularity.
\end{proof}

\section{Riemenschneider's conjecture for lrq singularities}
\label{sec: Riemenschneider}
Over the complex numbers, Riemenschneider \cite{Riemenschneider} conjectured that every 
deformation of an isolated quotient singularity is again an isolated quotient singularity. 
While this follows from Schlessinger's rigidity result \cite{Schlessinger} if the dimension is at least $3$, 
the dimension $2$ case is more subtle and was later settled by Esnault and Viehweg in \cite{EsnaultViehweg2}. 
Moreover, Koll\'ar and Shepherd-Barron \cite{KSB} showed that cyclic quotient 
singularities deform to cyclic quotient singularities.

\subsection{The conjecture}
First, we propose a version of Riemenschneider's conjecture for lrq singularities
in positive and mixed characteristic.
We also propose a conjectural version of the result of Koll\'ar and Shepherd-Barron
for cyclic lrq singularities in positive and mixed characteristic.
Here, we call an lrq singularity \emph{cyclic} if it satisfies any of the
equivalent conditions of Proposition \ref{prop: toric}. Recall that by Theorem \ref{thm: lrqmain}, an lrq singularity $x \in X$ uniquely determines the linearly reductive group scheme $G$ by which $X$ is a quotient singularity. Thus, it makes sense to say that $G$ is \emph{associated to} $x \in X$.

\begin{Conjecture}
 \label{conj: Riemenschneider}
  Let $B$ be the spectrum of a DVR with closed, generic, and geometric generic points
  $0$, $\eta$, and $\overline{\eta}$, respectively.
  Let $\mathcal{X}\to B$ be a flat morphism with special and geometric 
  generic fibre $\mathcal{X}_0$ and $\mathcal{X}_{\overline{\eta}}$, respectively. Assume that the non-smooth locus of $\mathcal{X} \to B$ is proper over $B$.
  \begin{enumerate}
    \item \label{item: conj: lrq} If $\mathcal{X}_0$ is an lrq singularity, then $\mathcal{X}_{\overline{\eta}}$ contains at worst 
       lrq singularities.
     \item \label{item: conj: length} Let $G_0$ be the group scheme associated to $\mathcal{X}_0$ and let $G_\eta$ be the group scheme associated to an lrq singularity on $\mathcal{X}_\eta$. Then, we have $|G_0| \geq |G_\eta|$, where $|-|$ denotes the length of a group scheme.
    \item \label{item: conj: cyclic} If $\mathcal{X}_0$ is a cyclic lrq singularity, then $\mathcal{X}_{\overline{\eta}}$
       contains at worst cyclic lrq singularities.
  \end{enumerate}
  
\end{Conjecture}

Here, \emph{containing at worst (cyclic) lrq singularities} means that the non-smooth locus of
$\mathcal{X}_{\overline{\eta}}$ if non-empty is a finite
set of (cyclic) lrq 
singularities.

\begin{Remark}
 In view of Proposition \ref{prop: toric}, Part (\ref{item: conj: cyclic}) 
 of Conjecture \ref{conj: Riemenschneider} is equivalent to the following statement: 
 if $\mathcal{X}_0$ is an isolated toric $\mathbb{Q}$-factorial singularity, then $\mathcal{X}_{\overline{\eta}}$ 
 contains at worst isolated toric $\mathbb{Q}$-factorial singularities.
\end{Remark}

\begin{Remark} \label{rem: char0}
 We note that Conjecture \ref{conj: Riemenschneider} \eqref{item: conj: lrq} 
 and \eqref{item: conj: cyclic} are true if the residue characteristic of $B$ 
 is zero as indicated above:
 Part (\ref{item: conj: lrq}) is \cite[Theorem 2.5]{EsnaultViehweg2} 
 and Part (\ref{item: conj: cyclic}) is \cite[Corollary 7.15]{KSB}.
 We do not know whether Conjecture \ref{conj: Riemenschneider} \eqref{item: conj: length} 
 is known in characteristic 0 in the stated generality, but we will see in 
 Proposition \ref{prop: riemenschneider length} that it holds if $\mathcal{X}_0$ 
 is Gorenstein or if it is a cyclic quotient singularities in characteristic 0. 
\end{Remark}

\begin{Proposition}
  \label{prop: conjecture dimension 3}
  Conjecture \ref{conj: Riemenschneider} is true in dimension $d\geq3$.
\end{Proposition}

\begin{proof}
Let $p_0$ (resp. $p_\eta$) be the residue characteristic of the special
(resp. generic) point of $B$.
First, we treat the equicharacteristic case:
if $p_0=p_\eta=0$, then the assertion follows from Schlessinger's rigidity
theorem \cite{Schlessinger}.
Moreover, if $p_0=p_\eta>0$, then it follows 
from Corollary \ref{cor: rigidity}.
Second, we treat the mixed characteristic case, that is,
$p_0>0$ and $p_\eta=0$. Here, the assertion follows from Corollary \ref{cor: canonicallift}. 
\end{proof}

Thus, we may now focus on dimension two in Conjecture \ref{conj: Riemenschneider}.
In this case, we can rephrase Parts \eqref{item: conj: lrq}
and \eqref{item: conj: length} of
Conjecture \ref{conj: Riemenschneider}  in more well-known terms.

\begin{Lemma} \label{lem: modifiedRiemenschneider}
In the setting of Conjecture \ref{conj: Riemenschneider}, assume that $dim(\mathcal{X}_0) = 2$ and that the residue characteristic of $B$ is $p > 0$. 
Then, 
\begin{itemize}
    \item Part (\ref{item: conj: lrq}) of Conjecture \ref{conj: Riemenschneider}  is equivalent to the following:
\begin{enumerate}
\item[(\ref{item: conj: lrq})'] Assume ${\rm char}(B) = p$. If $\mathcal{X}_0$ is F-regular, then $\mathcal{X}_{\overline{\eta}}$ contains at worst F-regular singularities.
\item[(\ref{item: conj: lrq})''] Assume ${\rm char}(B) = 0$. If $\mathcal{X}_0$ is F-regular, then $\mathcal{X}_{\overline{\eta}}$ contains at worst klt singularities.
\end{enumerate}
\item Part (\ref{item: conj: length}) of Conjecture \ref{conj: Riemenschneider}
is equivalent to the following:
\begin{enumerate}
\item[(\ref{item: conj: length})'] Assume ${\rm char}(B) = p$. Then, $s_0 \leq s_\eta$.
\item[(\ref{item: conj: length})''] Assume ${\rm char}(B) = 0$. Then, $s_0 \leq |\pietloc_{,\eta}|^{-1}$.
\end{enumerate}
Here,  $s_0$ and $s_\eta$ denote the F-signatures of $\mathcal{X}_0$ and an lrq singularity on $\mathcal{X}_{\bar\eta}$, respectively, 
and $\pietloc_{,\eta}$ denotes the local fundamental group of an lrq singularity on $\mathcal{X}_{\bar\eta}$.
\end{itemize}
\end{Lemma}

\begin{proof}
Equivalence of (\ref{item: conj: lrq}) with (\ref{item: conj: lrq})' and (\ref{item: conj: lrq})'' follows immediately from Theorem \ref{thm: Fregularklt}.

For the second part, note that by Proposition \ref{prop: lrq} the inverse F-signature $s_0^{-1}$ (resp. $s_\eta^{-1}$) coincides with the length of the group scheme by which $X_0$ (resp. $X_{\bar\eta}$) is a quotient singularity. This shows that \eqref{item: conj: length} is equivalent to \eqref{item: conj: length}'. In characteristic $0$, the group scheme by which $X_{\bar\eta}$ is a quotient singularity is $\pietloc_{,\eta}$, so we have \eqref{item: conj: length} $\iff$ \eqref{item: conj: length}''.
\end{proof}

\subsection{Examples}
First, we note that Conjecture \ref{conj: Riemenschneider} 
fails in every positive characteristic if we do not 
allow quotient singularities by group schemes, as the following 
two-dimensional examples shows:
the special fibre is a cyclic lrq singularity by an \'etale group
scheme and the generic fibre is a cyclic lrq singularity 
by an infinitesimal group scheme.

\begin{Example}
  \label{ex: deformation RDP1}
  Consider the family of rational double points
  $$
     z^{p+1} \,+\, tz^p \,+\, xy
   $$
  over $k[[t]]$. 
  Its special fibre $\mathcal{X}_0$ is a singularity of type $A_p$, 
  while its generic fibre $\mathcal{X}_\eta$ is of type $A_{p-1}$. 
  Thus, $\mathcal{X}_0$ (resp. $\mathcal{X}_\eta$) is a cyclic quotient singularity by 
  $\bmu_{p+1}$ (resp. $\bmu_p$).
  The local fundamental group of $\mathcal{X}_0$ (resp. $\mathcal{X}_\eta$)
  is $\C_{p+1}$ (resp. trivial).
  Thus, $\mathcal{X}_0$ is a quotient singularity by a finite group,
  while $\mathcal{X}_\eta$, being simply connected and non-smooth, 
  is not.

  One can also give an example where $\mathcal{X}_0$ is a quotient singularity 
  by a connected group scheme and $\mathcal{X}_\eta$ is a quotient singularity by a finite group.
\end{Example}

The following example shows that Conjecture \ref{conj: Riemenschneider} 
also fails for deformations of quotient singularities by 
finite group schemes that are not linearly reductive.

\begin{Example}
  In characteristic $p=2$,
  consider the family of rational double points
  $$
    z^2 \,+\, x^3 + y^5 \,+\, txy^3z
  $$
  over $k[[t]]$. 
  Its special fibre $\mathcal{X}_0$ is a singularity of type $E_8^0$, 
  while its generic fibre $\mathcal{X}_\eta$ is of type $E_8^1$. 
  We will see in \cite{RDP}
  that $\mathcal{X}_0$ is a quotient singularity by $\balpha_2$,
  while $\mathcal{X}_\eta$ is not a quotient singularity.
  Similarly, there exists a deformation of $E_8^2$, which is a quotient singularity by 
  $\C_2$, to $E_8^3$, which is not a quotient singularity.
\end{Example}

\subsection{Evidence}
In the final section of this article, we collect further evidence for Conjecture \ref{conj: Riemenschneider}. 
In particular, we will prove that all parts of Conjecture \ref{conj: Riemenschneider} hold if $\mathcal{X}_0$ is Gorenstein.
We start with Part (\ref{item: conj: lrq}).

\begin{Proposition}
 \label{prop: riemenschneider}
 Part (\ref{item: conj: lrq}) of Conjecture \ref{conj: Riemenschneider} is true if
 the total space $\mathcal{X}$ is $\QQ$-Gorenstein.
\end{Proposition}

\begin{proof}
Since the conjecture is known in dimension $d\geq3$ by 
Proposition \ref{prop: conjecture dimension 3}, we may assume $d=2$. 
Furthermore, by Remark \ref{rem: char0}, we may assume that the residue characteristic $p_0$ 
of the special point of $B$ is positive. 
In particular, by Proposition \ref{prop: lrq}, the special fibre $\mathcal{X}_0$ is F-regular.
Let $p_\eta$ be the residue characteristic of the generic point of $B$.

First, assume that $p_\eta = p$. 
Then, the claim follows from Theorem \ref{thm: Fregularklt} and the equivalence of (\ref{item: conj: lrq}) with (\ref{item: conj: lrq})' 
in Lemma \ref{lem: modifiedRiemenschneider}, since F-regularity is open in 
$\QQ$-Gorenstein rings \cite{AKM}.

Next, assume that $p_\eta = 0$. 
We will use the notion of \emph{(pure) BCM-regularity} of \cite{MaSchwede} and \cite{MaSchwede+}. 
Since $\mathcal{X}_0$ is F-regular, it is BCM-regular by \cite[Corollary 6.8]{MaSchwede+}. 
Since $\mathcal{X}$ is $\QQ$-Gorenstein, it follows from inversion of 
adjunction \cite[Corollary 3.3]{MaSchwede+} that the pair $(\mathcal{X},\mathcal{X}_0)$ 
is purely BCM-regular and so, $\mathcal{X}$ is BCM-regular. 
By \cite[Corollary 6.22]{MaSchwede}, this implies that $\mathcal{X}$ is klt
(note that resolution of singularities for the threefold $\mathcal{X}$ exists
by \cite{CoPi}).
Since the definition of being klt is local (see, for example \cite[Section 2.16]{Kollar}), 
this implies that $\mathcal{X}_{\bar\eta}$ is klt. 
By Lemma \ref{lem: modifiedRiemenschneider} (\ref{item: conj: lrq})'', this proves the assertion.
\end{proof}

In particular, Part (\ref{item: conj: lrq}) of Conjecture \ref{conj: Riemenschneider} 
holds if $\mathcal{X}_0$ is Gorenstein, because then $\mathcal{X}$ will be Gorenstein
as well and the proposition applies.

\begin{Remark}
\label{rem: SatoTakagi}
After finishing the first version of this article, we were informed by Sato and 
Takagi that they are able to prove Proposition \ref{prop: riemenschneider} 
without the $\QQ$-Gorenstein assumption, thereby proving Conjecture 
\ref{conj: Riemenschneider} \eqref{item: conj: lrq} in full generality. 
Their proof can be found in \cite[Corollary 4.8]{SatoTakagi}.
\end{Remark}

 In the proof of Proposition \ref{prop: riemenschneider}, we used that the F-regular locus is open in $\QQ$-Gorenstein families. This openness is false in general as shown by examples of Singh \cite{SinghFregularity}.
Let us recall Singh's example using our terminology and check that it does
\emph{not} give a counter-example to 
Conjecture \ref{conj: Riemenschneider} \eqref{item: conj: length}.

\begin{Example}
 Consider the deformation $\mathcal{X}\to\Spec \overline{\FF}_p[[t]]$ from 
 \cite[Theorem 1.1]{SinghFregularity}.

 \begin{enumerate}
\item The generic fibre $\mathcal{X}_\eta$ has a rational double point
 of type $A_1$, which is an lrq singularity with respect to the group scheme $\bmu_2$.
 \item The special fibre $\mathcal{X}_0$ has a singularity of 
 $D_{4n+1,2n}$ (in Riemenschneider's notation \cite{Riemenschneider2}),
 which is of index $2n+1$ and which is 
 an lrq singularity with respect to  
 $(\bmu_{2(2n+1),1},\bmu_{2(2n+1),1};\BD_n,\BD_n)$, that is,
 case (2a) of Theorem \ref{thm: smallingl2}.
 \end{enumerate}
 The total space $\mathcal{X}$ is not F-regular and 
 not $\QQ$-Gorenstein.
 On the other hand, this example does satisfy
 Conjecture \ref{conj: Riemenschneider}.
\end{Example}

Next, we give several special cases for which 
Part \eqref{item: conj: length} holds.

\begin{Proposition}
\label{prop: riemenschneider length}
 Part (\ref{item: conj: length}) of Conjecture \ref{conj: Riemenschneider} is true if
 \begin{enumerate}
     \item $B$ has equal characteristic $0$ and $\mathcal{X}_0$ is cyclic, or
     \item $B$ has equal characteristic $p$ and $\mathcal{X}$ is $\QQ$-Gorenstein, or
     \item $\mathcal{X}_0$ is Gorenstein.
 \end{enumerate}
\end{Proposition}

\begin{proof}
By Proposition \ref{prop: conjecture dimension 3}, we may assume that $d = 2$.

Let us first prove (1). 
Since $\mathcal{X}_0$ is cyclic, it is a quotient by $\bmu_n$ acting as
$\bmu_{n,q}$ for some $n,q \geq 1$ with $(n,q) = 1$ (see Theorem \ref{thm: smallingl2}). 
Let $[a_1;a_2,\hdots,a_k]$ be the Hirzebruch--Jung continued fraction for $n/q$ and let $A_i$ 
be the $i$-th continuant of $[a_1;a_2,\hdots,a_k]$, that is, 
the numerator of the continued fraction $[a_1;a_2,\hdots,a_i]$ and set $A_{0} = 1$
and $A_{-1} = 0$. 
Then, the $A_i$ satisfy the recursion $A_i = a_iA_{i-1} - A_{i-2}$. 

By \cite[Proposition 3.6]{Brohme}, the geometric generic fibre $\mathcal{X}_{\bar\eta}$ is a cyclic 
quotient singularity by some $\bmu_{n'}$ acting as $\bmu_{n',q'}$ such that $n'/q'$ admits a 
representation as a continued fraction $[a_1';a_2',\hdots,a_k']$ such that $a_i' \leq a_i$ for all $i$. 
Now, a straightforward proof via induction shows that this implies $A_i' \leq A_i$ and 
$A_i' - A_{i-1}' \leq A_i - A_{i-1}$ for all $i$, where the $A_i'$ are the continuants 
of $[a_1';a_2',\hdots,a_k']$. 
In particular, we have $n' \leq n$, so Conjecture \ref{conj: Riemenschneider} 
(\ref{item: conj: length}) holds in this case.

Next, let us prove (2). 
Since $\mathcal{X}$ is $\QQ$-Gorenstein, it follows from \cite[Corollary 1.2]{Taylor} that the F-signature of $\mathcal{X}$ is at least as big as the F-signature of $\mathcal{X}_0$. 
Since the F-signature is lower semicontinuous on $\mathcal{X}$, 
the same is true for $\mathcal{X}_{\bar\eta}$. 
Hence, the claim follows from the equivalence of (2) with (2)' in 
Lemma \ref{lem: modifiedRiemenschneider}.

Finally, let us prove (3). 
By assumption, $\mathcal{X}_0$ is a rational double point. 
It is well-known that $\mathcal{X}_{\bar\eta}$ is also a rational double point in this case. 
After possibly replacing $B$ by a finite extension, we may assume that $\mathcal{X}\to B$ 
admits a minimal simultaneous resolution of singularities $\mathcal{Y}\to\mathcal{X}\to B$.
Let $\Gamma_\eta$ (resp. $\Gamma_0$) be the dual resolution graph associated 
to a singularity on $\mathcal{Y}_\eta$ (resp. to $\mathcal{Y}_0$).
Let $\Lambda(\Gamma_\eta)$ and $\Lambda(\Gamma_0)$ be the lattices
associated to $\Gamma_\eta$ and $\Gamma_0$, respectively.
Moreover, let $G_\eta$ and $G_0$ be the finite and 
linearly reductive group schemes associated to lrq rational singularities
of type $\Gamma_\eta$ and $\Gamma_0$ (see Theorem \ref{thm: LS} for the 
group schemes and their lengths).
Specialisation induces an injection of lattices
$\Lambda(\Gamma_\eta) \subseteq \Pic(\mathcal{Y}_\eta)\to \Pic(\mathcal{Y}_0) = \Lambda(\Gamma_0)$.
Using \cite[Exercise 4.6.2, Theorem 4.6.12.]{Martinet}, which tells us which 
$\Lambda(\Gamma_\eta)$'s can occur for a given $\Lambda(\Gamma_0)$,
it is straightforward to check that $|G_\eta|\leq|G_0|$.
\end{proof}

We note that the assumption on the index of $\mathcal{X}_0$ in the proposition
is satisfied if it is an lrq singularity with respect to an \'etale group scheme.

\begin{Remark} 
 In dimension $d=2$, Conjecture \ref{conj: Riemenschneider} \eqref{item: conj: length} is related to the following
 more general fact:
 let  $\mathcal{X}\to B$ is a family of rational double point singularities.
 Let $\Gamma_\eta$ and $\Gamma_0$ be the dual resolution graphs (disjoint union of Dynkin diagrams)
 associated to the minimal resolutions of singularities of 
 $\mathcal{X}_\eta$ and $\mathcal{X}_0$, respectively.
 If the residue characteristic of $B$ is zero or positive and large enough 
 depending on $\Gamma_0$ (more precisely, if it is ``good'' in the sense of Slodowy), 
 then $\Gamma_\eta$ is a subgraph of $\Gamma_0$, see \cite[Section 8.10]{Slodowy}. 
 From this, we get another proof of Proposition \ref{prop: riemenschneider length} for good
 residue characteristics.
\end{Remark}

Finally, we note that Part \eqref{item: conj: cyclic} holds in the Gorenstein case.

\begin{Proposition}
 \label{prop: riemenschneider cyclic}
 Part (\ref{item: conj: cyclic}) of Conjecture \ref{conj: Riemenschneider} is true if 
 $\mathcal{X}_0$ is Gorenstein.
\end{Proposition}

\begin{proof}
In the proof of Proposition \ref{prop: riemenschneider length}, 
we have seen that every singularity on $\mathcal{X}_{\bar\eta}$ 
is a rational double point, whose associated lattice $\Lambda(\Gamma_\eta')$ 
embeds into the lattice $\Lambda(\Gamma_0)$ associated to $\mathcal{X}_0$. 
If $\mathcal{X}_0$ is cyclic, then $\Lambda(\Gamma_0)$ is of type $A_n$. 
Since every root sublattice of $A_n$ is of type $A_k$ for some $k \leq n$ 
(see, for example, \cite[Exercise 4.6.2]{Martinet}), the claim follows.
\end{proof}

Combining the previous propositions, we conclude the following.

\begin{Corollary}
 \label{cor: riemenschneider gorenstein}
 Conjecture \ref{conj: Riemenschneider} is true if $\mathcal{X}_0$ is Gorenstein.
\end{Corollary}

\bibliographystyle{alpha}
\bibliography{Bibliography}

\end{document}